\documentclass[a4paper,10pt]%[draft]
{amsart}
\usepackage[foot]{amsaddr}
\newtheorem{theorem}{Theorem}[section]

\theoremstyle{definition}
\newtheorem{definition}[theorem]{Definition}

\theoremstyle{remark}

\usepackage[margin=2.7cm]{geometry}
\usepackage{systeme}
\usepackage{graphicx}
\usepackage[]{units}
\usepackage{color}
\usepackage{mathrsfs}
\usepackage{bm}
\usepackage[ansinew]{inputenc}
\usepackage{float}
\usepackage{hyperref}
\numberwithin{equation}{section}

\usepackage{subcaption}
\usepackage{cases}
\usepackage{enumitem}
\usepackage{etoolbox}
\usepackage{algorithm, algorithmicx, algpseudocode}
\usepackage[superscript,biblabel]{cite}

\usepackage{amssymb}
\usepackage{diagbox}

\DeclareCaptionLabelFormat{andtable}{#1~#2  \&  \tablename~\thetable}
\AfterEndEnvironment{problem}{\noindent\ignorespaces}

\newcommand{\R}{{\rm I\!R}}  % The real numbers
\newcommand{\M}{{\rm I\!M}}  
\newcommand*\interior[1]{\mathring{#1}}
\newcommand\norm[1]{\left\lVert#1\right\rVert}

\newcommand\Nsteps{N_t}
\DeclareMathOperator*{\argmin}{argmin}

\newtheoremstyle{noIndent}% name
  {}%      Space above, empty = `usual value'
  {}%      Space below
  {\itshape}% Body font
  {}%         Indent amount (empty = no indent, \parindent = para indent)
  {\bfseries}% Thm head font
  {.}%        Punctuation after thm head
  {.5em}% Space after thm head: \newline = linebreak
  {}%         Thm head spec
\theoremstyle{noIndent}
\newtheorem{problem}{Problem}[section]

\begin{document}

\title[UMBERTO EMIL MORELLI et al]{Novel Methodologies for Solving the Inverse Unsteady Heat Transfer Problem of Estimating the Boundary Heat Flux in Continuous Casting Molds}
\thanks{Funded by the European Union's Horizon 2020 research and innovation programme under the Marie
Skaodowska-Curie Grant Agreement No. 765374.
It also was partially supported by the Ministry of Economy, Industry and Competitiveness through the Plan Nacional de I+D+i (MTM2015-68275-R), by the Agencia Estatal de Investigacion through project [PID2019-105615RB-I00/ AEI / 10.13039/501100011033], by the European Union Funding for Research and Innovation - Horizon 2020 Program - in the framework of European Research Council Executive Agency:  Consolidator Grant H2020 ERC CoG 2015 AROMA-CFD project 681447 "Advanced Reduced Order Methods with Applications in Computational Fluid Dynamics" and  INDAM-GNCS project "Advanced intrusive and non-intrusive model order reduction techniques and applications", 2019.}
%    Remove any unused author tags.

%    author one information
\author{Umberto Emil Morelli\textsuperscript{1,2,3*}}
\email{umbertoemil.morelli@usc.es}

\author{Patricia Barral\textsuperscript{1,2}}
\author{Peregrina Quintela\textsuperscript{1,2}}
\author{Gianluigi Rozza\textsuperscript{3}}
\author{Giovanni Stabile\textsuperscript{3}}

\address{\textsuperscript{1}Universidade de Santiago de Compostela, Santiago de Compostela, Spain}
\address{\textsuperscript{2} Technological Institute for Industrial Mathematics (ITMATI), Santiago de Compostela, Spain}
\address{\textsuperscript{3}Scuola Internazionale Superiore di Studi Avanzati (SISSA), Trieste, Italy}

\date{}

\dedicatory{}

\begin{abstract}
    In this work, we investigate the estimation of the transient mold-slab heat flux in continuous casting molds given some thermocouples measurements in the mold plates.
    Mathematically, we can see this problem as the estimation of a Neumann boundary condition given pointwise state observations in the interior of the domain.
    We formulate it in a deterministic inverse problem setting.
    After introducing the industrial problem, we present the mold thermal model and related assumptions.
    Then, we formulate the boundary heat flux estimation problem in a deterministic inverse problem setting using a sequential approach according to the sequentiality of the temperature measurements.
    We consider different formulations of the inverse problem.
    For each one, we develop novel direct methodologies exploiting a space parameterization of the heat flux and the linearity of the mold model.
    We construct these methods to be divided into a computationally expensive offline phase that can be computed before the process starts, and a cheaper online phase to be performed during the casting process.
    To conclude, we test the performance of the proposed methods in two benchmark cases.
\end{abstract}

\maketitle

%% ====================================================================
\section{Introduction}
\label{section:introduction}
Most of the steel produced everyday worldwide is made by Continuous Casting (CC).\cite{WorldSteel2018}
Continuous casters have been around for many decades now and a long sequence of improvements have increased through the years their productivity (i.e. the casting speed) and the quality of the casted products.

To motivate and contextualize this research, we provide a brief overview on the CC process.
It starts by tapping the liquid metal from the ladle into the tundish.
In the tundish, the metal flow is regulated and smoothed.
Through the Submerged Entry Nozzle~(SEN), the metal is drained into a mold.
The role of the mold is to cool down the steel until it has a solid skin which is thick and cool enough to be supported by rollers in the secondary cooling region.

At the outlet of the mold, the metal is still molten in its inner region.
Supported by rollers, it is cooled until complete solidification by directly spraying water over it.
At the end of this secondary cooling region, the casting is completed.
This is just a brief overview on the CC process.
We refer the interested reader to Irving's monograph on the subject.\cite{Irving1993}

In this work, we focus on CC of thin slabs, i.e. slabs with rectangular cross section with thickness smaller than 70~mm and width between 1 and 1.5~m.
Thanks to the small thickness, the solidification in the slab is relatively fast, consequently the casting speed is generally high, between 7 and 14 meters per minute.

Thin slab molds are made of four different plates: two wide plates and two lateral plates, all made of copper (see Figure~\ref{fig:castingSchematic}).
In general, lateral plates can be moved or changed to modify the slab section dimensions.
The geometry of these plates is more complex than one can expect: they have drilled channels where the cooling water flows, slots in the outside face for thermal expansion, thermocouples, and fastening bolts.
To compensate the shrinkage of the slab with the cooling and minimize the gap, the molds are tapered.
Moreover, the upper portion of the mold forms a funnel to accommodate the SEN.

\begin{figure}[!htb]
    \centering
    \includegraphics[width=0.6\textwidth]{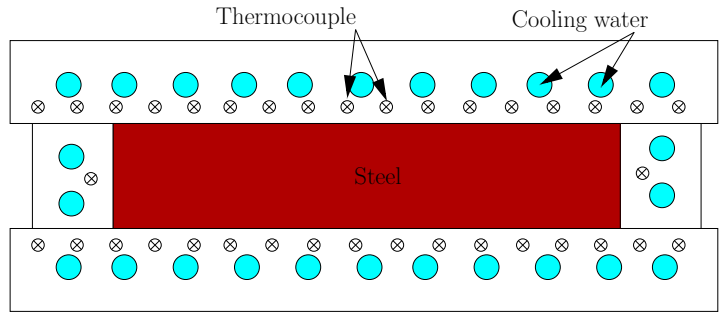}
    \caption{Schematic of a horizontal section of the mold (the casting direction is perpendicular to the image).}
\label{fig:castingSchematic}
\end{figure}

Due to the high casting speed and the related strong thermal gradient, several complex and coupled phenomena related to steel flow, solidification, mechanics, and heat transfer appear in the mold region.
This complexity makes the mold the most critical part of the CC process.
Here, safety and productivity issues must be addressed.

For example, a common issue is the sticking of the steel to the mold.
In this case, it is essential to quickly detect the problem and reduce the casting speed, otherwise it can lead to dangerous events that could force the shutdown of the caster.
Less frequent but more catastrophic events are the liquid break-out and the excessive increase of the mold temperature.
The former is due to a non-uniform cooling of the metal with the skin being so thin to break.
The latter is generally considered the most dangerous event in a casting plant.
In fact, if the mold temperature is high enough to cause the boiling of the cooling water, we have a dramatic decrease in the heat extraction.
Then, the temperature in the mold quickly rises, that could cause the melting of the mold itself.
Both these incidents are very dangerous and costly.
In fact, they generally require the shutdown of the caster, the substitution of expensive components and an extended turnaround.

For all these reasons, the early detection of problems in the mold is crucial for a safe and productive operation of continuous casters.
Their detection becoming more difficult as casting speed (thus productivity) of the casters increases.

Until now, operators faced all these problems by equipping the molds with sensors.
Among other parameters, they measure the pointwise temperature of the mold by thermocouples (see Figure~\ref{fig:castingSchematic}) and the cooling water temperature as well as its flow at the inlet and outlet of the cooling system.
On one hand, thermocouples temperatures are used to have insight of the mold temperature field.
On the other, the water temperature rise is used to approximate the heat extracted from the steel.

This approach allowed to run continuous casters for decades.
Nevertheless, it has several drawbacks: it relies on the experience of operators, gives very limited information about the heat flux at the mold-slab interface, and is customized for each geometry so it requires new effort to be applied to new designs.
So, with the always increasing casting speed of modern casters, a new and more reliable tool for analyzing the mold behavior is necessary.

According to CC operators and designers, knowing the local heat flux between mold and slab is the most important information in monitoring the mold.
Moreover, we should estimate it in real-time for the early detection of issues and a proper monitoring.
By considering the mold plates to be our domain and focusing our interest on its thermal behavior, the mold-slab heat flux can be seen as a Neumann Boundary Condition (BC) in the model.
To compute its value, we pose the following inverse problem: given the temperature measurements provided by the thermocouples, estimate the boundary heat flux at the mold-slab interface.
In a previous publication~\cite{Morelli2021}, we developed a novel methodology for the solution of this problem using a steady-state mold model.
The present work is an extension of the previous one considering the more challenging unsteady-state case.

After deriving the mold heat transfer model in Section~\ref{section:directProblem_assumptions}, we discuss in Section~\ref{section:inverse} the steel-mold heat flux estimation problem and propose novel methodologies for its solution.
Finally, we design in Section~\ref{section:numerical} some numerical benchmark test cases that we use to study the performance of the proposed inverse solvers.

%%%%%%%%%%%%%%%%%%%%%%%%%%%%%%%%%%%%%%%%%%%%%%%%%%%%%%%%%%%%%%%%%%%%%%%
\section{Mathematical Model}\label{section:directProblem_assumptions}
%%%%%%%%%%%%%%%%%%%%%%%%%%%%%%%%%%%%%%%%%%%%%%%%%%%%%%%%%%%%%%%%%%%%%%%

A detailed description of the physical phenomena that occur in the mold region of a caster can be found in our previous work\cite{Morelli2021}.
Here, we only mention that the physical phenomena happening in the interior of the mold are extremely complex and tightly coupled (thermodynamic reactions, multiphase flow, free liquid surfaces and interfaces, solidification, etc.).
Then, monitoring the casting by simulating all of them from the SEN to the secondary cooling region would be extremely complex and computationally expensive to deal with, especially for real-time applications.

According to CC operators, to monitor the mold behavior it is sufficient to know the mold-slab heat flux.
Then, given the mold plate physical properties, its geometry and the cooling water temperature, our approach is to solve an inverse problem having as input data the temperature measurements made by the thermocouples that are buried inside the mold plates.

As mentioned, the mold solid plates are the computational domain while the mold-slab heat flux is a Neumann BC on a portion of its boundary to be determined as solution of an inverse problem.
Then, the direct problem corresponds to a  model for the heat transfer in the mold plates.
In the rest of this section, we describe the mold thermal model that we use in the present investigation and the related assumptions.

In modeling the thermal behavior of the mold, we consider the following well established assumptions\cite{Morelli2021}:
\begin{itemize}
  \item The copper mold is assumed a homogeneous and isotropic solid material.
  \item The thermal expansion of the mold and its mechanical distortion are negligible.
  \item The material properties are assumed constant.
  \item The boundaries in contact with air are assumed adiabatic.
  \item The heat transmitted by radiation is neglected.
  \item The cooling water temperature is known at the inlet and outlet of the cooling system. Moreover, it is assumed to be constant in time and linear with respect to the $z$ coordinate (see Figure \ref{fig:directProblem_schematicDomain}).
  \item No boiling in the water is assumed.
\end{itemize}
We refer to our previous work\cite{Morelli2021} for the motivations related to these assumptions.

According to these assumptions, we consider in the following an unsteady-state three-dimensional heat conduction model posed on the (solid) copper mold, with a convective BC in the portion of the boundary in contact with the cooling water, a Neumann BC in the portion of the boundary in contact with the steel, and adiabatic BC in the portion of the boundary in contact with air.

After introducing some notation and the computational domain, we devote the present section to the formulation of the mold model.
As common when dealing with inverse problems, we refer to it as the direct problem.
We conclude this section by discussing its numerical discretization.

%%%%%%%%%%%%%%%%%%%%%%%%%%%%%%%%%%%%%%%%%%%%%%%%%%%%%%%%%%%%%%%%%%%%%%%
\subsection{Computational Domain and Notation}

Consider a solid domain, $\Omega$, which is assumed to be an open Lipschitz bounded subset of $\R^3$,  with smooth boundary $\Gamma$ (see Figure~\ref{fig:directProblem_schematicDomain}).
Let $\Gamma = \Gamma_{s_{in}}\cup\Gamma_{s_{ex}} \cup \Gamma_{sf}$ where $\interior{\Gamma}_{s_{in}}$, $\interior{\Gamma}_{s_{ex}}$ and $\interior{\Gamma}_{sf}$ are disjoint sets.
Moreover, given $t_f\in\R^+$, we consider the time domain $(0, t_f]$.
The Eulerian Cartesian coordinate vector is denoted by $\mathbf{x}\in \Omega$ and $\mathbf{n}(\mathbf{x})$ represents the unit normal vector that is directed outwards from $\Omega$ at point $\mathbf{x}\in \Gamma$.

\begin{figure}[htb]
    \centering
    \includegraphics[width=0.6\textwidth]{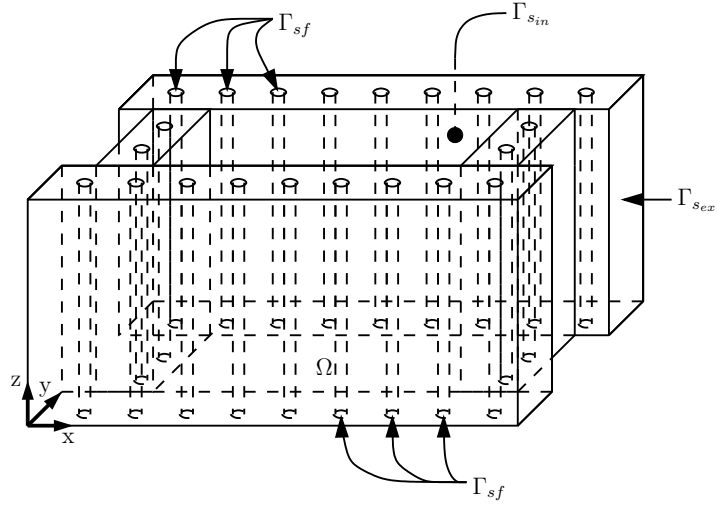}	
    \caption{Schematic of the mold domain, $\Omega$, and its boundaries (images taken from Morelli et al.\cite{Morelli2021}).}
\label{fig:directProblem_schematicDomain}
\end{figure}

In this setting, $\Omega$ corresponds to the region of the space occupied by the mold.
The interface between the mold and the cooling system is denoted by $\Gamma_{sf}$.
While $\Gamma_{s_{in}}$ is the portion of the mold boundary in contact with the solidifying steel.
Finally, we denote the remaining part of the mold boundary with $\Gamma_{s_{ex}}$.

\subsection{Direct Problem}
We shall assume all along the following assumptions on the data:
\begin{enumerate}[start=1,label={(H\arabic*)}]
        \setcounter{enumi}{0}
    \item\label{ass:unsteadyHeatConduction_conductivity} The thermal conductivity is constant and strictly positive: $k_s \in \R^+$.
    \item There is no heat source inside the mold domain.
    \item\label{ass:unsteadyHeatConduction_constantFlowDensitySpecHeat} The density and specific heat are constant and strictly positive: $\rho\in \R^+$, $C_p\in \R^+$.
    \item\label{ass:unsteadyHeatConduction_heatTransferCoeff} The heat transfer coefficient on $\Gamma_{sf}$ is constant and strictly positive: $h\in \R^+$.
    \item\label{ass:unsteadyHeatConduction_knownCoolingTemperature} The cooling water temperature, $T_f$, is known, constant in time, and belongs to $L^q(\Gamma_{sf})$.
    \item\label{ass:unsteadyHeatConduction_initialCondition} The initial temperature, $T_0$, is known and belongs to $L^2(\Omega)$.
    \item\label{ass:unsteadyHeatConduction_boundaryHeatFlux} The steel-mold heat flux, $g$, belongs to $L^r(0,t_f;L^q(\Gamma_{s_{in}}))$.
\end{enumerate}
In \ref{ass:unsteadyHeatConduction_knownCoolingTemperature}, \ref{ass:unsteadyHeatConduction_boundaryHeatFlux} we assume that $r,q \in (2, + \infty)$ and
\begin{equation}
    \frac{1}{r} + \frac{1}{q} < \frac{1}{2}.
    \label{eq:rq_inequality}
\end{equation}
Notice that it implies $r,q > 2$.

Under the assumptions \ref{ass:unsteadyHeatConduction_conductivity}-\ref{ass:unsteadyHeatConduction_boundaryHeatFlux}, we propose the following three-dimensional, unsteady-state, heat conduction model
\begin{problem}{}
Find $T$ such that
  \begin{equation}
     \rho C_{p}\frac{\partial T}{\partial t} - k_s \Delta T = 0, \quad      \text{in }\Omega \times (0,t_f],
     \label{eq:3DhcModelUnsteady_laplacian}
  \end{equation}
    with BCs and Initial Condition (IC)
    \begin{equation}
    \label{wams}
    \begin{cases}
        -k_s \nabla T \cdot \mathbf{n} = g             &    \text{on } \Gamma_{s_{in}} \times (0,t_f]\\
        -k_s \nabla T \cdot \mathbf{n} = 0             &    \text{on } \Gamma_{s_{ex}} \times (0,t_f],\\
        -k_s \nabla T \cdot \mathbf{n} = h(T -  T_f)   &    \text{on } \Gamma_{sf} \times (0,t_f],\\
          T(\cdot,0) = T_{0}                           &    \text{in } \Omega.
    \end{cases}
    \end{equation}
  \label{prob:directUnsteady}
\end{problem}

A weak solution is now defined by testing against a smooth function and formally integrating by parts.
\begin{definition}{}
    We say that a function $T \in C([0,t_f]; L^2(\Omega)) \cap L^2(0, t_f; H^1(\Omega))$ is a \textit{weak solution} of Problem~\ref{prob:directUnsteady} on $[0,t_f]$ for some $t_f>0$ if
    \begin{equation}
        \begin{aligned}
            -\rho C_{p} \int_0^{t_f} \int_{\Omega} T(\mathbf{x}, t) \frac{\partial \psi(\mathbf{x}, t)}{\partial t} d\mathbf{x}dt + k_s \int_0^{t_f} \int_{\Omega} \nabla T(\mathbf{x}, t) \nabla \psi(\mathbf{x}, t) d\mathbf{x}dt + \int_0^{t_f} \int_{\Gamma_{sf}} h T (\mathbf{x}, t) \psi(\mathbf{x}, t) d\Gamma dt = \\
            \rho C_p\int_{\Omega} T_0(\mathbf{x}) \psi(\mathbf{x}, 0) d\mathbf{x} - \int_0^{t_f} \int_{\Gamma_{s_{in}}} g(\mathbf{x}, t) \psi(\mathbf{x}, t) d\Gamma dt  + \int_0^{t_f} \int_{\Gamma_{sf}} h T_f(\mathbf{x}) \psi(\mathbf{x}, t)  d \Gamma dt, \\ 
        \end{aligned}
    \end{equation}
    for all $\psi \in H^1(0,t_f; H^1(\Omega))$ that satisfy $\psi(\cdot, t_f) = 0$ in  $\Omega$.
\end{definition}

%We introduce the following definition
%\begin{definition}{}
%    We define the operator $A_2$ on the space $L^2(\Omega) \times L^2(\Gamma_{sf})$ by
%    \begin{equation}
%        \begin{aligned}
%            &D(A_2) := \left\{ (T,0) : T \in H^1(\Omega), k \Delta T \in L^2(\Omega), k \nabla T \cdot \mathbf{n} \in L^2(\Gamma_{sf}) \right\}\\
%            &A_2(u,0) := \left( -k \Delta T, k \nabla T \cdot \mathbf{n} + h T|_{\Gamma_{sf}} \right)\\            
%        \end{aligned}
%    \end{equation}
%\end{definition}
%
%Now we can define mild and classical solutions of Problem~\ref{prob:directUnsteady}.

Thanks to Nittka\cite{Nittka2011b} (its Theorem 2.11 and Corollary 2.13), we have
\begin{theorem}
    Let assumptions \ref{ass:unsteadyHeatConduction_conductivity}-\ref{ass:unsteadyHeatConduction_boundaryHeatFlux} and (\ref{eq:rq_inequality}) hold.
    Then, there exists a unique weak solution of Problem~\ref{prob:directUnsteady} on $[0, t_f ]$.
\end{theorem} 
Finally, we recall Theorem~3.3 in Nittka\cite{Nittka2011b}
\begin{theorem}
    Let assumptions \ref{ass:unsteadyHeatConduction_conductivity}-\ref{ass:unsteadyHeatConduction_boundaryHeatFlux} and (\ref{eq:rq_inequality}) hold.
    Then, the weak solution $T$ of Problem~\ref{prob:directUnsteady}  is in $C([0,t_f]; C(\overline{\Omega}))$.
    So, in particular, $T(\mathbf{x},t) \rightarrow T_0(\mathbf{x})$ uniformly on $\overline{\Omega}$ as $t\rightarrow 0$.
\end{theorem} 

Regarding the numerical solution of Problem~\ref{prob:directUnsteady}, we use the finite volume method for its discretization.
Given a tessellation $\mathcal{T}$ of the domain, $\Omega$, we write the discrete unknown $(T_C(t))_{C\in \mathcal{T}} $ as the real vector $\mathbf{T}(t)$, belonging to $\R^{N_h}$ with $N_h = \text{size}(\mathcal{T})$.
Then, we write the spatially discretized problem as
\begin{equation}
    \rho C_p M \frac{d \mathbf{T}(t)}{d t} + A \mathbf{T}(t) = \mathbf{b}(t),\quad  t\in(0, t_f],
\label{eq:benchmark3DhcModelUnsteady_discreteDirectProblemLinSys}
\end{equation}
where $M\in \M^{N_h \times N_h}$ is the mass matrix,   $A\in \M^{N_h \times N_h}$ is the stiffness matrix and $\mathbf{b}\in \R^{N_h}$ the source term.
The value of each element of $M$, $A$ and $\mathbf{b}$ depends on the particular finite volume scheme for the discretization and the mesh used.
Since our problem is a classic diffusion problem, we refer for further details regarding the finite volume discretization to the Eymard's monograph.\cite{Eymard2000}

To discretize (\ref{eq:benchmark3DhcModelUnsteady_discreteDirectProblemLinSys}) in time, we divide the time interval of interest into $N_{T}$ regular steps
\begin{equation}
    t^0 = 0,\quad t^{n+1} = t^n + \Delta t,\quad n=0,\dots,N_T - 1,\quad \Delta t = \frac{t_f}{N_T}.
\end{equation}
From now on, we denote by $f^j$ an approximation of a given function $f(t)$ at time $t^j$.

For the time discretization, we consider the implicit Euler scheme.
It is a first-order implicit scheme.
With this discretization, (\ref{eq:benchmark3DhcModelUnsteady_discreteDirectProblemLinSys}) becomes
\begin{equation}
    (\rho C_p M + \Delta t A) \mathbf{T}^{n+1} =  \rho C_p M \mathbf{T}^n + \Delta t \mathbf{b}^{n+1}, \quad n = 0 , \dots , N_T - 1.
    \label{eq:3DhcModelUnsteady_discreteEuler}
\end{equation}
Notice that, thanks to hypotheses \ref{ass:unsteadyHeatConduction_conductivity}-\ref{ass:unsteadyHeatConduction_boundaryHeatFlux}, matrices $A$ and $M$ are time independent.

\section{Inverse Problem}
\label{section:inverse}
%***************************************************************************%
In this section, we discuss the formulation and solution of the boundary heat flux estimation problem.
We state it in an inverse problem setting using data assimilation.
We begin this section with a literature survey, then we discuss the mathematical formulation of the problem and, finally, the methodology that we developed for its solution.

%***************************************************************************%
\subsection{State of the Art}

The literature on inverse heat transfer problems is vast\cite{Ling2003, Loulou2006, Jin2007, Huang1996}.
We refer to Alifanov's\cite{Alifanov1988}, Orlande's\cite{Orlande2010}, Beck and Clair's\cite{Beck1985}, and Chang's\cite{Chang2017} works for a detailed review.
In the literature, other researchers also investigated the particular problem of computing the mold-slab heat flux from temperature measurements in the mold\cite{Ahin2006, Ranut2012, Udayraj2017, Mahapatra1991}.
From a mathematical point of view, the present problem fits in the framework of estimating a Neumann BC (the heat flux) having as data pointwise measurements of the state (the temperature) inside the domain.
Such problems were also addressed in investigations not related to heat transfer\cite{Raymond2013, Vitale2012, Huang1998}.

The story of inverse heat transfer problems started in the 50s when aerospace engineers were interested in knowing the thermal properties of heat shields and heat fluxes on the surface of space vehicles during re-entry.
The first approach was purely heuristic, then in the 60s and 70s, researchers moved to a more mathematically formal approach.
In fact, most of the regularization theory that we use nowadays to treat ill-posed problems was developed during these years\cite{Alifanov1988, Tikhonov1963, Beck1968, Beck1970, Chen1976}.
%Here, we discuss in general the most popular methodologies used for the solution of inverse heat transfer problems.
%The literature on the subject is vast and we refer to~\cite{Alifanov1988, Orlande2010, Beck1985, Chang2017} for a detailed description.

The first approach for estimating the boundary heat flux in CC molds was to select a heat flux profile, and then by trial and error adapt it to match the measured temperatures\cite{Mahapatra1991}.
Pinhero et al.\cite{Pinheiro2000} were the first to use an optimal control framework and regularization methods.
They used a steady-state version of the 2D mold model proposed by Samarasekera and Brimacombe\cite{Samarasekera1982} and parameterized the heat flux with a piecewise constant function.
Finally, they used Tikhonov's regularization for solving the inverse problem and validated the results with experimental measurements.
A similar approach was used more recently by Rauter et al.\cite{Ranut2012, Rauter2008, Ranut2011}.
They estimated the heat flux transferred from the solidifying steel to the mold wall both in a 2D and 3D domain.
They used a steady-state heat conduction model for the mold and parameterized the heat flux with a piecewise linear profile in 2D and symmetric cosine profile in 3D.
For the solution of the inverse problem, they used the Conjugate Gradient Method (CGM) and a mixed GA-SIMPLEX algorithm\cite{Nelder1965} in 2D while in 3D they only used the GA-SIMPLEX algorithm.
Their results were also tested with experimental data.

%Better describe Man2004 because they clam a real-time computation
Using a 3D unsteady-state heat conduction model in the strand and the mold with a Robin condition at the mold-strand interface, Hebi et al.\cite{Man2004, Hebi2006} attempted to estimate the solidification in CC round billets.
Similarly to the present work, they looked for the heat transfer coefficient that minimizes a distance between measured and computed temperatures at the thermocouples' points.
Assuming the heat transfer coefficient to be piecewise constant, they iteratively adapted each piece to match the measured temperature.
However, in the validation with plant measurements, they did not obtain good agreement.
A similar approach was used by Gonzalez et al.\cite{Gonzalez2003} and Wang et al.\cite{Wang2016, ZhangWang2017, Hu2018, Tang2012}, the latter using a Neumann condition at the mold-strand interface.

Udayraj et al.\cite{Udayraj2017} applied the conjugate gradient method with adjoint problem for the solution of the steady-state 2D mold-slab heat flux estimation problem.
This methodology was first proposed by Alifanov\cite{Alifanov1988} for the regularization of boundary inverse heat transfer problems without the need of parameterizing the heat flux.
However, as we also proved in our previous work~\cite{Morelli2021}, this method underestimates the heat flux away from the measurements.
To overcome this issue, Udayraj et al. proposed to average the computed heat flux at each step and use the uniform averaged value as initial estimation for the following step.
However, the obtained results were not satisfying.

Since the real-time requirement is common in industrial applications, real-time methodologies for the solution of these problems have already been investigated in the literature.
In particular, Videcoq et al.\cite{Videcoq2008} used a Branch Eigenmodes Reduced Model\cite{Videcoq2006} for the real-time identification of the heat source strength variations in a 3D  non-linear inverse heat conduction problem.
Later, for solving the same problem, Girault et al.\cite{Girault2010} used the  Modal Identification Method\cite{Girault2005} for generating the reduced model.
Finally, Aguado et al.\cite{Aguado2015} coupled classical harmonic analysis with recent model order reduction techniques (Proper Generalized Decomposition) to solve in real-time the transient heat equation at monitored points, also showing the applicability of their method to inverse problems.

To conclude, also deep learning techniques were investigated.
Wang and Yao\cite{WangYao2011} used the inverse problem solution technique developed by Hebi et al.\cite{Hebi2006} and a set of experimental temperature measurements to train a Neural Network (NN) for on-line computation.
Similarly, Chen et al.\cite{Chen2014} used the fuzzy inference method for estimating the mold heat flux.
In both works, they modeled the mold with a 2D steady-state heat conduction model in the solid and parameterized the boundary heat flux.

Our contribution to the literature is the  development of novel methods for solving the unsteady-state 3D inverse heat transfer problem in CC molds that exploits the parameterization of the heat flux.
We propose different novel direct methodologies that exploit an offline-online decomposition.
In fact, we divide them in a computationally expensive offline phase and an online phase whose computational cost is much smaller.
The advantage is that we compute offline phase once and for all before starting the casting process.
Then, while the machine is running, we only need to solve the cheap online phase.
Moreover, in this work, we design some benchmark cases for this application, and we use them to test the performances of the proposed methodologies.

%***************************************************************************%
\subsection{Inverse Problem Formulation}

Before proceeding with the mathematical formulation of the inverse problem, we do some technical considerations that will guide us in the process.
First, the thermocouples measure the temperature at the sampling frequency $f_{samp}$.
This sampling frequency is typically of $1$~Hz and we will assume this value all along this investigation (notice that different values of $f_{samp}$ are compatible with the following discussion).
Second, every sampling period $T_{samp} = 1 / f_{samp} = 1$ s, the thermocouples provide a new set of measurements, so we have a regular sequence of measurements in time.

That said, we consider the problem of estimating the heat flux, $g$, on $\Gamma_{s_{in}}$, in between the last acquired measurement instant and the previous one.
In this way, we follow the sequentiality of the measured data in our solution procedure according to the real-time purpose of this research.

We introduce the following notation.
Let $\Psi:=\{\mathbf{x}_1, \mathbf{x}_2 , \dots, \mathbf{x}_P \}$ be a collection of points in $\Omega$ and  $\Upsilon := \{\tau^0,\tau^1, \dots, \tau^{P_t}\}$ a collection of points in $[0,t_f]$ such that $\tau^k = t^{k\Nsteps}$ (see Figure~\ref{fig:timeline}).
According to the introduced sequential approach, we consider the following restriction of Problem~\ref{prob:directUnsteady} to $(\tau^{k-1}, \tau^k]$, $1 \leq k \leq P_t$, as direct problem
\begin{problem}{}
  Let $1 \leq k \leq P_t$ and $g^k(\mathbf{x},t)$ be a given heat flux on $\Gamma_{s_{in}} \times (\tau^{k-1},\tau^k]$. Find  $T^k$ such that
  \begin{equation}
     \rho C_p\frac{\partial T^k}{\partial t} - k_s \Delta T^k = 0, \quad      \text{in }\Omega \times (\tau^{k-1}, \tau^k],
      \label{eq:sequential_directUnsteady_EQ1}
  \end{equation}
  with BCs and IC
    \begin{equation}
    \label{wams}
    \begin{cases}
      -k_s \nabla T^k \cdot \mathbf{n} = g^k              & \text{on } \Gamma_{s_{in}} \times (\tau^{k-1}, \tau^k],\\
      -k_s \nabla T^k \cdot \mathbf{n} = 0                & \text{on } \Gamma_{s_{ex}} \times (\tau^{k-1}, \tau^k],\\
      -k_s \nabla T^k \cdot \mathbf{n} = h(T^k -  T_f)    & \text{on } \Gamma_{sf}     \times (\tau^{k-1}, \tau^k],\\
      T^k(\cdot,\tau^{k-1}) = T^{k-1}(\cdot, \tau^{k-1})  & \text{in } \Omega,
    \end{cases}
    \end{equation}
  \label{prob:sequential_directUnsteady}
\end{problem}
where $T^0(\cdot, \tau^{0}) = T_0$, being $T_0$ the initial temperature. 

\begin{figure}[!htb]
    \centering
    \includegraphics[width=0.6\textwidth]{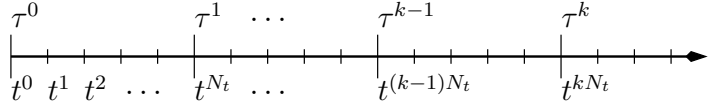}
    \caption{Time line for the inverse problem.}
\label{fig:timeline}
\end{figure}

So basically, we are dividing the time domain into chunks going from one measurement time to the next one in a way that facilitates the definition of the inverse problems below.
Before formulating it, we introduce some further notation.
We define the application $(\mathbf{x}_i, \tau^k) \in \Psi \times \Upsilon \rightarrow \hat{T}(\mathbf{x}_i, \tau^k)\in \R^+$, $1\leq i \leq P, 1 \leq k \leq P_t$, $\hat{T}(\mathbf{x}_i, \tau^k)$ being the experimentally measured temperature at $(\mathbf{x}_i, \tau^k) \in \Psi \times \Upsilon$.
Moreover, to simplify the notation, and if there is no room for error, we denote
\begin{equation}
    \hat{T}^k(\mathbf{x}_i) := \hat{T}(\mathbf{x}_i, \tau^k), \quad 1 \leq i \leq P, \ 1 \leq k \leq P_t,
\end{equation}
and we let $T^k[g]$ represent the solution of Problem~\ref{prob:sequential_directUnsteady} corresponding to heat flux $g$ on $\Gamma_{s_{in}} \times (\tau^{k-1}, \tau^k]$.

At each measurement interval $k$, $1 \leq k \leq P_t$, we propose an iterative procedure, assuming that, for $k \geq 1$, $g^l$ and $T^l[g^l]$, $0 \leq l \leq k-1$, have been computed.
Using a least square, deterministic approach, we state two different inverse problems for Problem~\ref{prob:sequential_directUnsteady}.
In the first one, we consider as functional to be minimized a distance between the measured and computed temperatures at the thermocouples.
Then, we state it as
\begin{problem}{(\textbf{Inverse})}
    Being $g^l$ and $T^l[g^l]$, $1 \leq l \leq k-1$, known, and given the temperature measurements $\hat{T}^k(\mathbf{x}_i)$, $1 \leq i \leq P$, find $g^k \in C(\tau^{k-1}, \tau^k; L^q(\Gamma_{s_{in}}))$ which minimizes the functional
  \begin{equation}
    S_1^k[g^k]=\frac{1}{2} \sum_{i=1}^{P} [T^k[g^k](\mathbf{x}_{i}, \tau^k) - \hat{T}^k(\mathbf{x}_i)]^2.
      \label{eq:unsteadyInverseFunctional1}
  \end{equation}
\label{prob:sequentialUnsteadyInverseProblem}
\end{problem}
Here, we denote $T^0[g^0]=T_0$.

The second inverse problem that we consider includes in the cost functional the $L^2$-norm of the heat flux.
Thus, we write it as
\begin{problem}{(\textbf{Inverse})}
    Being $g^l$ and $T^l[g^l]$, $1 \leq l \leq k-1$, known, and given the temperature measurements $\hat{T}^k(\mathbf{x}_i)$, $1 \leq i \leq P$, find $g^k \in C(\tau^{k-1}, \tau^k; L^q(\Gamma_{s_{in}}))$ which minimizes the functional
  \begin{equation}
      S_2^k[g^k]=\frac{1}{2} \sum_{i=1}^{P} [T^k[g^k](\mathbf{x}_{i}, \tau^k) - \hat{T}^k(\mathbf{x}_i)]^2 + p_g \langle g^k(\tau^k), g^k(\tau^k) \rangle_{L^2(\Gamma_{s_{in}})},
      \label{eq:sequentialUnsteadyInverseProblem_heatNorm_functional}
  \end{equation}
\label{prob:sequentialUnsteadyInverseProblem_heatNorm}
\end{problem}
where $p_g\ \left[ \text{K\textsuperscript{2} / W\textsuperscript{2}} \right]$ is a weight applied to the heat flux norm.

%%%%%%%%%%%%%%%%%%%%%%%%%%%%%%%%%%%%%%%%%%%%%%%%%%%%%%%%%%%%%%%%%%%%%%%%%%%%%%%
\subsection{Inverse Solver for $S_1^k$}\label{sec:sequentialUnsteadyInverseSolver}

In this section, we discuss a novel methodology for solving Problem~\ref{prob:sequentialUnsteadyInverseProblem}.
In particular, we mimic the methodology developed by the authors for a steady-state mold model~\cite{Morelli2021}, expanding it to the unsteady case.

We exploit a suitable parameterization of the heat flux, $g^k$.
To properly parameterize it, we start by considering that we want to parameterize an unknown function $g^k$ in $L^r(\tau^{k-1}, \tau^k; L^q(\Omega))$, $1 \leq k \leq P_t$.
Then, we notice that in thin slab casting molds, the thermocouples are all located few millimeters inward from $\Gamma_{s_{in}}$.
All together they form a uniform 2D grid on a surface parallel to the $\Gamma_{s_{in}}$ boundary (see Figure~\ref{fig:unsteadyAnalyticalBenchmarkDomainSchematic}(b)).
Thus, a possible choice for the parameterization of $g^k$ is to use Radial Basis Functions (RBFs) centered at the projections of the thermocouples points on $\Gamma_{s_{in}}$\cite{Buhmann2003}.
By using this parameterization, we end up having as many basis functions as thermocouples.
Note that the methodology is very well adapted to the application in use.

In particular, we parameterize $g^k$ by Gaussian RBFs which allow us to separate the time and space dependence.
These are continuous functions with global support in $\Gamma_{s_{in}}$.
However, the following discussion can be applied to other basis functions.

The parameterization of the boundary heat flux reads (see Prando's appendix\cite{Prando2016})
\begin{equation}
    g^k(\mathbf{x}, t) \approx \mathscr{g}^k(\mathbf{x}, t) =  \sum_{i=1}^{P} g_i^k(t) \phi_{i}(\mathbf{x}), \quad \text{ for } t\in(\tau^{k-1}, \tau^k],
    \label{eq:parametrizedSequentialHeatFlux}
\end{equation}
where the $\phi_i(\mathbf{x})$ are $P$ known basis functions, and the $g_i^k(t)$ are the respective time dependent unknown weights.

To define the RBFs, let $\pmb{\xi}_i, 1 \leq i \leq P$, be the projection of the point $\mathbf{x}_i \in \Psi$ on $\Gamma_{s_{in}}$, i.e. such that
\begin{equation}
  \pmb{\xi}_i = \argmin_{\pmb{\xi}\in\Gamma_{s_{in}}} \left\lVert \mathbf{x}_i - \pmb{\xi}  \right\rVert_2, \quad \mathbf{x}_i\in \Psi.
    \label{eq:themocouplesProjectionOnGammaIn}
\end{equation}
By centering the RBFs in these points, their expression is
\begin{equation}
    \phi_j(\mathbf{x}) = e^{- \left( \eta \left\lVert \mathbf{x} - \pmb{\xi}_j  \right\rVert_2 \right)^2},\quad \text{ for } j=1,2,\dots,P,
    \label{eq:GaussianRBF}
\end{equation}
where $\eta$ is the shape parameter of the Gaussian basis.
By increasing (decreasing) its values, the radial decay of the basis slows down (speeds up).

In this work, we explore two different approaches to the time parameterization. 
In the first one, we consider $g_i^k$ independent of time
\begin{equation}
    g_i^k(t)=w_i^k, \quad \text{ for } t\in(\tau^{k-1}, \tau^k],  1 \leq i \leq P,
    \label{eq:heatFluxConstTimeBasis}
\end{equation}
being $w_i^k$ real numbers. 
In this way, the heat flux is assumed to be piecewise constant, i.e. constant between consecutive measurement instants.

The second approach is to consider the heat flux to be continuous piecewise linear in $(0,t_f]$, being a polynomial of degree 1 between the sampling times. 
Then, we assume the weights $g_i^k(t)$ to be linear in time in the interval $(\tau^{k-1}, \tau^k]$.
Moreover, in this second case, the following continuity is assumed
\begin{equation}
    g_i^k(t)|_{t\downarrow \tau^{k-1}}=g_i^{k-1}(t)|_{t\uparrow \tau^{k-1}}.
\end{equation}
In turn, we characterize $g_i^k(t)$ as
\begin{equation}
    g_i^k(t) = w_i^{k-1} + (t - \tau^{k-1}) \frac{w^k_i - w_i^{k-1}}{\tau^k - \tau^{k-1}}, \quad \text{ for } t\in(\tau^{k-1}, \tau^k]. 
    \label{eq:heatFluxLinearTimeBasis}
\end{equation}

Notice that by doing parameterization (\ref{eq:parametrizedSequentialHeatFlux}), we change the problem from estimating a function in an infinite dimensional space at each time interval $(t^{(k - 1) N_t}, t^{k N_t}]=(\tau^{k-1}, \tau^k]$,  to estimating the vector $\mathbf{w}^k =( w_1^k, w_2^k ,\dots, w_P^k)^T$ in $\R^P$, for each $1 \leq k \leq  P_t$.

Now, at each time interval $(\tau^{k-1},\tau^k]$, the objective of the inverse problem is to determine $\mathbf{w}^k$ which identifies $\mathscr{g}^k$ once the elements of the basis $\phi_{i}$, $i = 1,2,\dots,P$ are fixed.
We state the inverse problem as
\begin{problem}{(\textbf{Inverse})}
    Given the temperature measurements $\hat{T}(\Psi, \tau^k)$, find $\hat{\mathbf{w}}^k\in\R^{P}$, $1 \leq k \leq P_t$, which minimizes the functional
  \begin{equation}
     S_1^k[\mathbf{w}^k]=\frac{1}{2}  \sum_{i=1}^{P} [T^k[\mathbf{w}^k](\mathbf{x}_{i}, \tau^k) - \hat{T}^k(\mathbf{x}_i)]^2,
      \label{eq:costFunction}
  \end{equation}
\label{prob:benchmark3DhcModelUnsteady_inverseProblem_parametrizedBC}
\end{problem}
where if there is not room for confusion $T^k[\mathbf{w}^k]$ denotes the temperature $T^k[\mathscr{g}^k]$, with $\mathscr{g}^k$ defined as in (\ref{eq:parametrizedSequentialHeatFlux}) and $g_i^k(t)$  given by (\ref{eq:heatFluxConstTimeBasis}) or (\ref{eq:heatFluxLinearTimeBasis}).

For a later use, we define the general vector $\mathbf{a}^k \in \R^P$ as the vector of the values of a general field $a(\mathbf{x},t)$ at the measurement points and at the measurement time $\tau^k$, such as
\begin{equation}
    (\mathbf{a}^k)_i = a(\mathbf{x}_i, \tau^k).
    \label{eq:unsteadyGeneralVector}
\end{equation}
Moreover, given $\mathbf{w}^k$, we define the residual vector $\mathbf{R}^k[\mathbf{w}^k] \in \R^{P}$ as
\begin{equation}
    (\mathbf{R}^k[\mathbf{w}^k])_{i} := (\mathbf{T}^k[\mathbf{w}^k])_{i} - (\hat{\mathbf{T}}^k)_{i},\quad i = 1,2,\dots,P.
    \label{eq:inverseUnsteady_residualDefinition}
\end{equation}
Thanks to (\ref{eq:inverseUnsteady_residualDefinition}), we rewrite the cost functional (\ref{eq:costFunction}) as
\begin{equation}
  S_1^k[\mathbf{w}^k]=\frac{1}{2} \mathbf{R}^k[\mathbf{w}^k]^T \mathbf{R}^k[\mathbf{w}^k].
\end{equation}

To minimize it, we write the critical point equation
\begin{equation}
    \frac{\partial S_1^k[\hat{\mathbf{w}}^k]}{\partial w_{j}^k} = \sum_{i=1}^{P} (R^k[\hat{\mathbf{w}}^k])_{i} \frac{\partial (\mathbf{T}^k[\hat{\mathbf{w}}^k])_{i}}{\partial w_j^k} = 0,\text{ for } j=1,2,\dots,P.
  \label{eq:benchmark3DhcModelUnsteadySequential_criticalPoint}
\end{equation}
Thus, for each $k$, $1 \leq k \leq P_t$,  the solution of this equation will provide the weights vector $\hat{\mathbf{w}}^k$ corresponding to a critical point of $S_1^k$.

To explicitly obtain from (\ref{eq:benchmark3DhcModelUnsteadySequential_criticalPoint}) an equation for the weights that minimize our functional $S^k_1$, we exploit the linearity of Problem~\ref{prob:sequential_directUnsteady}.
To derive it, we consider separately the piecewise constant (\ref{eq:heatFluxConstTimeBasis}) and the piecewise linear (\ref{eq:heatFluxLinearTimeBasis}) cases.

%************************************************************************************************************************************************************************************%
\subsubsection{Piecewise Constant Approximation of the Heat Flux}
%************************************************************************************************************************************************************************************%

Suppose to have the solutions to the following auxiliary problems
\begin{problem}{}
    For each $i$, $1 \leq i \leq P$, find $T_{\phi_{i}}$ such that
  \begin{equation}
      \rho C_p \frac{\partial T_{\phi_{i}}}{\partial t} - k_s \Delta T_{\phi_{i}} = 0, \quad      \text{in }\Omega \times (\tau^{0}, \tau^1],
      \label{eq:sequential_directUnsteady2_EQ1_const}
  \end{equation}
  with BCs and IC
    \begin{equation}
    \label{wams1}
    \begin{cases}
      -k_s \nabla T_{\phi_{i}} \cdot \mathbf{n} = \phi_{i}          & \text{on } \Gamma_{s_{in}} \times (\tau^{0}, \tau^1],\\
      -k_s \nabla T_{\phi_{i}} \cdot \mathbf{n} = 0                 & \text{on } \Gamma_{s_{ex}} \times (\tau^{0}, \tau^1],\\
      -k_s \nabla T_{\phi_{i}} \cdot \mathbf{n} = hT_{\phi_{i}}     & \text{on } \Gamma_{sf} \times   (\tau^{0}, \tau^1] ,\\
      T_{\phi_{i}}(\cdot,\tau^0) = 0                                & \text{in } \Omega.
    \end{cases}
    \end{equation}
    \label{prob:sequential_basis}
\end{problem}
\begin{problem}{}
    For each $k$, $1 \leq k \leq P_t$, find $T_{IC}^k$ such that
  \begin{equation}
      \rho C_p \frac{\partial T_{IC}^k}{\partial t} - k_s \Delta T_{IC}^k = 0, \quad      \text{in }\Omega \times (\tau^{k-1}, \tau^k],
      \label{eq:TICunsteady_EQ1}
  \end{equation}
  with BCs and IC
    \begin{equation}
    \label{wams2}
    \begin{cases}
      - k_s \nabla T_{IC}^k \cdot \mathbf{n} = 0                                & \text{on } (\Gamma_{s_{in}} \cup \Gamma_{s_{ex}}) \times (\tau^{k-1}, \tau^k],\\
      - k_s \nabla T_{IC}^k \cdot \mathbf{n} = h \left( T_{IC}^k - T_f \right)  & \text{on } \Gamma_{sf} \times (\tau^{k-1}, \tau^k],\\
      T_{IC}^k(\cdot,\tau^{k-1}) = T^{k - 1}(\cdot,\tau^{k-1})                  & \text{in } \Omega,
    \end{cases}
    \end{equation}
\label{prob:Tic}
\end{problem}
with
\begin{equation}
    T_{IC}^1(\cdot,\tau^{0}) = T_0.
    \label{eq:TICunsteady_EQ1_IC}
\end{equation}

Notice that Problem~\ref{prob:sequential_basis} does not depend on the measurement instants index $k$.
Then, we define it only in the first interval $(\tau^{0}, \tau^1]$ and,  if needed, translate it such as
\begin{equation}
    T_{\phi_i}^k (\mathbf{x}, t) = T_{\phi_i}(\mathbf{x}, t - \tau^{k - 1}), \text{ for } t \in (\tau^{k-1}, \tau^k].
    \label{eq:basisTranslation}
\end{equation}

We can now state
\begin{theorem}
    Given $T_{IC}^k$ and $\mathbf{w}^k, 1\leq k \leq P_t$ (and so $\mathscr{g}^k$ defined as (\ref{eq:parametrizedSequentialHeatFlux})) and $T_{\phi_{i}}, 1 \leq i \leq P$,  the function defined as
    \begin{equation}
        T^k[\mathscr{g}^k] = \sum_{i=1}^{P} g_i^k T^k_{\phi_{i}} + T_{IC}^k,
      \label{eq:directProbDecomposition_sequential_const}
    \end{equation}
    is the solution to Problem~\ref{prob:sequential_directUnsteady} associated with the heat flux  $\mathscr{g}(x,t)$ which in each measurement subinterval $(\tau^{k-1}, \tau^k]$ coincides with $\mathscr{g}^k(x,t)$ given by (\ref{eq:parametrizedSequentialHeatFlux}) with $g_i^k(t)$ as in (\ref{eq:heatFluxConstTimeBasis}).
    \label{teo:decomposition_const}
\end{theorem}

\begin{proof}
    Substituting (\ref{eq:directProbDecomposition_sequential_const}) in (\ref{eq:sequential_directUnsteady_EQ1}), taking into account (\ref{eq:heatFluxConstTimeBasis}) so $g_i^k$ is constant in $(\tau^{k-1},\tau^k]$, and considering (\ref{eq:sequential_directUnsteady2_EQ1_const}) and (\ref{eq:TICunsteady_EQ1}), we have
    \begin{equation}
      \begin{aligned}
          \rho C_{p} \frac{\partial \left(\sum_{i=1}^P  g_i^k T^k_{\phi_i} + T_{IC}^k \right)}{\partial t} - k_s \Delta \left(\sum_{i=1}^P  g_i^k T^k_{\phi_i} + T_{IC}^k \right) =  \sum_{i=1}^P  g_i^k \left(\rho C_{p} \frac{\partial T^k_{\phi_i}}{\partial t} - k_s \Delta T^k_{\phi_i} \right) + \\
          \rho C_{p} \frac{\partial T_{IC}^k}{\partial t} - k_s \Delta T_{IC}^k  = 0 \text{ in }\Omega \times (\tau^{k-1}, \tau^k].\\
      \end{aligned}
    \end{equation}
    Similarly, for the BCs we have
    \begin{equation}
        - k_s \nabla \left(\sum_{i=1}^{P} g_i^k T^k_{\phi_i} + T_{IC}^k \right) \cdot \mathbf{n} = \sum_{i=1}^{P} g_i^k \left( - k_s \nabla T^k_{\phi_{i}} \cdot \mathbf{n} \right) =  \sum_{i=1}^{P} g_i^k \phi_{i} = \mathscr{g}^k,\text{ on } \Gamma_{s_{in}} \times (\tau^{k-1}, \tau^k], 
    \end{equation}
    \begin{equation}
        - k_s \nabla \left( \sum_{i=1}^{P} g_i^k  T^k_{\phi_i} + T_{IC}^k \right) \cdot \mathbf{n} = 0,  \text{ on }  \Gamma_{s_{ex}} \times (\tau^{k-1}, \tau^k],
    \end{equation}
    and
    \begin{equation}
            - k_s \nabla \left( \sum_{i=1}^{P} g_i^k  T^k_{\phi_i} + T_{IC}^k \right) \cdot \mathbf{n} = h \left[ \sum_{i=1}^{P} g_i^k T^k_{\phi_{i}} + T_{IC}^k - T_f \right] = h \left( T^k[\mathscr{g}^k] - T_f \right),  \text{ on } \Gamma_{sf} \times (\tau^{k-1}, \tau^k].
    \end{equation}
    With respect to the IC, at each interval we must proceed by induction.
    For $k=1$, thanks to (\ref{eq:TICunsteady_EQ1_IC}),
    \begin{equation}
        T^1[g^1](\cdot,\tau^0) = T^0(\cdot).
    \end{equation}
    For $k>1$,  taking into account (\ref{wams1}) and (\ref{wams2})
    \begin{equation}
            T^k[g^k](\cdot, \tau^{k-1}) = \sum_{i=1}^{P}  g_i^{k} T^k_{\phi_i}(\cdot, \tau^{k-1}) + T_{IC}^k(\cdot, \tau^{k-1}) = T^{k-1}(\cdot,\tau^{k-1}), \text{ in } \Omega.
    \end{equation}
    This ends the proof.
\end{proof}

\paragraph{Solving the minimization problem for $S_1^k$ with piecewise constant approximation of the heat flux}

Thanks to (\ref{eq:basisTranslation}) and (\ref{eq:directProbDecomposition_sequential_const}), (\ref{eq:benchmark3DhcModelUnsteadySequential_criticalPoint}) can be written as
\begin{equation}
    \begin{aligned}
        \sum_{i=1}^{P} (R^k[\hat{\mathbf{w}}^k])_{i} \frac{\partial \left( \sum_{l=1}^{P} g_l^k \mathbf{T}_{\phi_{l}} + \mathbf{T}_{IC}^k \right)_{i}}{\partial w_j^k} = \sum_{i=1}^{P} (R^k[\hat{\mathbf{w}}^k])_{i} \frac{\partial \left( \sum_{l=1}^{P} w_l^k \mathbf{T}_{\phi_{l}} + \mathbf{T}_{IC}^k \right)_{i}}{\partial w_j^k} =& \\
        \mathbf{R}^k[\hat{\mathbf{w}}^k]^T \left(\mathbf{T}_{\phi_{j}} \right) = & \  0,\quad \text{ for } j=1,2,\dots,P,
    \end{aligned}
  \label{eq:benchmark3DhcModelUnsteadySequential_criticalPoint_exp_const}
\end{equation}
where $(\mathbf{T}_{\phi_{l}})_i$ is the vector containing the values of the field $T_{\phi_l}(\mathbf{x}_i, \tau_1)$ at the measurement points.
We recall that $\mathbf{T}_{\phi_{l}}$ is independent of $k$.

Let us define the matrix $\Theta$ in $\M^{P \times P}$ such that
\begin{equation}
    \Theta_{i, j} = T_{\phi_{j}}(\mathbf{x}_i, \tau^1).
    \label{eq:sequentialThetaMatrix_const}
\end{equation}
Equation (\ref{eq:benchmark3DhcModelUnsteadySequential_criticalPoint_exp_const}) can now be written as
\begin{equation}
    \Theta^T\mathbf{R}^k[\hat{\mathbf{w}}^k] = \mathbf{0}.
\end{equation}

Using (\ref{eq:directProbDecomposition_sequential_const}), the vector associated to the solution of the direct problem at the measurement points, $\mathbf{T}^k[\mathbf{w}^k] \in \R^P$, for each $k$, can be written as
\begin{equation}
    \mathbf{T}^k[\hat{\mathbf{w}}^k] = \sum_{j=1}^P w^k_j \mathbf{T}_{\phi_{j}} + \mathbf{T}^k_{IC} = \Theta \hat{\mathbf{w}}^k + \mathbf{T}^k_{IC}.
    \label{eq:benchmark3DhcModelUnsteadySequential_solAtMeasurements_const}
\end{equation}
Recalling the definition of $\mathbf{R}^k$ and (\ref{eq:benchmark3DhcModelUnsteadySequential_solAtMeasurements_const}), we have
\begin{equation}
    \Theta^T \mathbf{R}^k[\hat{\mathbf{w}}^k] = \Theta^T(\Theta \hat{\mathbf{w}}^k  + \mathbf{T}^k_{IC} - \hat{\mathbf{T}}^k) = \mathbf{0}.
\end{equation}
Therefore, for each $k$, $1 \leq k \leq P_t$, a solution of the inverse problem, $\hat{\mathbf{w}}^k$, is obtained by solving the linear system
\begin{equation}
    \Theta^T \Theta \hat{\mathbf{w}}^k = \Theta^T(\hat{\mathbf{T}}^k - \mathbf{T}^k_{IC}).
  \label{eq:linSys_parametrizedBC_sequential_constant}
\end{equation}
It is important to notice that the system matrix, $\Theta^T \Theta$, is $k$-independent. 

Equation (\ref{eq:linSys_parametrizedBC_sequential_constant}) is generally called the normal equation.
By solving this linear system, we obtain the weights, $\hat{\mathbf{w}}^k$, that correspond to a critical point of the functional $S_1^k$, defined by (\ref{eq:costFunction}).
As mentioned, $\Theta$ is constant.
So, we can compute it once and for all in an offline phase.  

The proposed methodology for the solution of the inverse Problem~\ref{prob:sequentialUnsteadyInverseProblem} is summarized in Algorithm~\ref{alg:inverseSolver_constant}.
It is important to notice that, for each time interval $(\tau^{k-1},\tau^k]$, $T_{\phi_i}$ and the related vector do not change but $T_{IC}^k$ does because its IC depends on the temperature field at time $\tau^{k-1}$.

\begin{algorithm}[htb!]
 \caption{Inverse solver for the solution of Problem~\ref{prob:sequentialUnsteadyInverseProblem} with piecewise constant parameterization in time of the heat flux, $g$.}
 \label{alg:inverseSolver_constant}
    \hspace*{\algorithmicindent} \textbf{OFFLINE}\\
    \hspace*{\algorithmicindent} \textbf{Input} RBF shape parameter, $\eta$; thermocouples measurement points and times, $\Psi, \Upsilon$ 
    \begin{algorithmic}[1]
        \State Setup RBF parameterization by (\ref{eq:GaussianRBF})
        \State Compute $T_{\phi_i}$ for $i = 1,2,\dots,P$ by solving Problem~\ref{prob:sequential_basis}
        \State Assemble matrix $\Theta$ by (\ref{eq:sequentialThetaMatrix_const})
    \end{algorithmic} 
    \hspace*{\algorithmicindent} \\
    \hspace*{\algorithmicindent} \textbf{ONLINE}\\
    \hspace*{\algorithmicindent} \textbf{Input} Initial condition, $T_0$  
    \begin{algorithmic}[1]
        \State Set $k = 1$ 
        \While{\texttt{$k \leq P_t$}}
            \State Read the thermocouples measurements, $\hat{\mathbf{T}}^k$
            \State Compute $T_{IC}^k$ by solving Problem~\ref{prob:Tic}
            \State Assemble $\mathbf{T}_{IC}^k$
            \State Compute $\hat{\mathbf{w}}^k$ by solving (\ref{eq:linSys_parametrizedBC_sequential_constant})
            \State Compute $g_i^k(t)$ by (\ref{eq:heatFluxConstTimeBasis})
            \State Compute the heat flux $\mathscr{g}(\mathbf{x}, t)$ for $t\in(\tau^{k-1}, \tau^k]$ by (\ref{eq:parametrizedSequentialHeatFlux}) 
            \State Use (\ref{eq:directProbDecomposition_sequential_const}) to compute $T^k[g^k]$
            \State $k = k + 1$
        \EndWhile
    \end{algorithmic}
\end{algorithm}

Notice that, in this setting, (\ref{eq:linSys_parametrizedBC_sequential_constant}) is an affine map from the observations, $\hat{\mathbf{T}}^k$, to the heat flux weights, $\mathbf{w}^k$.
Consequently, we have that the existence and uniqueness of the solution of the inverse problem depends on the invertibility of the matrix $\Theta^T \Theta$.
We can easily see that the matrix is symmetric and positive semi-definite.
In general, however, we cannot ensure that it is invertible.
In fact, the invertibility depends on the choice of the basis function, the computational domain, and the BCs.

Before moving to the piecewise linear case, we recall the offline-online decomposition of Algorithm~\ref{alg:inverseSolver_constant}.
In the offline phase, we compute $T_{\phi_i}$ for $i=1,2,\dots,P$ by solving Problem~\ref{prob:sequential_basis} and assemble the related matrix $\Theta$.
Then, in the online phase, we input the measurements $\hat{\mathbf{T}}$, solve Problem~\ref{prob:Tic} and the linear system (\ref{eq:linSys_parametrizedBC_sequential_constant}).

For the choice made when selecting the basis functions, this linear system has the dimensions of the number of thermocouples (quite small, in general).
However, solution of Problem~\ref{prob:Tic} involves the solution of a full order model whose computational cost depends on the discretization size.
Consequently, this method is not suitable for real-time as it is.
To achieve real-time performances, we need to apply model order reduction techniques.
This will be the subject of our future work.

As a final remark, we notice, that for the application of this method, linearity of the direct problem is essential.
In fact, it is a necessary condition for Theorem~\ref{teo:decomposition_const}.

%************************************************************************************************************************************************************************************%
\subsubsection{Piecewise Linear Heat Flux}\label{sec:linearInverseSolver}
%************************************************************************************************************************************************************************************%
Suppose to have the solution to
\begin{problem}{}
    For each $i$, $1 \leq i \leq P$, find $T_{d_i}$ such that
    \begin{equation}
      \rho C_p\frac{\partial T_{d_i}}{\partial t} - k_s \Delta T_{d_i} = -\rho C_p  T_{\phi_i}, \quad      \text{in }\Omega \times (\tau^{0}, \tau^1],
  \end{equation}
  with BCs and IC
    \begin{equation}
    \label{wams3}
    \begin{cases}
      -k_s \nabla T_{d_i} \cdot \mathbf{n} = 0          & \text{on } (\Gamma_{s_{in}} \cup \Gamma_{s_{ex}}) \times (\tau^{0}, \tau^1],\\
      -k_s \nabla T_{d_i} \cdot \mathbf{n} = h T_{d_i}  & \text{on } \Gamma_{sf} \times (\tau^{0}, \tau^1],\\
      T_{d_i}(\cdot,\tau^{0}) = 0                       & \text{in } \Omega.
    \end{cases}
    \end{equation}
\label{prob:Td}
\end{problem}
From the solution of this problem, we define
\begin{equation}
    T_{d_i}^k (\mathbf{x}, t) = T_{d_i}(\mathbf{x}, t - \tau^{k - 1}), \quad \text{for } t \in (\tau^{k-1}, \tau^k].
\end{equation}

Then, we state
\begin{theorem}
    Given $T_{IC}^k$ and $\mathbf{w}^k, 1\leq k \leq P_t$ (and so $\mathscr{g}^k$ defined as (\ref{eq:parametrizedSequentialHeatFlux})), $T_{d_i}$ and $T_{\phi_{i}}$, $1\leq i \leq P$,  then the function defined as
    \begin{equation}
        T^k[\mathscr{g}^k] = \sum_{i=1}^{P} \left( g_i^k T^k_{\phi_{i}} + {g_i^k}' T^k_{d_i} \right) +  T_{IC}^k,
      \label{eq:directProbDecomposition_sequential_linear}
    \end{equation}
    is the solution to Problem~\ref{prob:sequential_directUnsteady} associated with the heat flux  $\mathscr{g}(x,t)$ which in each measurement subinterval $(\tau^{k-1}, \tau^k]$ coincides with $\mathscr{g}^k(x,t)$ given by (\ref{eq:parametrizedSequentialHeatFlux}), with $g_i^k(t)$ as in (\ref{eq:heatFluxLinearTimeBasis}).
    \label{teo:decomposition_linear}
\end{theorem}

\begin{proof}
    Substituting (\ref{eq:directProbDecomposition_sequential_linear}) in (\ref{eq:sequential_directUnsteady_EQ1}) and considering (\ref{eq:sequential_directUnsteady2_EQ1_const}) and (\ref{eq:TICunsteady_EQ1}), we have
    \begin{equation}
      \begin{aligned}
          \rho C_{p} \frac{\partial \left[ \sum_{i=1}^P  \left( g_i^k(t) T^k_{\phi_i} + {g_i^k}' T^k_{d_i} \right) + T_{IC}^k \right]}{\partial t} - k_s \Delta \left[\sum_{i=1}^P  \left( g_i^k(t) T^k_{\phi_i} + {g_i^k}' T^k_{d_i} \right) + T_{IC}^k \right] =  \\
          \rho C_{p} \left[ \sum_{i=1}^P  {g_i^k}' T^k_{\phi_i} + \sum_{i=1}^P g_i^k(t) \left( \frac{\partial  T^k_{\phi_i}}{\partial t} - k_s \Delta T^k_{\phi_i} \right) \right]  + \sum_{i=1}^P {g_i^k}' \left( \rho C_{p}  \frac{\partial T^k_{d_i}}{\partial t} - k_s \Delta T^k_{d_i} \right) + \rho C_{p} \frac{\partial T_{IC}^k}{\partial t} - k_s\Delta T_{IC}^k =\\
          \rho C_{p} \sum_{i=1}^P  {g_i^k}' T^k_{\phi_i}  + \sum_{i=1}^P {g_i^k}' \left( \rho C_{p}  \frac{\partial T^k_{d_i}}{\partial t} - k_s \Delta T^k_{d_i} \right)   = 0, \text{ in }\Omega \times (\tau^{k-1}, \tau^k].\\
      \end{aligned}
    \end{equation}
    Similarly, for the BCs we have
    \begin{equation}
        \begin{aligned}
            - k_s \nabla \left( \sum_{i=1}^{P} \left( g_i^k T^k_{\phi_{i}} + {g_i^k}' T^k_{d_i} \right) +  T_{IC}^k \right) \cdot \mathbf{n} = \sum_{i=1}^{P} g_i^k \left( - k_s \nabla T^k_{\phi_{i}} \cdot \mathbf{n} \right) = \sum_{i=1}^{P} g_i^k \phi_{i} = \mathscr{g}^k, \text{ on } \Gamma_{s_{in}} \times (\tau^{k-1}, \tau^k],
        \end{aligned}
    \end{equation}
    \begin{equation}
        - k_s \nabla \left( \sum_{i=1}^{P} \left( g_i^k T^k_{\phi_{i}} + {g_i^k}' T^k_{d_i} \right) +  T_{IC}^k \right) \cdot \mathbf{n} = 0  \text{ on } \Gamma_{s_{ex}} \times (\tau^{k-1}, \tau^k],
    \end{equation}
    and
    \begin{equation}
        \begin{aligned}
            - k_s \nabla \left( \sum_{i=1}^{P} \left( g_i^k T^k_{\phi_{i}} + {g_i^k}' T^k_{d_i} \right) +  T_{IC}^k \right) \cdot \mathbf{n} =  h \left[ \sum_{i=1}^{P} \left( g_i^k T^k_{\phi_{i}} + {g_i^k}' T^k_{d_i} \right) + T_{IC}^k - T_f \right] =  h \left( T^k[\mathscr{g}^k] - T_f \right),\\
             \text{ on } \Gamma_{sf} \times (\tau^{k-1}, \tau^k].
        \end{aligned}
    \end{equation}
    With respect to the IC, at each interval we must proceed by induction.
    For $k=1$,
    \begin{equation}
        T^1[g^1](\cdot,\tau^0) = T^0(\cdot) \text{ in } \Omega.
    \end{equation}
    For $k>1$, thanks to (\ref{wams1}), (\ref{wams2}) and (\ref{wams3}), we deduce,
    \begin{equation}
            T^k[g^k](\cdot, \tau^{k-1}) = \sum_{i=1}^{P} \left(  g_i^{k} (\tau^{k-1}) T^k_{\phi_i}(\cdot, \tau^{k-1}) + {g_i^k}' T^k_{d_i}(\cdot, \tau^{k-1}) \right) + 
            T_{IC}^k(\cdot, \tau^{k-1})  =  T^{k-1}(\cdot,\tau^{k-1}), \text{ in } \Omega.   
    \end{equation}
\end{proof}

\paragraph{Solving the minimization problem for $S_1^k$ with piecewise linear approximation of the heat flux}

Thanks to Theorem~\ref{teo:decomposition_linear}, we have
\begin{equation}
    \begin{aligned}
        \frac{\partial \mathbf{T}^k[\mathbf{w}^k]}{\partial w_j^k} = \frac{\partial \left[ \sum_{l=1}^{P} \left( g_l^k(\tau^k) \mathbf{T}^k_{\phi_{l}}(\tau^k) + {g_l^k}' \mathbf{T}^k_{d_l}(\tau^k) \right) +  \mathbf{T}_{IC}^k(\tau^k) \right]}{\partial w_j^k} = \\
        \frac{\partial \left\{ \sum_{l=1}^{P} \left[ w_l^k \mathbf{T}^k_{\phi_{l}}(\tau^k) + (w_l^k - w_l^{k-1}) f_{samp} \mathbf{T}^k_{d_l}(\tau^k) \right] +  \mathbf{T}_{IC}^k(\tau^k) \right\} }{\partial w_j^k}  = \mathbf{T}^k_{\phi_{j}}(\tau^k) + f_{samp} \mathbf{T}^k_{d_j}(\tau^k) = \mathbf{T}_{\phi_{j}}(\tau^1) + f_{samp} \mathbf{T}_{d_j}(\tau^1).
    \end{aligned}
    \label{eq:dTdw}
\end{equation}
Thanks to (\ref{eq:dTdw}), (\ref{eq:benchmark3DhcModelUnsteadySequential_criticalPoint}) in the linear case is rewritten as
\begin{equation}
    \mathbf{R}^k[\hat{\mathbf{w}}^k]^T (\mathbf{T}_{\phi_{j}}  + f_{samp} \mathbf{T}_{d_j}) = 0,\text{ for all } j=1,2,\dots,P.
  \label{eq:benchmark3DhcModelUnsteadySequential_criticalPoint2_linear}
\end{equation}
Let us define the matrices $\tilde{\Theta}, \Theta_d$ in $\M^{P \times P}$ such that
\begin{equation}
    \left( \Theta_d \right)_{i,j} := f_{samp} T_{d_j}(\mathbf{x}_i,\tau^1), \quad \left( \tilde{\Theta} \right)_{i,j} := \left( \Theta \right)_{i,j} + \left( \Theta_d \right)_{i,j}   ,
    \label{eq:sequentialThetaMatrix_linear}
\end{equation}
both independent of the index $k$.

Using (\ref{eq:sequentialThetaMatrix_linear}) and (\ref{eq:directProbDecomposition_sequential_linear}), we have
\begin{equation}
\begin{aligned}
    \mathbf{T}^k[\hat{\mathbf{w}}^k] =  \sum_{l=1}^{P} \left( g_l^k(\tau^k) \mathbf{T}_{\phi_{l}} + {g_l^k}' \mathbf{T}_{d_l} \right) +  \mathbf{T}_{IC}^k = \sum_{l=1}^{P} \left[ w_l^k \mathbf{T}_{\phi_{l}} + (w_l^k - w_l^{k-1}) f_{samp} \mathbf{T}_{d_l} \right] +  \mathbf{T}_{IC}^k =  \tilde{\Theta} \hat{\mathbf{w}}^k -  \Theta_d \hat{\mathbf{w}}^{k-1} + \mathbf{T}^k_{IC}.
    \label{eq:benchmark3DhcModelUnsteadySequential_solAtMeasurements_linear}
\end{aligned}
\end{equation}
Recalling the definition of $\mathbf{R}^k$, taking into account (\ref{eq:sequentialThetaMatrix_linear}) and (\ref{eq:benchmark3DhcModelUnsteadySequential_solAtMeasurements_linear}), system (\ref{eq:benchmark3DhcModelUnsteadySequential_criticalPoint2_linear}) can be written as
\begin{equation}
    {(\tilde{\Theta})}^T  \mathbf{R}^k[\hat{\mathbf{w}}^k] = {(\tilde{\Theta})}^T\left( \tilde{\Theta} \hat{\mathbf{w}}^k -  \Theta_d \hat{\mathbf{w}}^{k-1} + \mathbf{T}^k_{IC}   - \hat{\mathbf{T}}^k \right) = \mathbf{0}.
\end{equation}
Therefore, a solution of the inverse problem $S_1^k$, considering $g^k(t)$ piecewise linear, is obtained by solving the linear system
\begin{equation}
    {(\tilde{\Theta})}^T \tilde{\Theta} \hat{\mathbf{w}}^k = {(\tilde{\Theta})}^T   \left(\hat{\mathbf{T}}^k + \Theta_d \hat{\mathbf{w}}^{k-1} - \mathbf{T}^k_{IC} \right).
  \label{eq:linSys_parametrizedBC_sequential_linear}
\end{equation}
Notice that as in the piecewise constant case, the system matrix, $(\tilde{\Theta})^T \tilde{\Theta}$, is $k$-independent. 

We summarize the proposed methodology for the solution of the inverse Problem~\ref{prob:sequentialUnsteadyInverseProblem} in Algorithm~\ref{alg:inverseSolver_linear}.
Similarly to the piecewise constant case, for each time interval $(\tau^{k-1},\tau^k]$, $T_{\phi_i}$, $T_{d_i}$ and the related vectors do not change but $T_{IC}^k$ does because its IC depends on the temperature field at time $\tau^{k-1}$.

\begin{algorithm}[htb!]
 \caption{Inverse solver for the solution of Problem~\ref{prob:sequentialUnsteadyInverseProblem} with piecewise linear parameterization in time of the heat flux, $g$.}
 \label{alg:inverseSolver_linear}
    \hspace*{\algorithmicindent} \textbf{OFFLINE}\\
    \hspace*{\algorithmicindent} \textbf{Input} RBF shape parameter, $\eta$; thermocouples measurement points and times, $\Psi, \Upsilon$ 
    \begin{algorithmic}[1]
        \State Setup RBF parameterization by (\ref{eq:GaussianRBF})
        \State Compute $T_{\phi_i}$ for $i = 1,2,\dots,P$ by solving Problem~\ref{prob:sequential_basis}
        \State Compute $T_{d_i}$ for $i = 1,2,\dots,P$ by solving Problem~\ref{prob:Td}
        \State Assemble matrices $\tilde{\Theta}$ and $\Theta_d$
    \end{algorithmic} 
    \hspace*{\algorithmicindent} \\
    \hspace*{\algorithmicindent} \textbf{ONLINE}\\
    \hspace*{\algorithmicindent} \textbf{Input} Initial condition, $T_0$
    \begin{algorithmic}[1]
        \State Set $k = 1$ 
        \While{\texttt{$k \leq P_t$}}
            \State Read the thermocouples measurements, $\hat{\mathbf{T}}^k$
            \State Compute $T_{IC}^k$ by solving Problem~\ref{prob:Tic}
            \State Assemble $\mathbf{T}_{IC}^k$
            \State Compute $\hat{\mathbf{w}}^k$ by solving (\ref{eq:linSys_parametrizedBC_sequential_linear}) 
            \State Compute $g_i^k(t)$ by  (\ref{eq:heatFluxLinearTimeBasis})
            \State Compute the heat flux $\mathscr{g}(\mathbf{x}, t)$ for $t\in(\tau^{k-1}, \tau^k]$ by (\ref{eq:parametrizedSequentialHeatFlux}) 
            \State Use (\ref{eq:benchmark3DhcModelUnsteadySequential_solAtMeasurements_linear}) to compute $T^k[g^k]$
            \State $k = k + 1$
        \EndWhile
    \end{algorithmic}
\end{algorithm}

Also in this setting, (\ref{eq:linSys_parametrizedBC_sequential_linear}) is an affine map from the observations, $\hat{\mathbf{T}}^k$, to the heat flux weights, $\hat{\mathbf{w}}^k$.
Consequently, we have that the existence and uniqueness of the solution of the inverse problem depends on the invertibility of the matrix ${(\tilde{\Theta})}^T \tilde{\Theta}$.
It is symmetric and positive semi-definite.
However, we cannot ensure that it is invertible.
In fact, the invertibility depends on the choice of the basis functions, the computational domain, and the BCs.

We notice that, also in the piecewise linear case, the offline-online decomposition holds.
Moreover, the linear system dimensions are the same of the piecewise constant case and the linearity of the direct problem is still a necessary condition also for Theorem~\ref{teo:decomposition_linear}.

%***************************************************************************%
\subsection{Inverse Solver for $S_2^k$}\label{sec:sequentialUnsteadyInverseSolver_heatNorm}

In this section, we discuss the solution of the inverse Problem~\ref{prob:sequentialUnsteadyInverseProblem_heatNorm}.
In particular, we extend the previously developed methodologies adapting them to the cost function $S_2^k$ as defined in (\ref{eq:sequentialUnsteadyInverseProblem_heatNorm_functional}).

As shown in detail in Section~\ref{sec:unsteadyInverseBenchmark_linear}, the piecewise linear inverse solver of Section~\ref{sec:linearInverseSolver} presents instability issues, under certain conditions.
Then, we stated this second inverse problem with the purpose of stabilizing the solution.
To do this, we designed Problem~\ref{prob:sequentialUnsteadyInverseProblem_heatNorm} from Problem~\ref{prob:sequentialUnsteadyInverseProblem} by adding to the cost function (\ref{eq:sequentialUnsteadyInverseProblem_heatNorm_functional}) a term that penalizes the heat flux norm $\langle g^k(\tau^k), g^k(\tau^k) \rangle_{L^2(\Gamma_{s_{in}})}$. 

Also in this case, we exploit the parameterization of the heat flux (\ref{eq:parametrizedSequentialHeatFlux}).
As a consequence, we introduce the inverse problem in terms of $\mathbf{w}^k$ as
\begin{problem}{(\textbf{Inverse})}
    Given the temperature measurements $\hat{T}(\Psi, \Upsilon)$, find $\hat{\mathbf{w}}^k\in\R^{P}$, $1 \leq k \leq P_t$, which minimizes the functional
  \begin{equation}
          S_2^k[\mathbf{w}^k] = \frac{1}{2}  \sum_{i=1}^{P} [T^k[\mathbf{w}^k](\mathbf{x}_{i}, \tau^k) - \hat{T}^k(\mathbf{x}_i)]^2 + p_g \int_{\Gamma_{s_{in}}} \left( \sum_{j=1}^P w_j^k \phi_j(\mathbf{x}) \right)\left( \sum_{q=1}^P w_q^k \phi_q(\mathbf{x}) \right) d\Gamma .
      \label{eq:costFunction_heatNorm}
  \end{equation}
\label{prob:benchmark3DhcModelUnsteady_inverseProblem_parametrizedBC_heatNorm}
\end{problem}
Notice that (\ref{eq:costFunction_heatNorm}) holds true for both the piecewise constant and linear  parameterization of the heat flux since
\begin{equation}
    g_i^k(\tau^k) = w_i^{k},
\end{equation}
for (\ref{eq:heatFluxConstTimeBasis}) as well as for (\ref{eq:heatFluxLinearTimeBasis}). 

Considering the second term of the right hand side of (\ref{eq:costFunction_heatNorm}), we can write
\begin{equation}
        \int_{\Gamma_{s_{in}}} \left( \sum_{l=1}^P w_l^k \phi_l(\mathbf{x}) \right)\left( \sum_{q=1}^P w_q^k \phi_q(\mathbf{x}) \right) d\Gamma = \sum_{l=1}^P \sum_{q=1}^P w_l^k w_q^k \int_{\Gamma_{s_{in}}} \phi_l(\mathbf{x}) \phi_q(\mathbf{x}) d\Gamma.
    \label{eq:costFunction_heatNorm_secondTerm2}
\end{equation}

Let us define the vectors of $\R^{P^2}$
\begin{equation}
    \pmb{\phi}_\phi = 
        \begin{bmatrix}
            \int_{\Gamma_{s_{in}}} \phi_1(\mathbf{x}) \phi_1(\mathbf{x}) d\Gamma & \int_{\Gamma_{s_{in}}} \phi_1(\mathbf{x}) \phi_2(\mathbf{x}) d\Gamma & \cdots & \int_{\Gamma_{s_{in}}} \phi_P(\mathbf{x}) \phi_P(\mathbf{x}) d\Gamma\\
        \end{bmatrix}^T,
\end{equation}
and
\begin{equation}
    \mathbf{a}_w^k = 
        \begin{bmatrix}
            {w_1^k}^2 & w_1^k w_2^k & \cdots & {w_P^k}^2
        \end{bmatrix}^T.
\end{equation}
We can now rewrite (\ref{eq:costFunction_heatNorm_secondTerm2}) as
\begin{equation}
        \sum_{l=1}^P \sum_{q=1}^P w_l^k w_q^k \int_{\Gamma_{s_{in}}} \phi_l(\mathbf{x}) \phi_q(\mathbf{x}) d\Gamma = \pmb{\phi}_\phi^T \mathbf{a}_w^k.
        \label{eq:costFunction_heatNorm_secondTerm3}
\end{equation}

Furthermore, deriving (\ref{eq:costFunction_heatNorm_secondTerm3}) with respect to the weights, we obtain 
\begin{equation}
        \frac{\partial \left[ \int_{\Gamma_{s_{in}}} \left( \sum_{l=1}^P w_l^k \phi_l(\mathbf{x}) \right)\left( \sum_{q=1}^P w_q^k \phi_q(\mathbf{x}) \right) d\Gamma \right]}{\partial w_j^k} = \frac{\partial \pmb{\phi}_\phi^T \mathbf{a}_w^k}{\partial w_j^k} = \pmb{\phi}_\phi^T \frac{\partial \mathbf{a}_w^k}{\partial w_j^k}, \quad \text{for } j = 1,2,\dots,P.
    \label{eq:costFunction_heatNorm_secondTerm4}
\end{equation}
Considering the case $j=1$, we have
\begin{equation}
        \pmb{\phi}_\phi^T \frac{\partial \mathbf{a}_w^k}{\partial w_1^k} = \pmb{\phi}_\phi^T 
        \begin{bmatrix}
            2 w_1^k &  w_2^k  &\cdots & w_P^k  & w_2^k  & 0 &\cdots & 0 & w_3^k & 0 &\cdots & 0 & w_P^k  & 0 & \cdots 
        \end{bmatrix}^T,
\end{equation}
that we can rewrite as
\begin{equation}
        \pmb{\phi}_\phi^T \frac{\partial \mathbf{a}_w^k}{\partial w_1^k} =  2 \pmb{\phi}_{\phi_1} w_1^k + \pmb{\phi}_{\phi_2} w_2^k + \pmb{\phi}_{\phi_3} w_3^k +  \cdots + \pmb{\phi}_{\phi_P} w_P^k + \pmb{\phi}_{\phi_{P + 1}} w_2^k + \pmb{\phi}_{\phi_{2P + 1}} w_3^k + \cdots + \pmb{\phi}_{\phi_{(P-1)P + 1}} w_P^k.
\end{equation}
Now, by noticing that
\begin{equation}
    \pmb{\phi}_{\phi_{(r-1)P + s}} = \int_{\Gamma_{s_{in}}} \phi_r(\mathbf{x}) \phi_s(\mathbf{x}) d\Gamma = \pmb{\phi}_{\phi_{(s-1)P + r}},\quad 1 \leq r,s \leq P,
\end{equation}
we obtain
\begin{equation}
        \pmb{\phi}_\phi^T \frac{\partial \mathbf{a}_w^k}{\partial w_1^k} =  2 \pmb{\phi}_{\phi_1} w_1^k + 2 \pmb{\phi}_{\phi_2} w_2^k + 2 \pmb{\phi}_{\phi_3} w_3^k +  \cdots + 2 \pmb{\phi}_{\phi_P} w_P^k = 2 \pmb{\phi}_{\phi_{1:P}}^T \mathbf{w}^k.
\end{equation}
Similarly, if we consider the general case, we have
\begin{equation}
    \begin{aligned}
        \pmb{\phi}_\phi^T \frac{\partial \mathbf{a}_w^k}{\partial w_j^k} =  2 \pmb{\phi}_{\phi_{(j-1)P + 1}} w_1^k + 2 \pmb{\phi}_{\phi_{(j-1)P + 2}} w_2^k + 2 \pmb{\phi}_{\phi_{(j-1)P + 3}} w_3^k +  \cdots + 2 \pmb{\phi}_{\phi_{(j-1)P + P}} w_P^k = 2 \pmb{\phi}_{\phi_{(j-1)P + 1:(j-1)P + P}}^T \mathbf{w}^k, \quad \text{for } j = 1,2,\dots,P.
    \end{aligned}
        \label{eq:costFunction_heatNorm_secondTerm5}
\end{equation}
Therefore, thanks to (\ref{eq:costFunction_heatNorm_secondTerm4}) and (\ref{eq:costFunction_heatNorm_secondTerm5}), we can write
\begin{equation}
        \frac{\partial \left[ \int_{\Gamma_{s_{in}}} \left( \sum_{l=1}^P w_l^k \phi_l(\mathbf{x}) \right)\left( \sum_{q=1}^P w_q^k \phi_q(\mathbf{x}) \right) d\Gamma \right]}{\partial w_j^k} =  \pmb{\phi}_\phi^T \frac{\partial \mathbf{a}_w^k}{\partial w_j^k} = 2 \pmb{\phi}_{\phi_{(j-1)P + 1:(j-1)P + P}}^T \mathbf{w}^k , \quad \text{for } j = 1,2,\dots,P.
    \label{eq:costFunction_heatNorm_secondTerm6}
\end{equation}

Let us define the matrix $\Phi \in \M^{P \times P}$ such that
\begin{equation}
    \Phi_{r,s} := \int_{\Gamma_{s_{in}}} \phi_r(\mathbf{x}) \phi_s(\mathbf{x}) d\Gamma. 
  \label{eq:costFunction_heatNorm_Phi}
\end{equation}
If we now consider the minimization of $S_2^k$ with respect to the weights, $\mathbf{w}^k$, as in (\ref{eq:benchmark3DhcModelUnsteadySequential_criticalPoint}), we have
\begin{equation}
        \frac{\partial S_2^k[\hat{\mathbf{w}}^k]}{\partial w_{j}^k} =  \sum_{i=1}^{P} (R^k[\hat{\mathbf{w}}^k])_{i} \frac{\partial (\mathbf{T}^k[\hat{\mathbf{w}}^k])_{i}}{\partial w_j^k} + 2 p_g \pmb{\phi}_{\phi_{(j-1)P + 1:(j-1)P + P}}^T \mathbf{w}^k =   0,\text{ for } j=1,2,\dots,P.
  \label{eq:benchmark3DhcModelUnsteadySequential_heatNorm_criticalPoint}
\end{equation}

Considering the piecewise constant case, thanks to (\ref{eq:benchmark3DhcModelUnsteadySequential_solAtMeasurements_const}) and (\ref{eq:costFunction_heatNorm_secondTerm4}), we rewrite (\ref{eq:benchmark3DhcModelUnsteadySequential_heatNorm_criticalPoint}) as
\begin{equation}
    \Theta^T(\Theta \hat{\mathbf{w}}^k  + \mathbf{T}^k_{IC} - \hat{\mathbf{T}}^k) + 2 p_g  \Phi \hat{\mathbf{w}}^k = \mathbf{0},
\end{equation}
being $\Theta$ the matrix defined in (\ref{eq:sequentialThetaMatrix_const}).
Therefore, for each $k$, $1 \leq k \leq P_t$, the solution of the inverse problem, $\hat{\mathbf{w}}^k$, is obtained by solving the linear system
\begin{equation}
    \left( \Theta^T \Theta + 2 p_g  \Phi \right) \hat{\mathbf{w}}^k = \Theta^T(\hat{\mathbf{T}}^k - \mathbf{T}^k_{IC}).
  \label{eq:linSys_parametrizedBC_sequential_heatNorm_constant}
\end{equation}

Similarly, for the piecewise linear case, thanks to (\ref{eq:benchmark3DhcModelUnsteadySequential_solAtMeasurements_linear}) and (\ref{eq:costFunction_heatNorm_secondTerm4}),  we rewrite (\ref{eq:benchmark3DhcModelUnsteadySequential_heatNorm_criticalPoint}) as
\begin{equation}
    {(\tilde{\Theta})}^T\left( \tilde{\Theta} \hat{\mathbf{w}}^k -  \Theta_d \hat{\mathbf{w}}^{k-1} + \mathbf{T}^k_{IC}   - \hat{\mathbf{T}}^k \right) + 2 p_g  \Phi \hat{\mathbf{w}}^k  = \mathbf{0},
\end{equation}
being $\tilde{\Theta}$ and $\Theta_d$ the matrices defined in (\ref{eq:sequentialThetaMatrix_linear}).
Therefore, for each $k$, $1 \leq k \leq P_t$, a solution of the inverse problem, $\hat{\mathbf{w}}^k$, is obtained by solving the linear system
\begin{equation}
    \left( \tilde{\Theta}^T \tilde{\Theta} + 2 p_g  \Phi \right) \hat{\mathbf{w}}^k = \tilde{\Theta}^T(\hat{\mathbf{T}}^k + \Theta_d \hat{\mathbf{w}}^{k-1}- \mathbf{T}^k_{IC}).
  \label{eq:linSys_parametrizedBC_sequential_heatNorm_linear}
\end{equation}

The resulting inverse solvers are straightforward modifications of Algorithm~\ref{alg:inverseSolver_constant} and \ref{alg:inverseSolver_linear}.
Then, we show them in the following Algorithm~\ref{alg:inverseSolver_constant_heatNorm} and \ref{alg:inverseSolver_linear_heatNorm}.

\begin{algorithm}[htb!]
 \caption{Inverse solver for the solution of Problem~\ref{prob:sequentialUnsteadyInverseProblem_heatNorm} with piecewise constant parameterization in time of the heat flux, $g$.}
 \label{alg:inverseSolver_constant_heatNorm}
    \hspace*{\algorithmicindent} \textbf{OFFLINE}\\
    \hspace*{\algorithmicindent} \textbf{Input} RBF shape parameter, $\eta$; thermocouples measurement points and times, $\Psi, \Upsilon$; cost functional parameter, $p_g$ 
    \begin{algorithmic}[1]
        \State Setup RBF parameterization by (\ref{eq:GaussianRBF})
        \State Compute $T_{\phi_i}$ for $i = 1,2,\dots,P$ by solving Problem~\ref{prob:sequential_basis}
        \State Assemble matrix $\Theta$ by (\ref{eq:sequentialThetaMatrix_const})
        \State Assemble matrix $\Phi$ by (\ref{eq:costFunction_heatNorm_Phi})
    \end{algorithmic} 
    \hspace*{\algorithmicindent} \\
    \hspace*{\algorithmicindent} \textbf{ONLINE}\\
    \hspace*{\algorithmicindent} \textbf{Input} Initial condition, $T_0$  
    \begin{algorithmic}[1]
        \State Set $k = 1$ 
        \While{\texttt{$k \leq P_t$}}
            \State Read the thermocouples measurements, $\hat{\mathbf{T}}^k$
            \State Compute $T_{IC}^k$ by solving Problem~\ref{prob:Tic}
            \State Assemble $\mathbf{T}_{IC}^k$
            \State Compute $\hat{\mathbf{w}}^k$ by solving (\ref{eq:linSys_parametrizedBC_sequential_heatNorm_constant})
            \State Compute $g_i^k(t)$ by (\ref{eq:heatFluxConstTimeBasis})
            \State Compute the heat flux $\mathscr{g}(\mathbf{x}, t)$ for $t\in(\tau^{k-1}, \tau^k]$ by (\ref{eq:parametrizedSequentialHeatFlux}) and (\ref{eq:heatFluxConstTimeBasis}) 
            \State Use (\ref{eq:directProbDecomposition_sequential_const}) to compute $T^k[\mathscr{g}^k]$
            \State $k = k + 1$
        \EndWhile
    \end{algorithmic}
\end{algorithm}

\begin{algorithm}[htb!]
 \caption{Inverse solver for the solution of Problem~\ref{prob:sequentialUnsteadyInverseProblem_heatNorm} with piecewise linear parameterization in time of the heat flux, $g$.}
 \label{alg:inverseSolver_linear_heatNorm}
    \hspace*{\algorithmicindent} \textbf{OFFLINE}\\
    \hspace*{\algorithmicindent} \textbf{Input} RBF shape parameter, $\eta$; thermocouples measurement points and times, $\Psi, \Upsilon$; cost functional parameter, $p_g$ 
    \begin{algorithmic}[1]
        \State Setup RBF parameterization by (\ref{eq:GaussianRBF})
        \State Compute $T_{\phi_i}$ for $i = 1,2,\dots,P$ by solving Problem~\ref{prob:sequential_basis}
        \State Compute $T_{d_i}$ for $i = 1,2,\dots,P$ by solving Problem~\ref{prob:Td}
        \State Assemble matrices $\tilde{\Theta}$ and $\Theta_d$
        \State Assemble matrix $\Phi$ by (\ref{eq:costFunction_heatNorm_Phi})
    \end{algorithmic} 
    \hspace*{\algorithmicindent} \\
    \hspace*{\algorithmicindent} \textbf{ONLINE}\\
    \hspace*{\algorithmicindent} \textbf{Input} Initial condition, $T_0$
    \begin{algorithmic}[1]
        \State Set $k = 1$ 
        \While{\texttt{$k \leq P_t$}}
            \State Read the thermocouples measurements, $\hat{\mathbf{T}}^k$
            \State Compute $T_{IC}^k$ by solving Problem~\ref{prob:Tic}
            \State Assemble $\mathbf{T}_{IC}^k$
            \State Compute $\hat{\mathbf{w}}^k$ by solving (\ref{eq:linSys_parametrizedBC_sequential_heatNorm_linear})
            \State Compute $g_i^k(t)$ by  (\ref{eq:heatFluxLinearTimeBasis})
            \State Compute the heat flux $\mathscr{g}(\mathbf{x}, t)$ for $t\in(\tau^{k-1}, \tau^k]$ by (\ref{eq:parametrizedSequentialHeatFlux}) and (\ref{eq:heatFluxLinearTimeBasis}) 
            \State Use (\ref{eq:directProbDecomposition_sequential_linear}) to compute $T^k[\mathscr{g}^k]$
            \State $k = k + 1$
        \EndWhile
    \end{algorithmic}
\end{algorithm}

As a final remark, notice that for $p_g = 0 \frac{K^2}{W^2}$, we end up with the same solution as for $S_1^k$.

%************************************************************************************************************************************************************************************%
\subsection{Regularization}
\label{sec:unsteadyRegularization}
%************************************************************************************************************************************************************************************%

After the development of novel inverse solvers, we provide a brief discussion about regularization.
It is well known that inverse problems as the ones here considered are ill-posed.
This means that for our problem at least one of the following properties does not hold: for all admissible data, a solution exists; for all admissible data, the solution is unique; the solution depends continuously on the data~\cite{Engl1996}.
In our discussion, we turned the infinite dimensional inverse Problem \ref{prob:sequentialUnsteadyInverseProblem} into the solution of the discrete linear systems (\ref{eq:linSys_parametrizedBC_sequential_constant}) and (\ref{eq:linSys_parametrizedBC_sequential_linear}) by making some assumptions on the heat flux (i.e. parameterizing it).
In this new setting, if the matrices $\Theta^T \Theta$ and $\tilde{\Theta}^T \tilde{\Theta}$ are invertible, we have the existence of a unique solution for our inverse problem.

As we will see in the numerical tests section, it turns out that these matrices are very ill-conditioned.
This can cause the matrix to be numerically rank deficient, losing the uniqueness of a solution.
However, this is not the only concern.
We still have the problem of a continuous dependence of the solution on the data.
The ill-conditioning of the linear system causes that, if we have some noise in the data vector (as usual in an industrial measurement equipment), the solution of the linear system diverges from the correct value.

To address both these problems, we require regularization.
There are several techniques available for regularizing a discrete ill-posed problem as the present one.
In general, they are divided into direct methods like Truncated Singular Values Decomposition (TSVD) and Tikhonov regularization, and iterative methods such as the conjugate gradient method.
For a deep discussion of all regularization methods, we refer the interested reader to Hansen's monograph on the subject\cite{Hansen2010}.

In the present investigation, we use TSVD.
To briefly describe this regularization technique, we denote the Singular Values Decomposition (SVD) of a matrix $K$ by
\begin{equation}
    K = U \Sigma V^T = \sum_{i=1}^r \mathbf{u}_i \sigma_i \mathbf{v}_i^T,
\end{equation}
where $\sigma_i$ denotes the i-th singular value of $K$ (numbered according to their decreasing value), $r$ denotes the last no null singular value (i.e. the rank of $K$), $\mathbf{u}_i$ and $\mathbf{v}_i$ are the i-th column of  the semi-unitary matrices $U$ and $V$ respectively (both belonging to $\M^{P \times r}$), and $\Sigma$ is the square matrix of $\M^{r \times r}$ such that $\Sigma_{ii} = \sigma_i$ and $\Sigma_{ij} = 0$ if $i\neq j$.
Then, given $\alpha_{TSVD} \leq r$, the TSVD regularized solution of the general linear system $K \mathbf{z} = \mathbf{c}$ is
\begin{equation}
    \mathbf{z} = \sum_{i=1}^{\alpha_{TSVD}} \left(\frac{\mathbf{u}_i^T \mathbf{c} }{\sigma_i}\right) \mathbf{v}_i.
\end{equation}
This solution differs from the least square solution only in that the sum is truncated at $i = \alpha_{TSVD}$ instead of $i=r$.
In this way, we cut off the smallest singular values that are responsible of the errors propagation.
For a detailed discussion on the solution of discrete ill-posed inverse problems, we refer the reader to Hansen's monograph on the subject~\cite{Hansen2010}.

Together with the classical aforementioned regularization method, we investigate also the regularization by discretization\cite{Aster2019, Kirsch2011}.
Using this method, we exploit the regularizing properties of coarsening the time and/or space discretization to improve the heat flux estimation.
In the next section, we will test the performance of these regularization methods also by adding noise to the thermocouples measurements.

%%%%%%%%%%%%%%%%%%%%%%%%%%%%%%%%%%%%%%%%%%%%%%%%%%%%%%%%
\subsection{Discretization Selection Algorithm}
\label{sec:unsteadyDiscretizationSelectionAlgorithm}

To conclude this section, we propose an algorithm for the automated selection of some of the parameters required by Algorithm~\ref{alg:inverseSolver_linear_heatNorm}. 
As will be shown in Section~\ref{section:numerical}, the numerical tests highlight that this algorithm is very sensitive to the mesh and time discretization refinement as well as to the parameter $p_g$.
We anticipate here that this inverse solver shows severe instabilities for fine discretizations.
However, these instabilities are effectively eliminated for values of $p_g$ that are above a threshold that depends on the discretization refinement.
In for these values of $p_g$, we notice a drastic decrease of the dependency of the algorithm from the discretization.

However, the uncontrolled increase in $p_g$ does not lead to a monotonic improvement of the inverse solver performances.
As can be observed in the numerical results of Section~\ref{section:numerical} (see Figures~\ref{fig:unsteadyNumericalBenchmarkInverseLinear_linear_costFunction_differentDt}, \ref{fig:unsteadyNumericalBenchmarkInverseLinear_linear_costFunction_differentMeshes}, \ref{fig:unsteadyNumericalBenchmarkInverseLinear_linear_timeSpaceRefinement_costFunc5e-11}, \ref{fig:unsteadyNumericalBenchmarkInverseNonLinear_linear_costFunction}, \ref{fig:unsteadyNumericalBenchmarkInverseNonLinear_linear_costFunction_meshes}), the dependency of the algorithm from $p_g$ is such that it is unstable for low values of $p_g$ then, increasing further $p_g$, it sharply achieves an optimum of performance before reaching a plateau at which the algorithm is stable but the term $\langle g^k(\tau^k), g^k(\tau^k) \rangle_{L^2(\Gamma_{s_{in}})}$ in (\ref{eq:sequentialUnsteadyInverseProblem_heatNorm_functional}) overcomes the measurements distance one, $S_1^k$ defined in (\ref{eq:unsteadyInverseFunctional1}).
Thus, for too high values of $p_g$, we have a stable algorithm that is almost independent from the discretization refinement but that provides poor heat flux estimations.

To allow an industrial use of the proposed  inverse solver, the objective of this section is to develop a method for automatically selecting the $\Delta t$, the mesh and the value of $p_g$ such that the algorithm is stable and accurately estimates the mold-steel heat flux.

In developing such method, we assume to have available a reliable dataset of thermocouples measurements, $\hat{T}_{train}(\Psi, \Upsilon_{train})$, that we can use to perform this tuning offline.
Moreover, we assume that, independently from the mold physical parameters and the heat flux values, this inverse solver always shows the previously described behavior with respect to $p_g$.
In particular, we assume that, for values of $p_g$ higher than a problem specific threshold, the algorithm is stable for all the discretizations and independent from them (i.e. we obtain similar solutions for any given mesh and $\Delta t$).   

We recall, that in the real industrial case, we do not have any information about the true heat flux that we want to estimate.
Thus, this selection methodology cannot be based on the heat flux estimation error.
However, Figures~\ref{fig:unsteadyNumericalBenchmarkInverseLinear_linear_costFunction_measDiscrepancy} and \ref{fig:unsteadyNumericalBenchmarkInverseNonLinear_linear_costFunction_measDiscrepancy} show that the measurement discrepancy functional $S_1^k$ and the heat flux estimation error have a similar behavior as functions of $p_g$ and we will use this quantity to determine the quality of the heat flux estimation.

All that said, we begin by selecting an ordered set of meshes $(\Delta x_1,$ $\Delta x_2,$ $\dots,$  $\Delta x_{n_M})$ and an ordered set of timestep sizes $(\Delta t_1,$$  \Delta t_2, $$ \dots, $$ \Delta t_{n_t})$.
We order them from the finest to the coarsest discretization (i.e. $\Delta x_1 <$  $\Delta x_2 < $ $\dots < $ $\Delta x_{n_M}$ and $\Delta t_1 <$  $\Delta t_2 < $ $\dots < $ $\Delta t_{n_t}$).
Then, our first objective is to identity a $p_g^0$ within the aforementioned stability region.

To do it, we start with a tentative $p_g^0$.
For this value of the parameter, we solve the inverse problem on the training measurement dataset for all $\Delta x$ and $\Delta t$.
Let us denote by $T[\Delta x, \Delta t]$ the corresponding solution.
Having done so, we compute 
\begin{equation}
    \Delta T := \max_{i,j,q,p}{\norm{T[\Delta x_i, \Delta t_j] - T[\Delta x_q, \Delta t_p]}_{L^\infty\left((0, t_f]; L^2(\Omega_s)\right)}}.
\end{equation}
If we have
\begin{equation}
    \Delta T > \Delta x_{n_M} + \Delta t_{n_t},
\end{equation}
we consider that the solution is too dependent on the discretization refinement.
Then, we increase the value of $p_g^0$ and redo the calculations until 
\begin{equation}
    \Delta T \leq \Delta x_{n_M} + \Delta t_{n_t},
    \label{eq:discrSelection_stabilityCondition}
\end{equation}
is satisfied.

Once we find a value of $p_g^0$ within the stability region, we choose the discretization setup ($\Delta x^1$, $\Delta t^1$) that corresponds to the minimum for $p_g^0$ of
\begin{equation}
    m_{S}[\Delta x, \Delta t, p_g] := mean_k \left( S_1^k[\Delta x, \Delta t, p_g] \right).
    \label{eq:discrSelection_mesurementsDiscrepancy}
\end{equation}
Once we select $\Delta x^1$ and $\Delta t^1$, we choose $p_g^1$  as the value of $p_g$ that minimizes $m_{S}[\Delta x^1, \Delta t^1, p_g]$.
Then, we fix $p_g = p_g^1$ and we solve again the inverse problem for all the considered discretization setups.
If the previously selected discretization is the one that corresponds to the lowest value of $m_{S}[\Delta x, \Delta t, p_g^1]$, we choose $\Delta x^1$, $\Delta t^1$, and $p_g = p_g^1$, and we stop the process.
Otherwise, we continue iterating by selecting $\Delta x^2$ and $\Delta t^2$ as the ones corresponding to the smallest $m_{S}[\Delta x, \Delta t, p_g^1]$ and looking for the $p_g^2$ that minimizes $m_{S}[\Delta x^2, \Delta t^2, p_g]$, and so on.
We summarize all this process in Algorithm~\ref{alg:meshSelection}.

\begin{algorithm}[!htb]
    \caption{Offline selection of the mesh, the timestep size and $p_g$ for the inverse solver in Algorithm~\ref{alg:inverseSolver_linear_heatNorm}.}
    \label{alg:meshSelection}
    \hspace*{\algorithmicindent} \textbf{Input} Ordered set of meshes, $(\Delta x_1, $$ \Delta x_2,$$  \dots,$$  \Delta x_{n_M})$; ordered set of timestep sizes, $(\Delta t_1,$$  \Delta t_2, $$ \dots, $$ \Delta t_{n_t})$; $p_g^0$; training set, $\hat{T}_{train}(\Psi, \Upsilon_{train})$
    \begin{algorithmic}[1]
        \While{$\Delta T > \Delta x_{n_M} + \Delta t_{n_t}$}\Comment{Identify stability region}
            \For{$i = 1$ to $n_M$}
                \For{$j = 1$ to $n_t$}
                    \State Solve the inverse problem on the training set, $\hat{T}_{train}$, by using Algorithm~\ref{alg:inverseSolver_linear_heatNorm} with $\Delta x_i$, $\Delta t_j$ and $p_g = p_g^0$
                    \State Compute $m_{S}[\Delta x_i, \Delta t_j, p_g^0]$ by (\ref{eq:discrSelection_mesurementsDiscrepancy})
                \EndFor
            \EndFor
            \State Compute $\Delta T$ by (\ref{eq:discrSelection_stabilityCondition})
            \If{$\Delta T > \Delta x_{n_M} + \Delta t_{n_t}$}
                \State $p_g^0 = 10 p_g^0$
            \EndIf
        \EndWhile
        \State Choose $\Delta x^0$ and $\Delta t^0$ corresponding to $\min_{i,j}m_{S}[\Delta x_i, \Delta t_j, p_g^0]$ for $1 \leq i \leq n_M$, $1 \leq j \leq n_t$ 
        \State $l = 1$, $f = 0$
        \While{$f = 0$}
            \State Find $p_g^l = \argmin_{p_g}(m_{S}[\Delta x^{l-1}, \Delta t^{l-1}, p_g])$ \label{alg:meshSelection_minimization}
            \For{$i = 1$ to $n_M$}
                \For{$j = 1$ to $n_t$}
                    \State Solve the inverse problem on the training set, $\hat{T}_{train}$, by using Algorithm~\ref{alg:inverseSolver_linear_heatNorm} with $\Delta x_i$, $\Delta t_j$ and $p_g = p_g^l$
                    \State Compute $m_{S}[\Delta x_i, \Delta t_j, p_g^l]$ by (\ref{eq:discrSelection_mesurementsDiscrepancy})
                \EndFor
            \EndFor
            \State Choose $\Delta x^l$ and $\Delta t^l$ corresponding to $\min_{i,j}m_{S}[\Delta x_i, \Delta t_j, p_g^l]$ for $1 \leq i \leq n_M$, $1 \leq j \leq n_t$\label{alg:meshSelection_chooseMeshDt}
            \If{$\Delta x^l = \Delta x^{l-1} \And \Delta t^l = \Delta t^{l-1}$}
                \State $f = 1$
            \EndIf
            \State $l = l +1$
        \EndWhile
    \end{algorithmic}
\end{algorithm}

This method allows a data-driven, automated selection of the discretization refinement and the $p_g$ parameter.
This result comes to the cost of computing $n_M \cdot n_t$ solutions to the inverse problem at each iteration.
If the available memory allows it, we can keep in the memory the results of the offline computations related to each discretization.
Otherwise, we have the recompute every time these offline phases.
However, we designed this algorithm to be used offline.
Then, even if it is computationally expensive, we can run it before the caster starts to work and it only requires the dataset of thermocouples measurements $\hat{T}_{train}$.

%%%%%%%%%%%%%%%%%%%%%%%%%%%%%%%%%%%%%%%%%%%%%%%%%%%%%%%%%%%%%%%%%%%%%%%%%%%%%%%%%
\section{Numerical Tests}
\label{section:numerical}
%%%%%%%%%%%%%%%%%%%%%%%%%%%%%%%%%%%%%%%%%%%%%%%%%%%%%%%%%%%%%%%%%%%%%%%%%%%%%%%%%

To test the previously developed methodologies, we design different benchmark cases.
Through these tests, we validate and analyze the performances of the inverse solvers that we proposed in the previous sections.
We design two benchmarks to perform different tests for the inverse problem solvers proposed in Section~\ref{section:inverse}.

Notice that all the computations are performed in ITHACA-FV\cite{Stabile2017,ithaca} which is a C++ library based on OpenFOAM\cite{Moukalled2015} developed at SISSA Mathlab.

%%%%%%%%%%%%%%%%%%%%%%%%%%%%%%%%%%%%%%%%%%%%%%%%%%%%%%%%%%%%%%%%%%%%%%%%%%%%%%%%%
\subsection{Benchmark 1}\label{sec:unsteadyInverseBenchmark_linear}

In this section, we test the performances of the inverse solvers proposed in  Section~\ref{section:inverse} in the reconstruction of a linear in time heat flux, which is non-linear in space.

\subsubsection{Setup of the Test Case}

To design a numerical test case for the inverse problems, we proceed as follows: we arbitrarily define a boundary heat flux, $g_{tr}(\mathbf{x},t)$, and the thermocouples positions, $\Psi$, and sampling frequency, $f_{samp}$.
Then, we solve the direct Problem~\ref{prob:directUnsteady} associated with $g_{tr}(\mathbf{x},t)$ in the time domain $(0, t_f]$, obtaining the related temperature field.
Finally, we use its values at the thermocouples points and sampling times as input measurements to the inverse problem, $\hat{\mathbf{T}}$. 
Using this approach, we are able to analyze the inverse problem performance in the reconstruction of the boundary heat flux, $g_{tr}(\mathbf{x},t)$.

Table~\ref{tab:unsteadyNumericalBenchmark_parameters} shows the geometrical and physical parameters selected for the present benchmark case.
In the attempt of mimicking the real industrial situation of estimating the boundary heat flux in a plate of a CC mold, these parameters are close to real industrial values.
We use the computational domain in Figure~\ref{fig:unsteadyAnalyticalBenchmarkDomainSchematic}(a) where $L$, $W$ and $H$ are set as in a real mold plate.
Finally, Figure~\ref{fig:unsteadyAnalyticalBenchmarkDomainSchematic}(b) shows the thermocouple locations.

\begin{figure}[!htb]
        \centering
        \begin{subfigure}[c]{.5\linewidth}
            \captionsetup{width=.85\linewidth}
            \includegraphics[width=.95\textwidth]{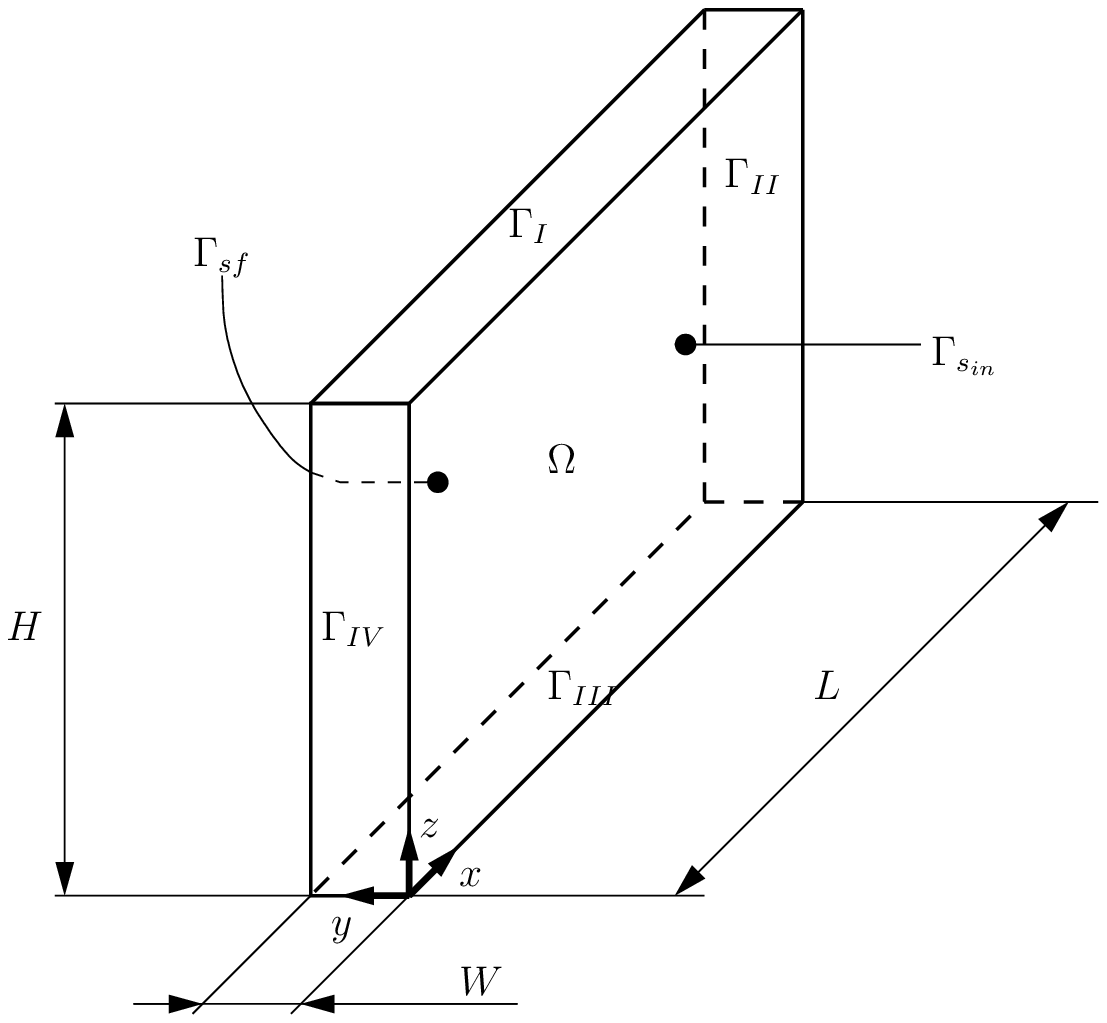}
            \caption{Domain}
        \end{subfigure}%
        \begin{subfigure}[c]{.5\linewidth}
            \captionsetup{width=.85\linewidth}
            \includegraphics[width=.95\textwidth]{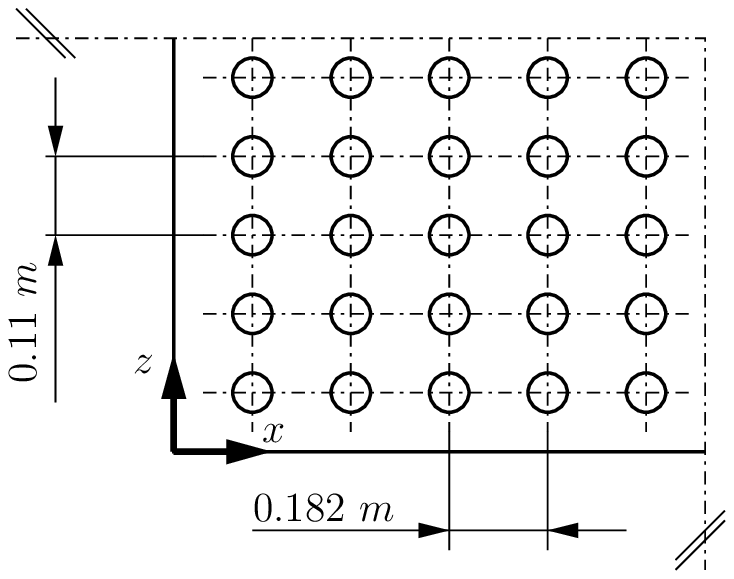}
            \caption{Thermocouples positions}
        \end{subfigure}%
    \caption{Schematic of the domain used in the benchmark test cases (a) and position of the 100 thermocouples at the plane $y=0.02~m$ (b) used for the inverse solver tests (images taken from Morelli et al.\cite{Morelli2021}).}
    \label{fig:unsteadyAnalyticalBenchmarkDomainSchematic}
\end{figure}

\begin{table}
\centering
    \caption{Geometrical and physical parameters used for the benchmark test cases.}
    \label{tab:unsteadyNumericalBenchmark_parameters}
    \begin{tabular}{ ll }
        \hline
        \textbf{Parameter}    &   \textbf{Value}\\
        \hline
        Thermal conductivity, $k_s$     & $383$ W/(m K) \\
        Density, $\rho_s$ 	        & $8940$ kg/m\textsuperscript{3}\\
        Specific heat capacity, $C_{p_s}$ 	& $390$ J/(kg\ K)\\
        Heat transfer coefficient, $h$  & $5.66e4$ W/(m\textsuperscript{2} $\cdot$ K)\\
        Water temperature, $T_f$        & $350$ K \\
        Initial condition, $T_0$        & $350$ K\\
        $L$                             & $2$ m\\
        $W$                             & $0.1$ m\\
        $H$                             & $1.2$ m\\
        Sampling frequency, $f_{samp}$  & $1$ Hz\\
        $a$	                        & $1100$ W/(m\textsuperscript{2}$\cdot$ s) \\
        $b$	                        & $1200$ W/(m\textsuperscript{2}$\cdot$ s) \\
        $c$ 	                        & $3000$ W/(m\textsuperscript{2}$\cdot$ s) \\
        Final time, $t_f$               & $50$ s\\
        \hline
    \end{tabular}
\end{table}

To test the effect of the regularization by discretization, we use different space and time discretizations.
For the time discretization, we use homogeneous time discretization with $\Delta t = 0.1s$, $0.2s$, $0.25s$ and $0.5s$.
For the space discretization, we use the uniform, structured, orthogonal, hexahedral meshes presented in Table~\ref{tab:unsteadyNumericalBenchmark_meshes}.

\begin{table}[htb]
\centering
    \caption{Summary of the different meshes used in the numerical tests.}
    \label{tab:unsteadyNumericalBenchmark_meshes}
    \begin{tabular}{ |l|c|c|c|c|c| }
        \hline
         & \textbf{Mesh 1}    &  \textbf{Mesh 2} &  \textbf{Mesh 3} &  \textbf{Mesh 4} &  \textbf{Mesh 5} \\
        \hline
        Number of elements &  $1.7e5$   & $4.5e4$   & $2.1e4$   & $7.5e3$  & $1.5e3$ \\ 
        \hline
    \end{tabular}
\end{table}

For this test case, we select the heat flux $g_{tr}$ to be linear in time and quadratic in space.
In particular, given $g_1(\mathbf{x}) = bz^2 + c$ with $b$ and $c$ as in Table~\ref{tab:unsteadyNumericalBenchmark_parameters}, we select the heat flux
\begin{equation}
    g_{tr}(\mathbf{x},t) = - k_s \left(0.5 t g_1(\mathbf{x}) + g_1(\mathbf{x}) \right).
\end{equation}

Moreover, to analyze the performance of the inverse solvers, we introduce the relative error
\begin{equation}
    e_{rel}(\mathbf{x},t) := \frac{g_{tr}(\mathbf{x},t) - g_c(\mathbf{x},t)}{g_{tr}(\mathbf{x},t)},
    \label{eq:unsteadyBenchmark_relativeError}
\end{equation}
where $g_c$ is the heat flux computed with the different methodologies described in Section~\ref{section:inverse} that will be tested in the following.

\subsubsection{Effect of Time and Space Discretization Refinement}
\label{sec:benchmark_spaceTimeRef1}

Now, provided all the details for the first benchmark setup, we can proceed presenting the results.
Firstly, we show the effects of mesh and time discretization refinement.
To do it, we do not add any noise to the temperature measurements and do not apply any regularization in the solution of the linear systems solving them by a LU factorization with full pivoting.

We start by analyzing the case in which we minimize the functional $S_1^k$ (i.e. $p_g = 0$).
We show in Figure~\ref{fig:unsteadyNumericalBenchmarkInverseLinear_constant_timeSpaceRefinement} the maximum and mean value of the $L^2$- and $L^\infty$-norm of the relative error, $e_{rel}$, in the interval $(0, t_f]$, as the time and space discretization changes for Algorithm~\ref{alg:inverseSolver_constant} (i.e. piecewise constant approximation in time of the heat flux).

\begin{figure}[!htb]
    \begin{subfigure}{\linewidth}
        \centering
        \begin{subfigure}[c]{.4\linewidth}
            \captionsetup{width=.85\linewidth}
            \includegraphics[width=.95\textwidth]{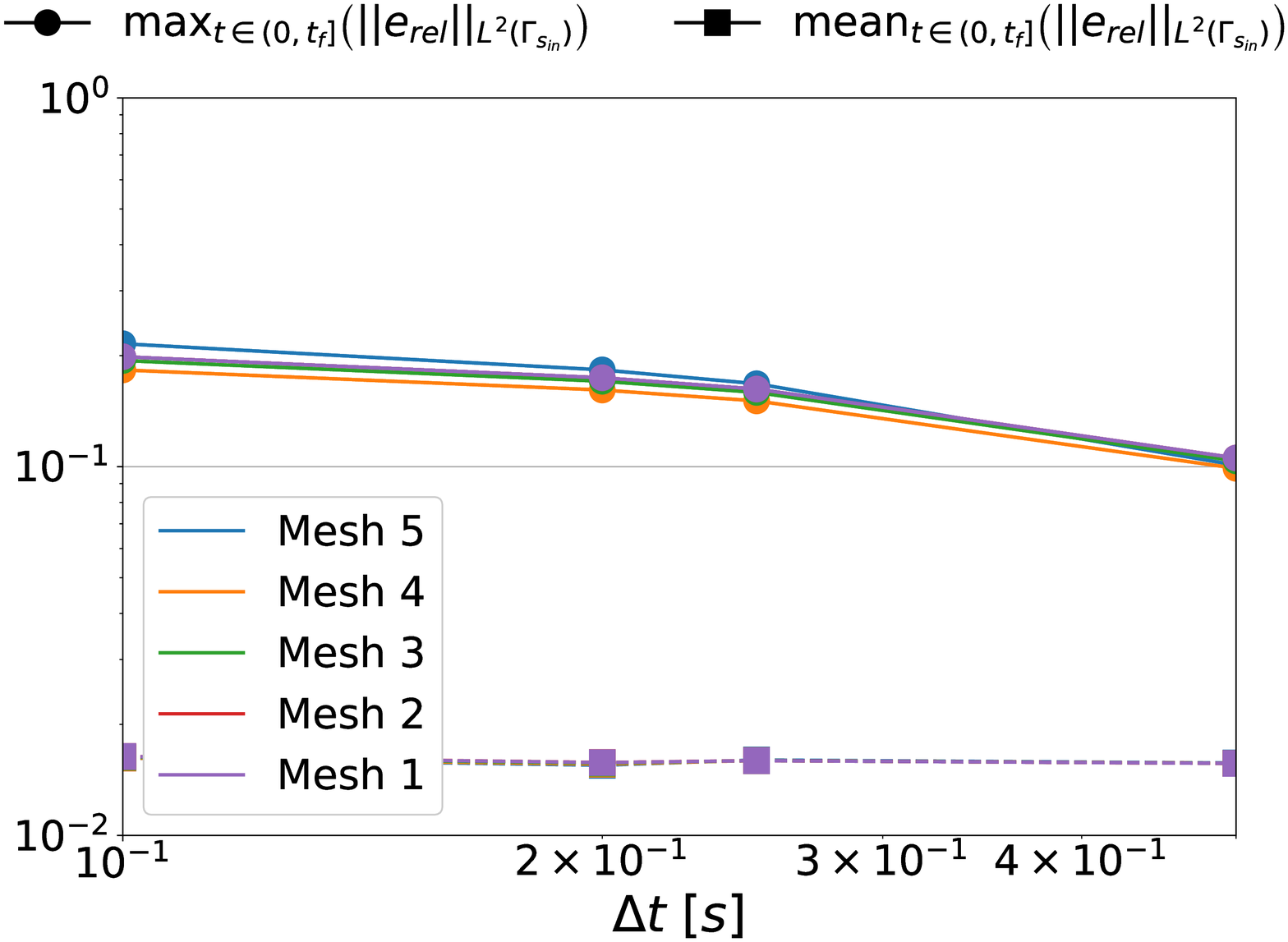}
            \caption{$L^2$-norm of the relative error as a function of $\Delta t$}
        \end{subfigure}%
        \begin{subfigure}[c]{.4\linewidth}
            \captionsetup{width=.85\linewidth}
            \includegraphics[width=.95\textwidth]{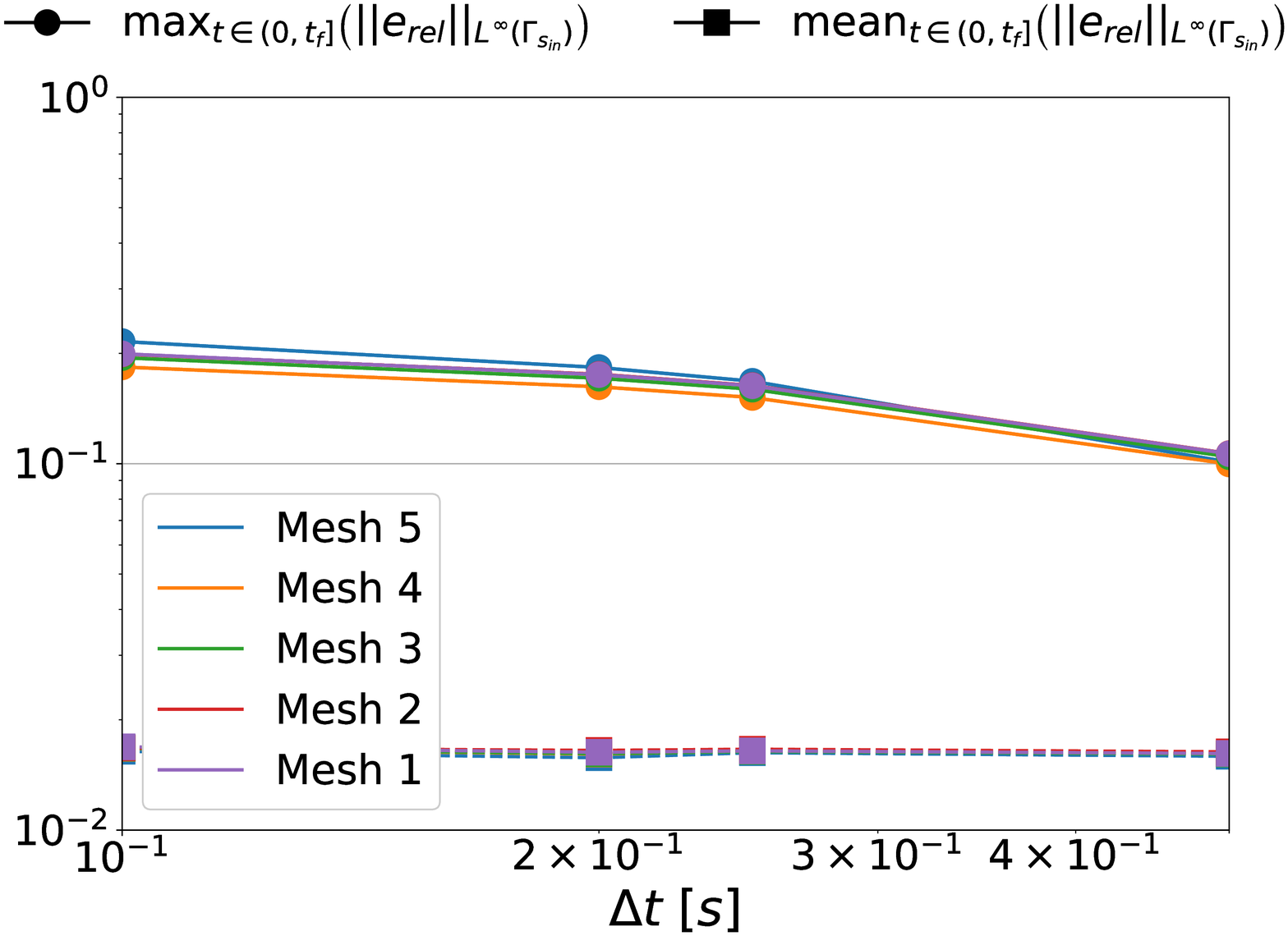}
            \caption{$L^\infty$-norm of the relative error as a function of $\Delta t$}
        \end{subfigure}%
    \end{subfigure}%
    \\
    \begin{subfigure}{\linewidth}
        \centering
        \begin{subfigure}[c]{.4\linewidth}
            \captionsetup{width=.85\linewidth}
            \includegraphics[width=.95\textwidth]{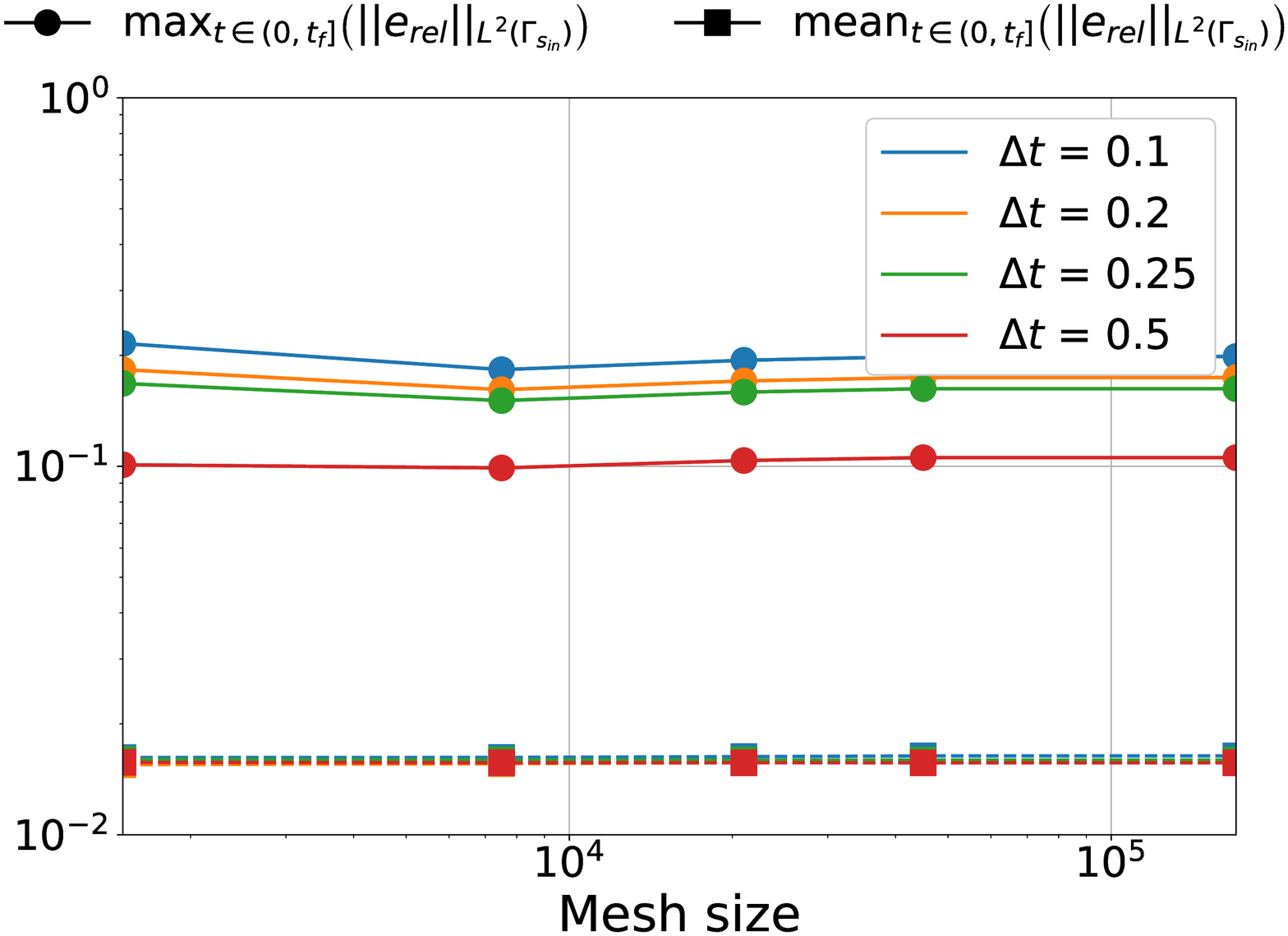}
            \caption{$L^2$-norm of the relative error as a function of the mesh size}
        \end{subfigure}%
        \begin{subfigure}[c]{.4\linewidth}
            \captionsetup{width=.85\linewidth}
            \includegraphics[width=.95\textwidth]{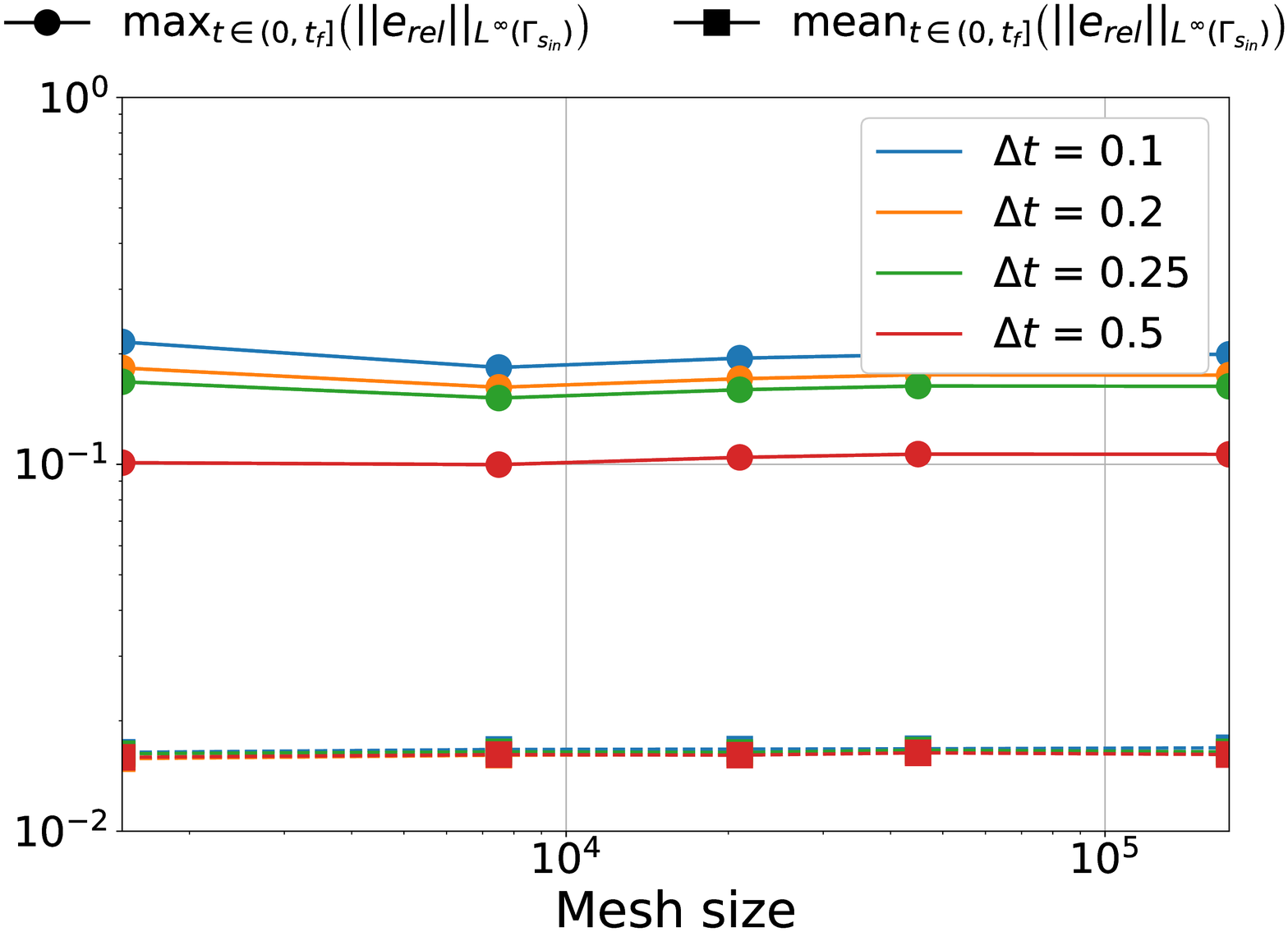}
            \caption{$L^\infty$-norm of the relative error as a function of the mesh size}
        \end{subfigure}%
    \end{subfigure}%
    \caption{Maximum (circles) and mean (squares) values of the $L^2$- and $L^\infty$-norm of the relative error, $e_{rel}$, in the interval $(0, t_f]$, for Benchmark 1 as the time and space discretization changes for Algorithm~\ref{alg:inverseSolver_constant} (piecewise constant time approximation of the heat flux and $p_g = 0 \frac{K^2}{W^2}$).}
    \label{fig:unsteadyNumericalBenchmarkInverseLinear_constant_timeSpaceRefinement}
\end{figure}

From the figures, we appreciate on one side that the time discretization coarsening has very little effects on Algorithm~\ref{alg:inverseSolver_constant} with a small decrease of the error as $\Delta t$ increases.
On the other, the space discretization does not have any effect on this inverse solver.

We now perform the same test for the piecewise linear time approximation of Algorithm~\ref{alg:inverseSolver_linear}.
Similarly, Figure~\ref{fig:unsteadyNumericalBenchmarkInverseLinear_linear_timeSpaceRefinement} shows the maximum and mean value of the $L^2$- and $L^\infty$-norm of the relative error, $e_{rel}$, in the interval $(0, t_f]$, as the time and space discretization changes.

\begin{figure}[!htb]
    \begin{subfigure}{\linewidth}
        \centering
        \begin{subfigure}[c]{.4\linewidth}
            \captionsetup{width=.85\linewidth}
            \includegraphics[width=.95\textwidth]{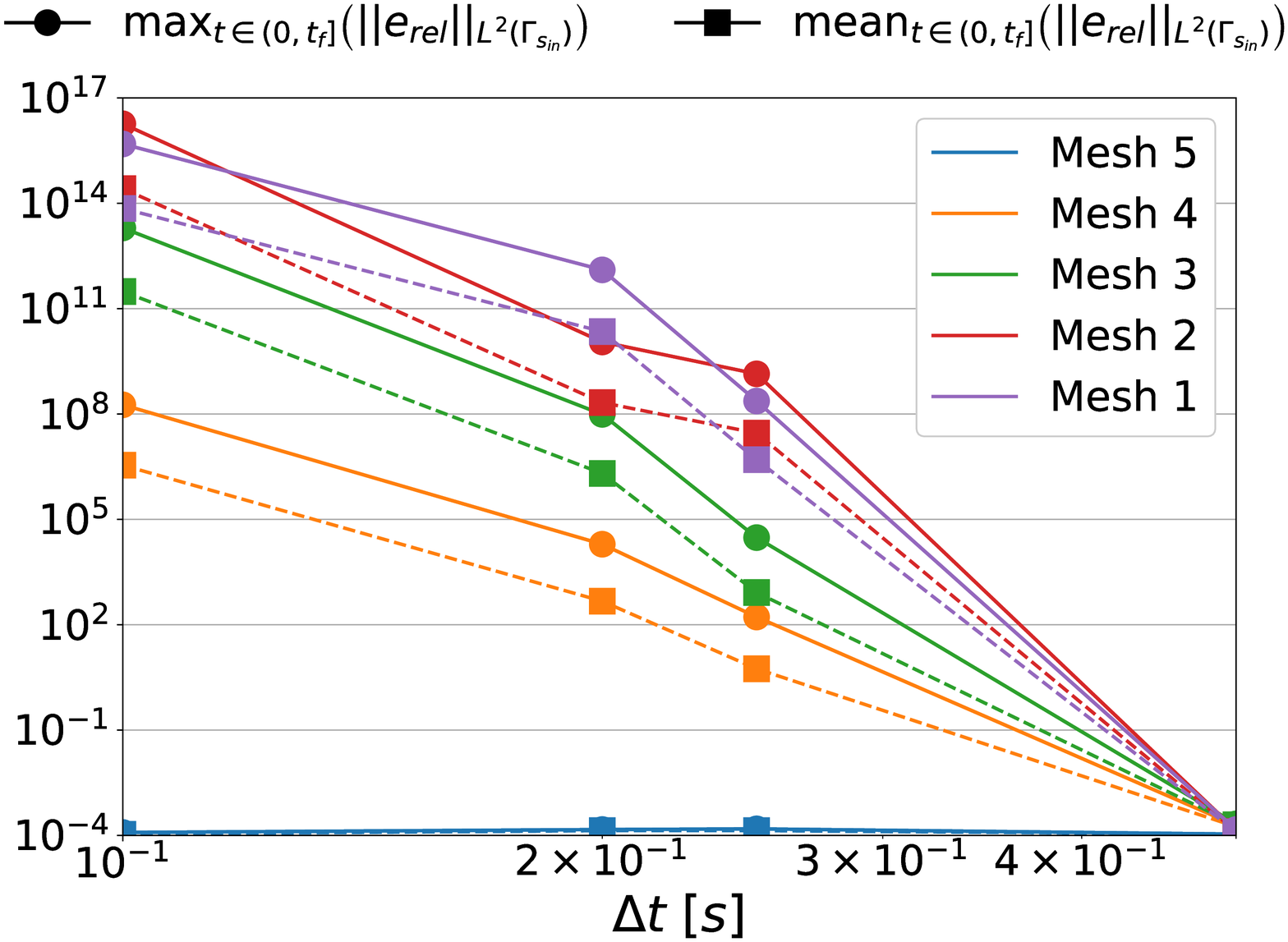}
            \caption{$L^2$-norm of the relative error as a function of $\Delta t$}
        \end{subfigure}%
        \begin{subfigure}[c]{.4\linewidth}
            \captionsetup{width=.85\linewidth}
            \includegraphics[width=.95\textwidth]{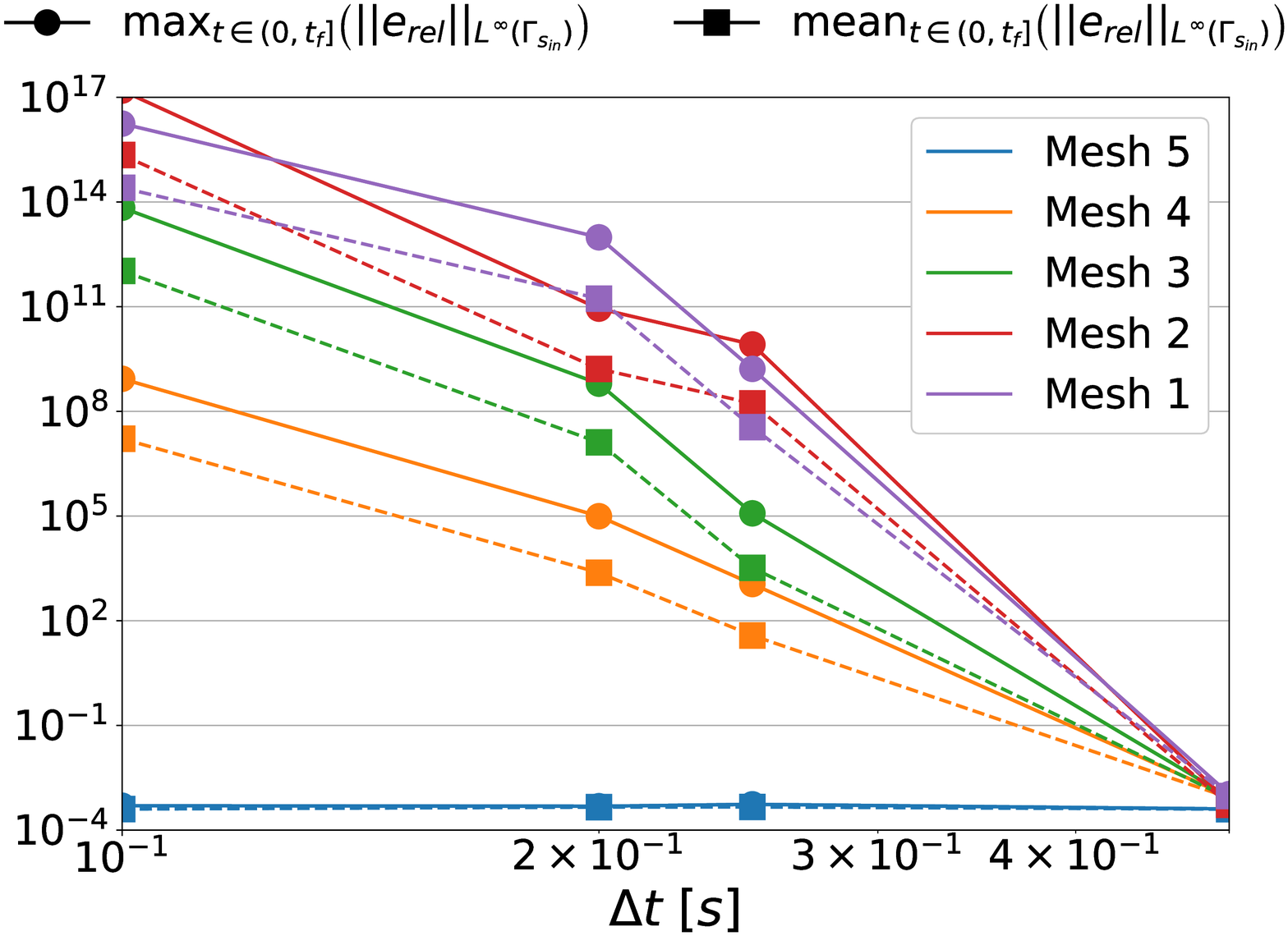}
            \caption{$L^\infty$-norm of the relative error as a function of $\Delta t$}
        \end{subfigure}%
    \end{subfigure}%
    \\
    \begin{subfigure}{\linewidth}
        \centering
        \begin{subfigure}[c]{.4\linewidth}
            \captionsetup{width=.85\linewidth}
            \includegraphics[width=.95\textwidth]{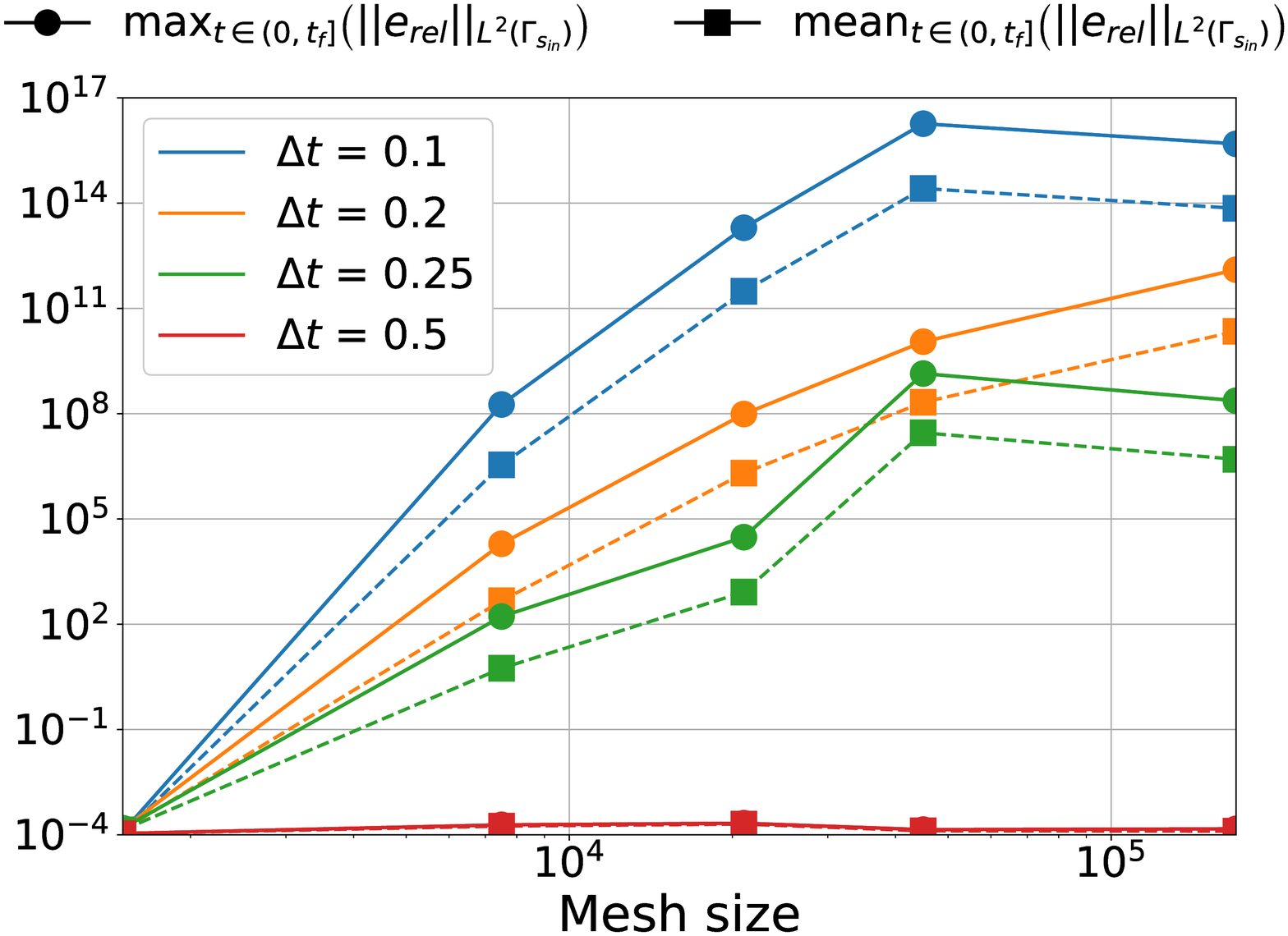}
            \caption{$L^2$-norm of the relative error as a function of the mesh size}
        \end{subfigure}%
        \begin{subfigure}[c]{.4\linewidth}
            \includegraphics[width=.95\textwidth]{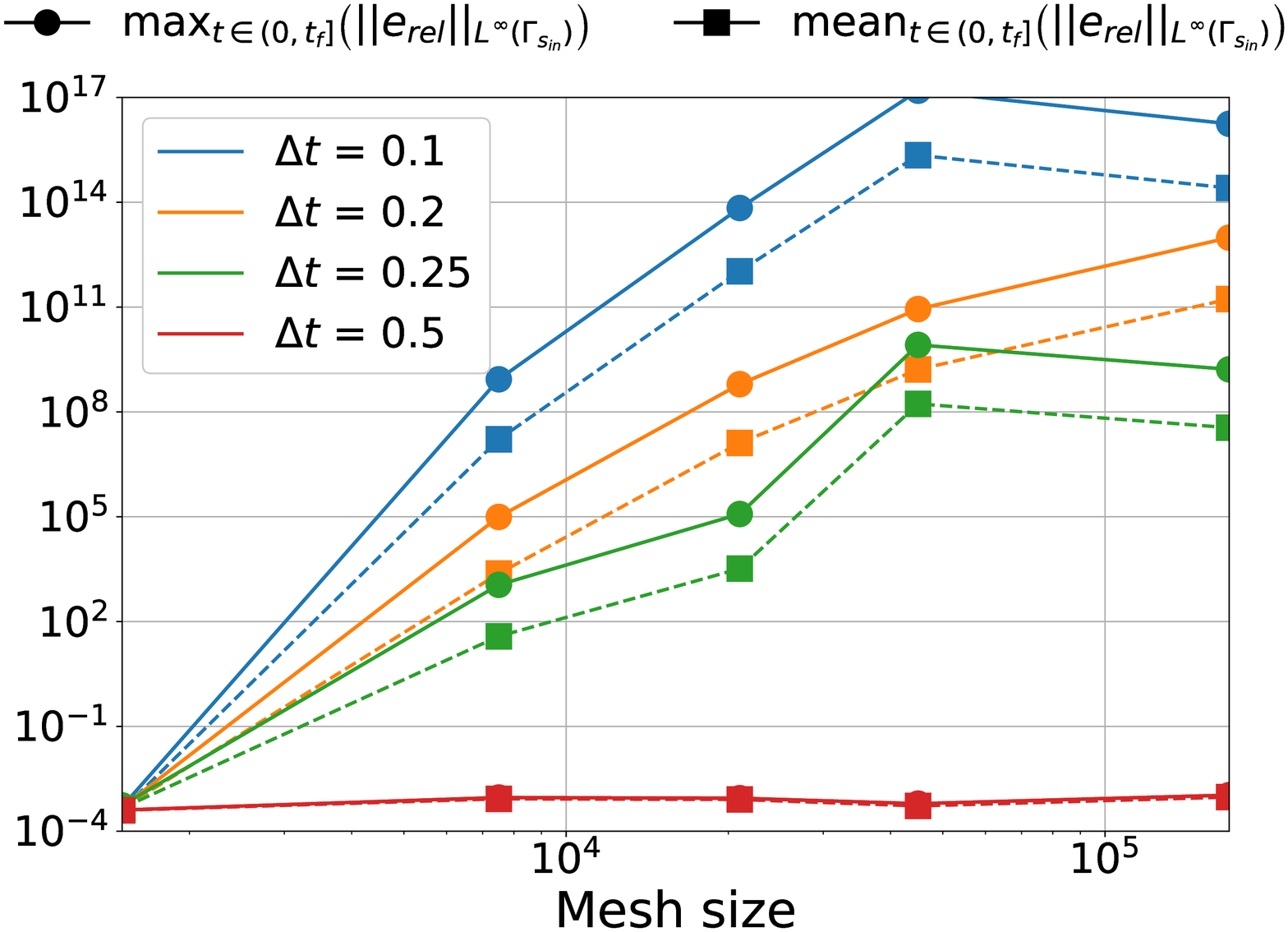}
            \captionsetup{width=.85\linewidth}
            \caption{$L^\infty$-norm of the relative error as a function of the mesh size}
        \end{subfigure}%
    \end{subfigure}%
    \caption{Maximum (circles) and mean (squares) values of the $L^2$- and $L^\infty$-norm of the relative error, $e_{rel}$, in the interval $(0, t_f]$, for Benchmark 1 as the time and space discretization changes for Algorithm~\ref{alg:inverseSolver_linear} (piecewise linear time approximation of the heat flux and $p_g = 0 \frac{K^2}{W^2}$).}
    \label{fig:unsteadyNumericalBenchmarkInverseLinear_linear_timeSpaceRefinement}
\end{figure}

In this setting, the obtained results are very different from the previous case.
First of all, we notice a massive influence of both the space and time discretization refinement on the performances of the inverse solver.
As anticipated in Section~\ref{sec:unsteadyRegularization}, the regularization by discretization plays an important role as the algorithm performances are improved by several orders of magnitude by the coarsening of the discretization.
Moreover, when comparing the results for Algorithm~\ref{alg:inverseSolver_constant} and \ref{alg:inverseSolver_linear}, we notice that the piecewise linear solver is able to outperform the constant one by three orders of magnitude but is also very unstable depending on the discretization.

To better understand the behavior of this inverse solver, Figure~\ref{fig:unsteadyNumericalBenchmarkInverseLinear_linear_errorForDifferentDeltat} illustrates the $L^2$-norm of the relative error, $e_{rel}$, as a function of time for mesh 3 with different $\Delta t$.
From these results, we see that the high errors shown in Figure~\ref{fig:unsteadyNumericalBenchmarkInverseLinear_linear_timeSpaceRefinement} are caused by diverging oscillations in the algorithm.
However, we also notice from the figure that, coarsening the time discretization, monotonically reduces such instability until achieving a stable solution, eventually.

\begin{figure}[!htb]
    \centering
    \includegraphics[width=.4\textwidth]{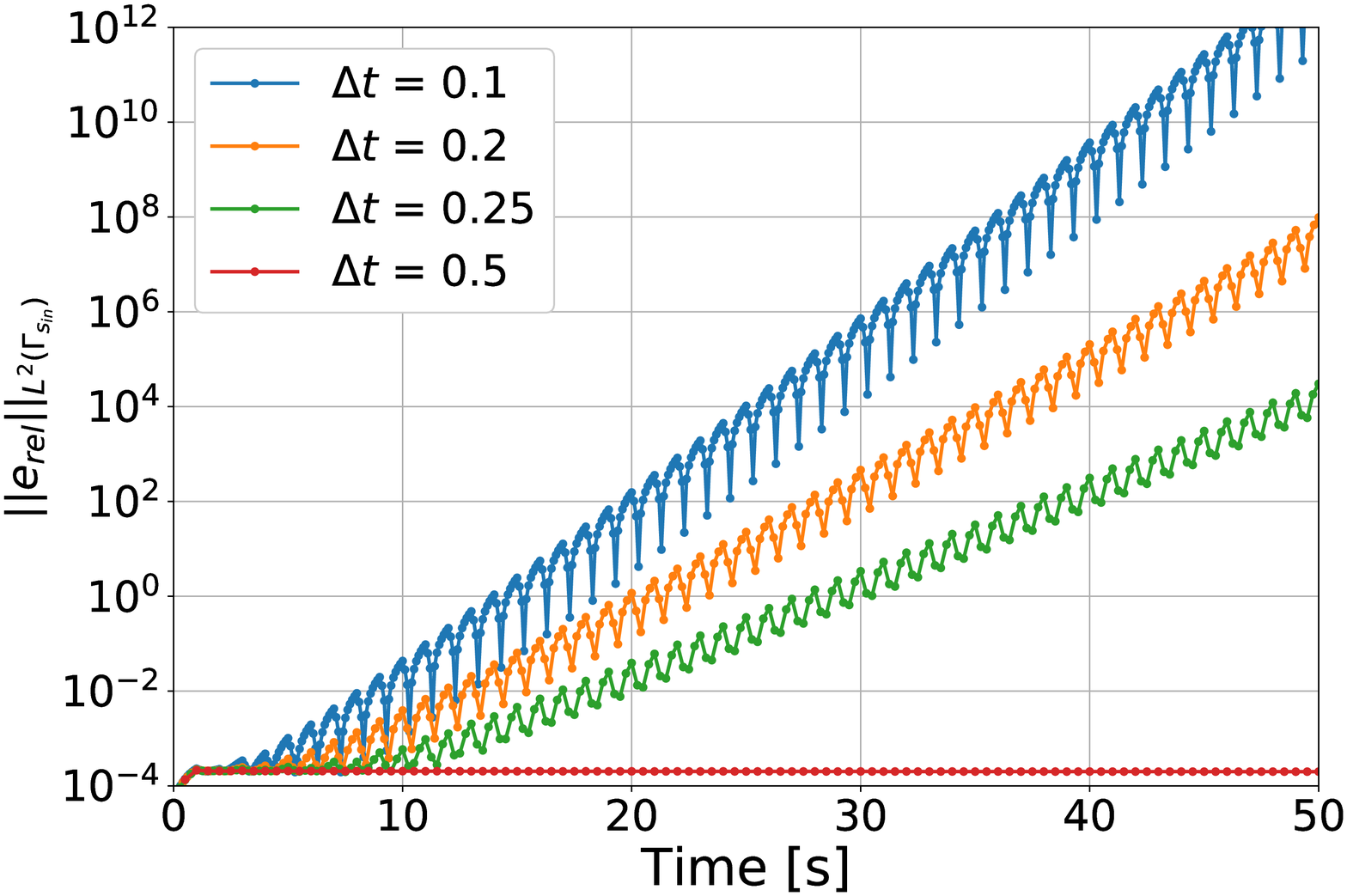}
    \caption{$L^2$-norm of the relative error, $e_{rel}$, in Benchmark 1 for Algorithm~\ref{alg:inverseSolver_linear} (piecewise linear time approximation of the heat flux and $p_g = 0 \frac{K^2}{W^2}$).
    The presented results are obtained with Mesh 3.}
    \label{fig:unsteadyNumericalBenchmarkInverseLinear_linear_errorForDifferentDeltat}
\end{figure}

\subsubsection{Effect of Cost Functional Parameter, $p_g$}
\label{sec:unsteadyBenchmarkLinear_pgEffects}

In this section, we analyze the role that the cost functional parameter, $p_g$, in (\ref{eq:sequentialUnsteadyInverseProblem_heatNorm_functional}), has on the performance of the proposed inverse solvers.
To do it, we solve several times this benchmark case using the different meshes of Table~\ref{tab:unsteadyNumericalBenchmark_meshes} and different timestep sizes.
Then, we plot the maximum and mean value of the $L^2$-norm of the relative error, $e_{rel}$, over the entire interval $t=(0, t_f]$ as a function of the cost functional parameter, $p_g$.

We start with Algorithm~\ref{alg:inverseSolver_constant_heatNorm} (i.e. piecewise constant approximation in time of the heat flux).
Figure~\ref{fig:unsteadyNumericalBenchmarkInverseconstant_constant_costFunction_differentDt} shows the obtained results for different timestep sizes and a fixed space discretization.

\begin{figure}[!htb]
    \begin{subfigure}{\linewidth}
        \centering
        \begin{subfigure}[c]{.4\linewidth}
            \captionsetup{width=.85\linewidth}
            \includegraphics[width=.95\textwidth]{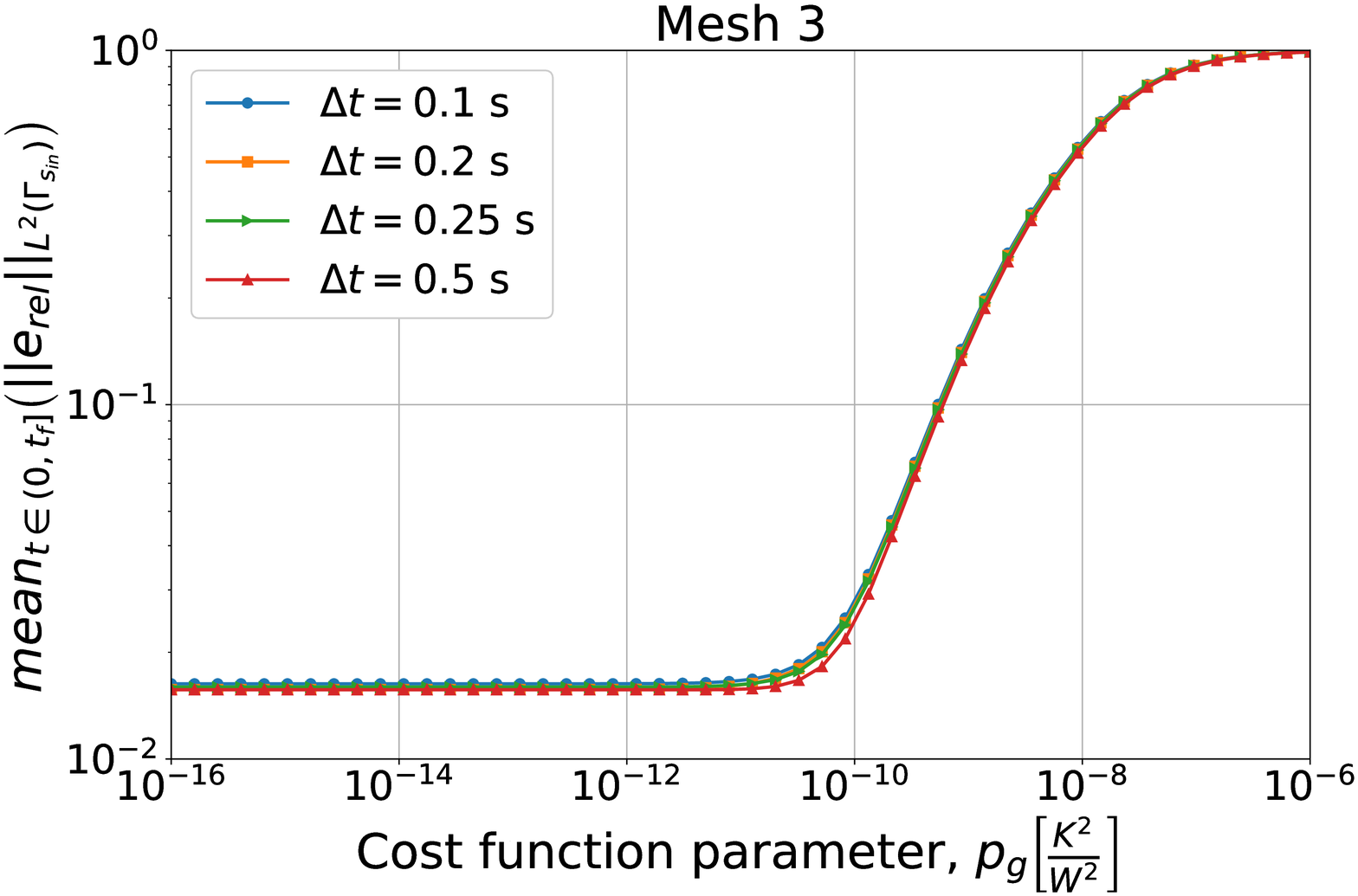}
            \caption{Mean relative error norm.}
        \end{subfigure}%
        \begin{subfigure}[c]{.4\linewidth}
            \captionsetup{width=.85\linewidth}
            \includegraphics[width=.95\textwidth]{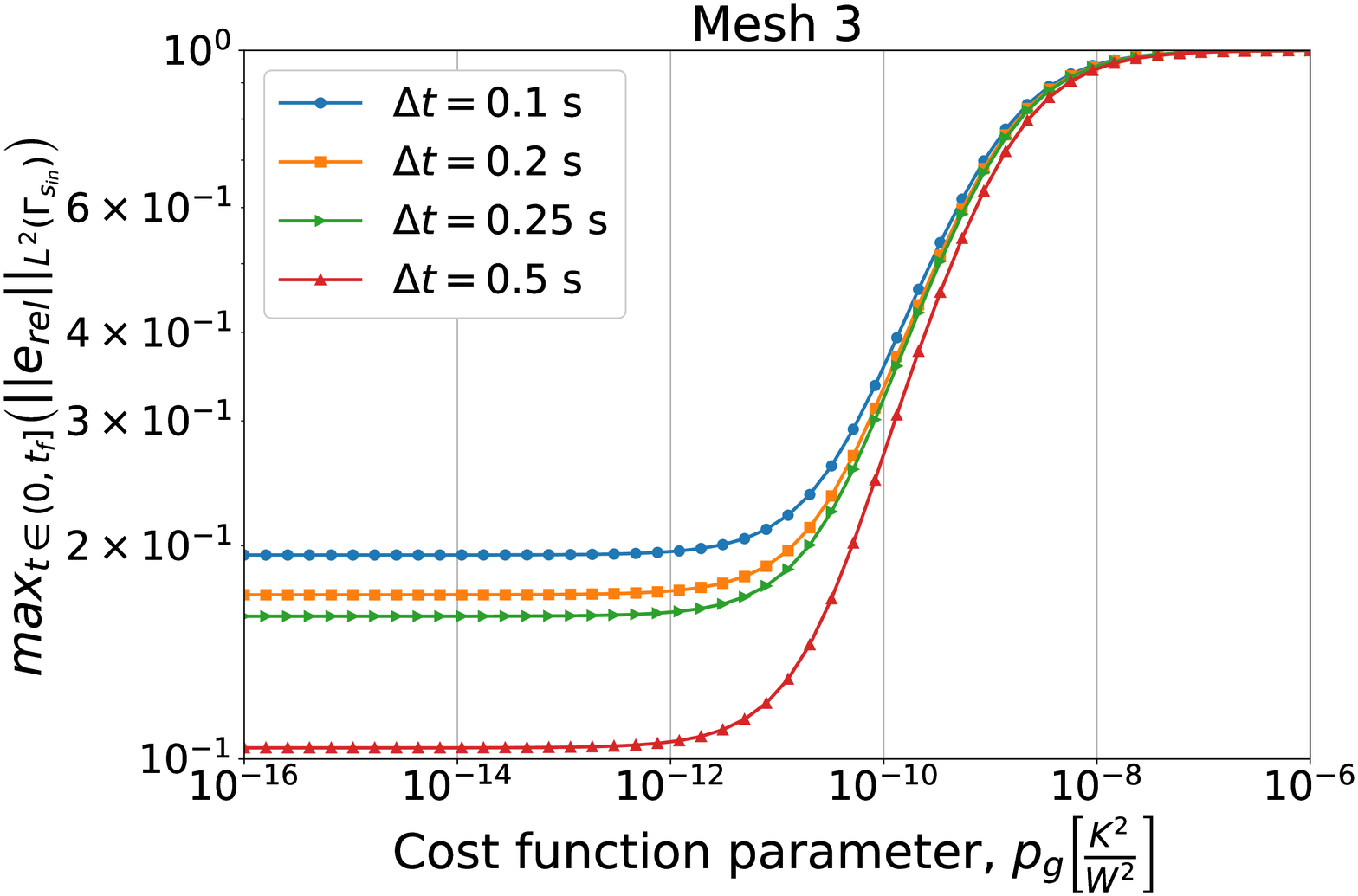}
            \caption{Max. relative error norm.}
        \end{subfigure}%
    \end{subfigure}%
    \caption{Mean (a) and maximum (b) values of the $L^2$-norm of the relative error, $e_{rel}$, in the interval $(0, t_f]$, for Benchmark 1 as the value of the cost function parameter, $p_g$, changes for Algorithm~\ref{alg:inverseSolver_constant_heatNorm} (piecewise constant time approximation of the heat flux).
    We show the results for Mesh 3 and different $\Delta t$.}
    \label{fig:unsteadyNumericalBenchmarkInverseconstant_constant_costFunction_differentDt}
\end{figure}

From the results, we notice that increasing the value of $p_g$ monotonically decreases the quality of the reconstruction.
Moreover, it is true for all considered $\Delta t$ with a slight improvement of the performances as the time discretization gets coarser.

Now, we perform a similar test but this time we keep $\Delta t = 0.25~s$ and test the different meshes of Table~\ref{tab:unsteadyNumericalBenchmark_meshes}.
We illustrate in Figure~\ref{fig:unsteadyNumericalBenchmarkInverseconstant_constant_costFunction_differentMesh} the obtained results.

\begin{figure}[!htb]
    \begin{subfigure}{\linewidth}
        \centering
        \begin{subfigure}[c]{.4\linewidth}
            \captionsetup{width=.85\linewidth}
            \includegraphics[width=.95\textwidth]{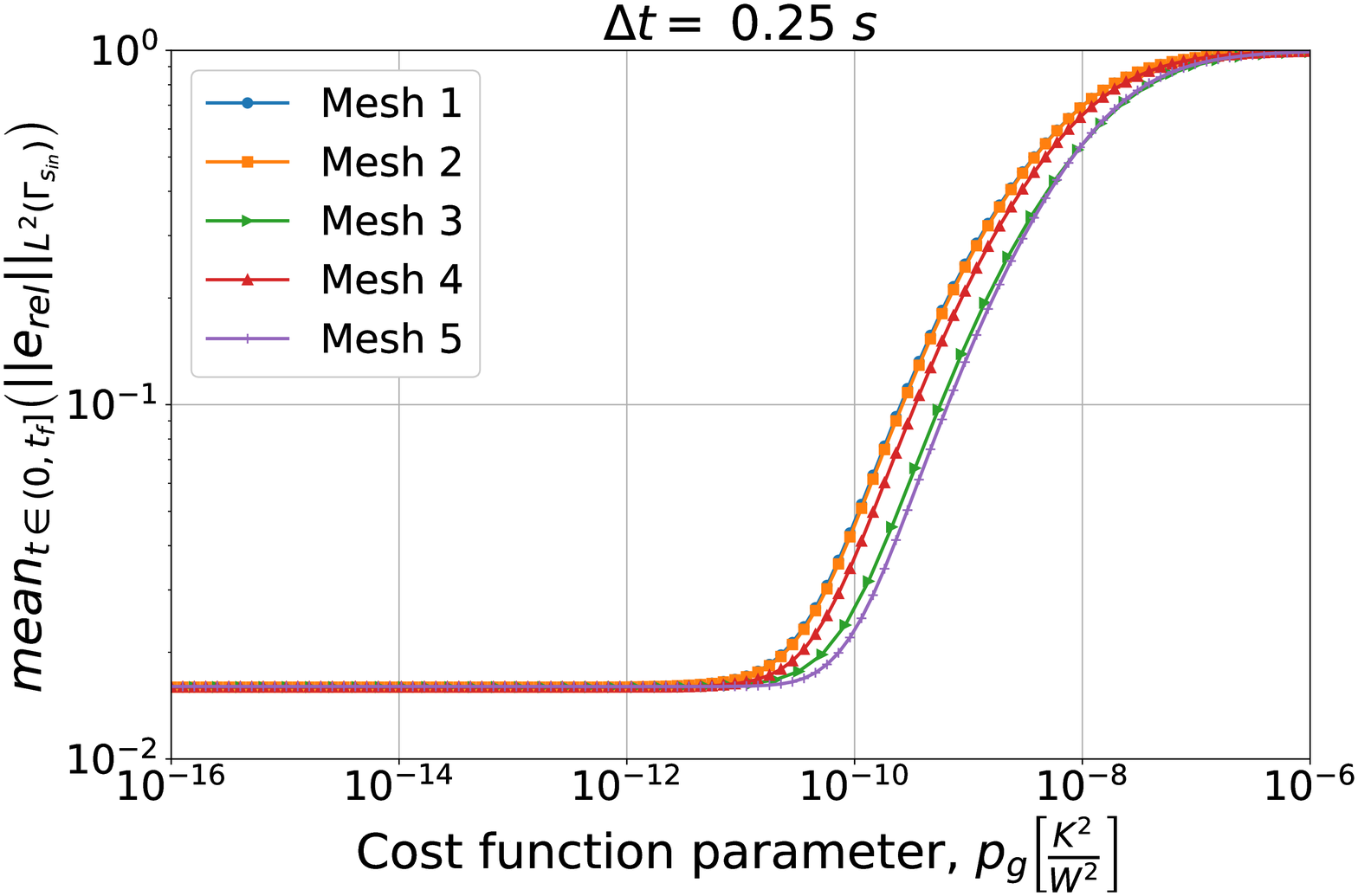}
            \caption{Mean relative error norm.}
        \end{subfigure}%
        \begin{subfigure}[c]{.4\linewidth}
            \captionsetup{width=.85\linewidth}
            \includegraphics[width=.95\textwidth]{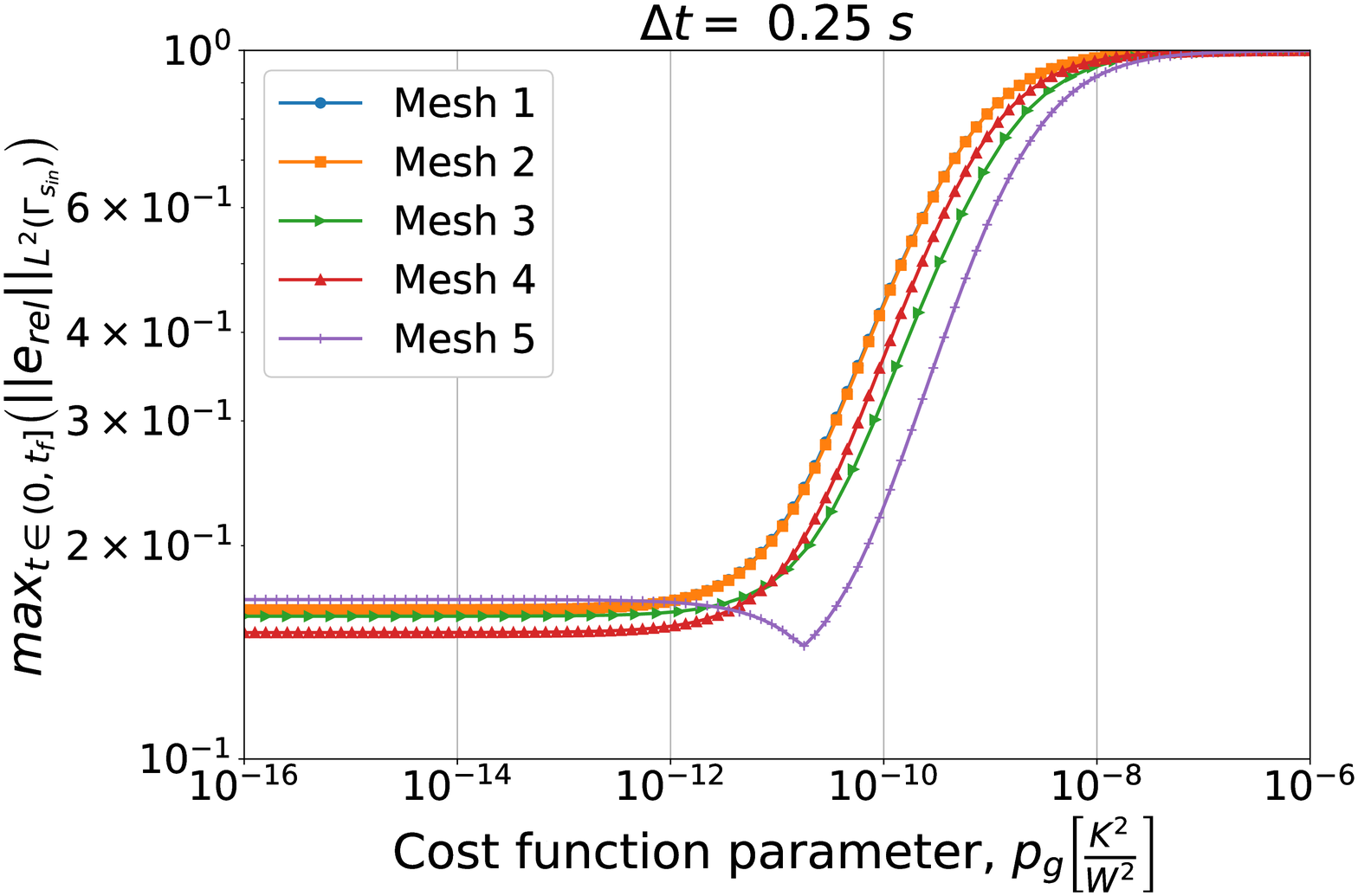}
            \caption{Max. relative error norm.}
        \end{subfigure}%
    \end{subfigure}%
    \caption{Mean (a) and maximum (b) values of the $L^2$-norm of the relative error, $e_{rel}$, in the interval $(0, t_f]$, for Benchmark 1 as the value of the cost function parameter, $p_g$, changes for Algorithm~\ref{alg:inverseSolver_constant_heatNorm} (piecewise constant time approximation of the heat flux).
    We show the results for $\Delta t = 0.25~s$ and different meshes.}
    \label{fig:unsteadyNumericalBenchmarkInverseconstant_constant_costFunction_differentMesh}
\end{figure}

This figure confirms that Algorithm~\ref{alg:inverseSolver_constant_heatNorm} is badly affected by the implementation of the second term of (\ref{eq:sequentialUnsteadyInverseProblem_heatNorm_functional}).
In fact, its performance dramatically deteriorates as soon as this term begins to play a role (i.e. $p_g \gtrsim 10^{-12} K^2/W^2$).
Moreover, the results are almost independent from the discretization refinement.
This further confirms the insensibility of this algorithm from the used discretization.

We continue by performing the same kind of tests on Algorithm~\ref{alg:inverseSolver_constant_heatNorm} (i.e. piecewise linear approximation of the heat flux in time).
We start by testing different timestep sizes while using Mesh 3 for the space discretization.
We present the results in Figure~\ref{fig:unsteadyNumericalBenchmarkInverseLinear_linear_costFunction_differentDt}.

\begin{figure}[!htb]
    \begin{subfigure}{\linewidth}
        \centering
        \begin{subfigure}[c]{.4\linewidth}
            \captionsetup{width=.85\linewidth}
            \includegraphics[width=.95\textwidth]{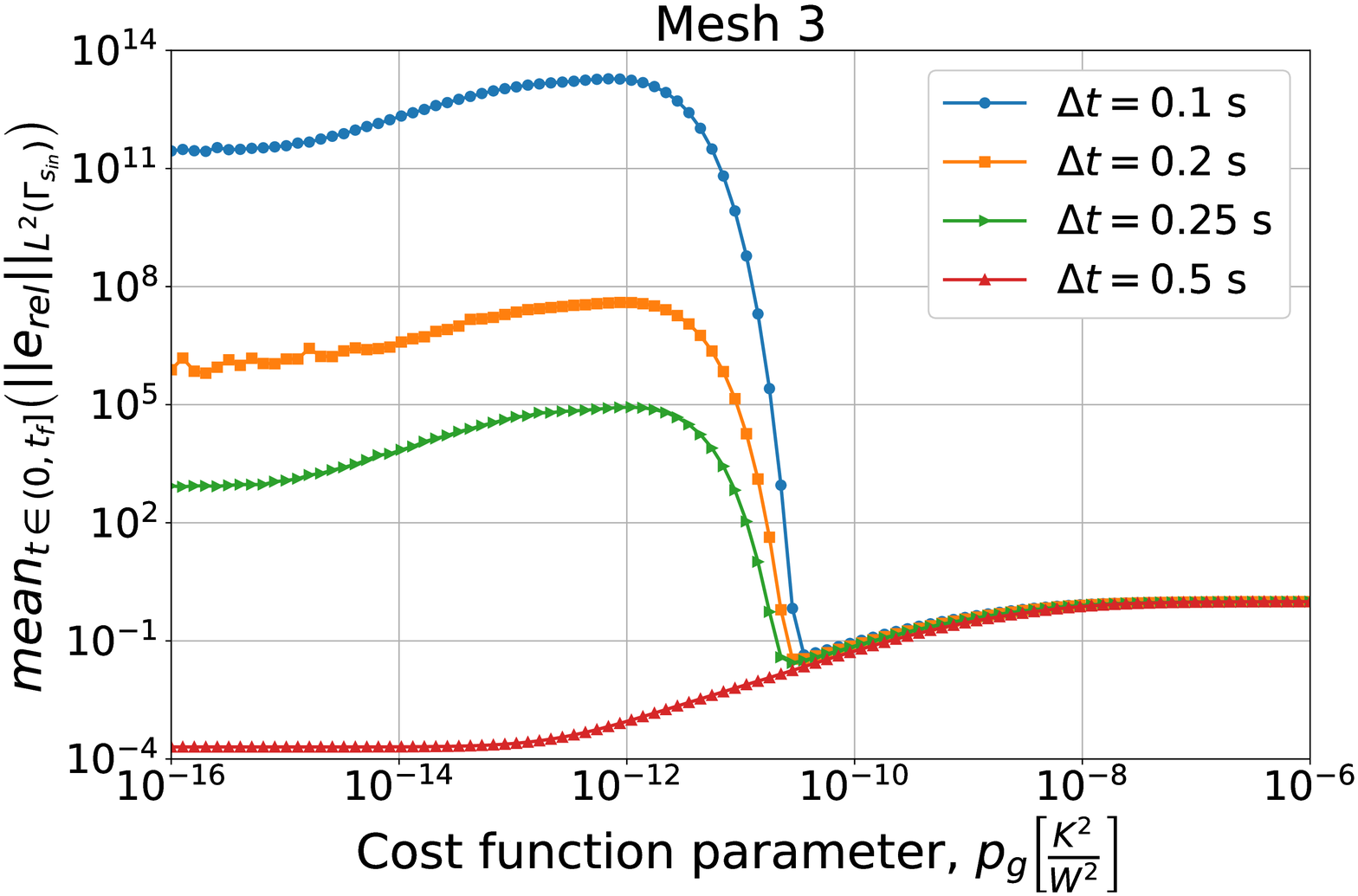}
            \caption{Mean relative error norm.}
        \end{subfigure}%
        \begin{subfigure}[c]{.4\linewidth}
            \captionsetup{width=.85\linewidth}
            \includegraphics[width=.95\textwidth]{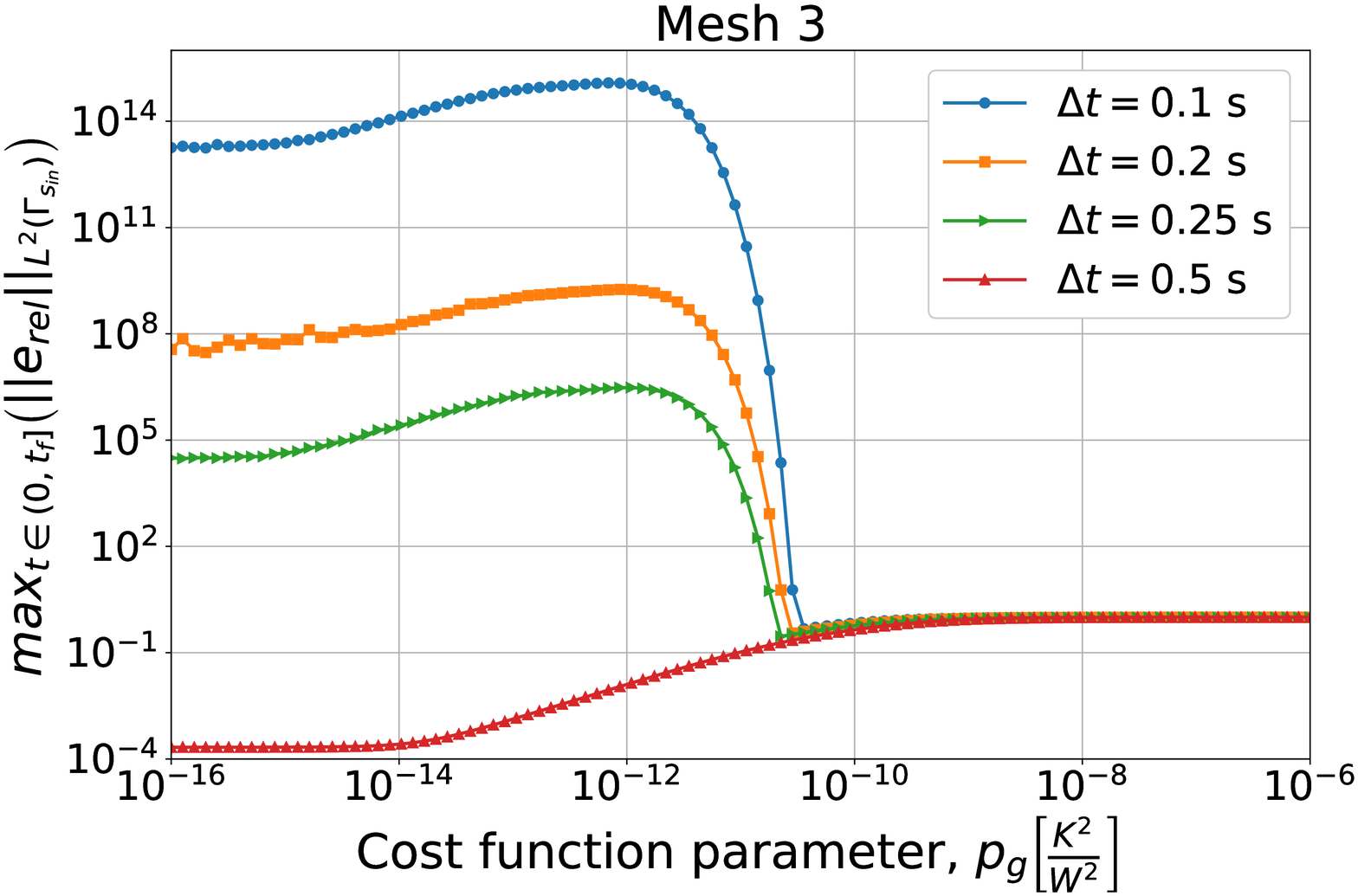}
            \caption{Max. relative error norm.}
        \end{subfigure}%
    \end{subfigure}%
    \caption{Mean (a) and maximum (b) values of the $L^2$-norm of the relative error, $e_{rel}$, in the interval $(0, t_f]$, for Benchmark 1 as the value of the cost function parameter, $p_g$, changes for Algorithm~\ref{alg:inverseSolver_linear_heatNorm} (piecewise linear time approximation of the heat flux).
    We show the results for Mesh 3 and different $\Delta t$.}
    \label{fig:unsteadyNumericalBenchmarkInverseLinear_linear_costFunction_differentDt}
\end{figure}

At first, we notice that this algorithm has a very different behavior with respect to the piecewise constant case.
In this case, the timestep size dramatically affects the results.
We can depict two different behaviors as $p_g$ changes for a chosen $\Delta t$.
In the first one (i.e. $\Delta t = 0.1~s$, $0.2~s$ and $0.25~s$), the inverse solver is very unstable and provides completely useless solutions for low values of $p_g$ (i.e. $p_g \lesssim 10^{-12} K^2/W^2$).
As $p_g$ increases, the quality of the approximation rapidly rises up until the error reaches a minimum.
Here, we have stable solutions and a good approximation of the heat flux.
For higher values of $p_g$, the error monotonically increases until it reaches a plateau at $100\%$.

On the other hand, we have a different behavior for $\Delta t = 0.5~s$.
In this case, the inverse solver performs similarly to the piecewise constant case, but the quality of the estimation is by almost two orders of magnitude better.
Then, we have stable and accurate solutions for low values of $p_g$.
For $p_g \gtrsim 10^{-12} K^2/W^2$ we have a monotonic degradation of the heat flux estimation until we reach the $100\%$ plateau.

It is interesting to notice that the second term in the functional $S_2^k$ can make the solver insensible to the discretization refinement.
In fact, after a certain value of $p_g$, the relative error norms for the different $\Delta t$ are almost coincident.

We can see a similar behavior in Figure~\ref{fig:unsteadyNumericalBenchmarkInverseLinear_linear_costFunction_differentMeshes} where we show the results obtained refining the mesh and keeping $\Delta t = 0.25~s$.
Also in this case, we notice the two previously described behaviors with the coarsest mesh being always stable and providing the best results for the lowest values of $p_g$.

\begin{figure}[!htb]
    \begin{subfigure}{\linewidth}
        \centering
        \begin{subfigure}[c]{.4\linewidth}
            \captionsetup{width=.85\linewidth}
            \includegraphics[width=.95\textwidth]{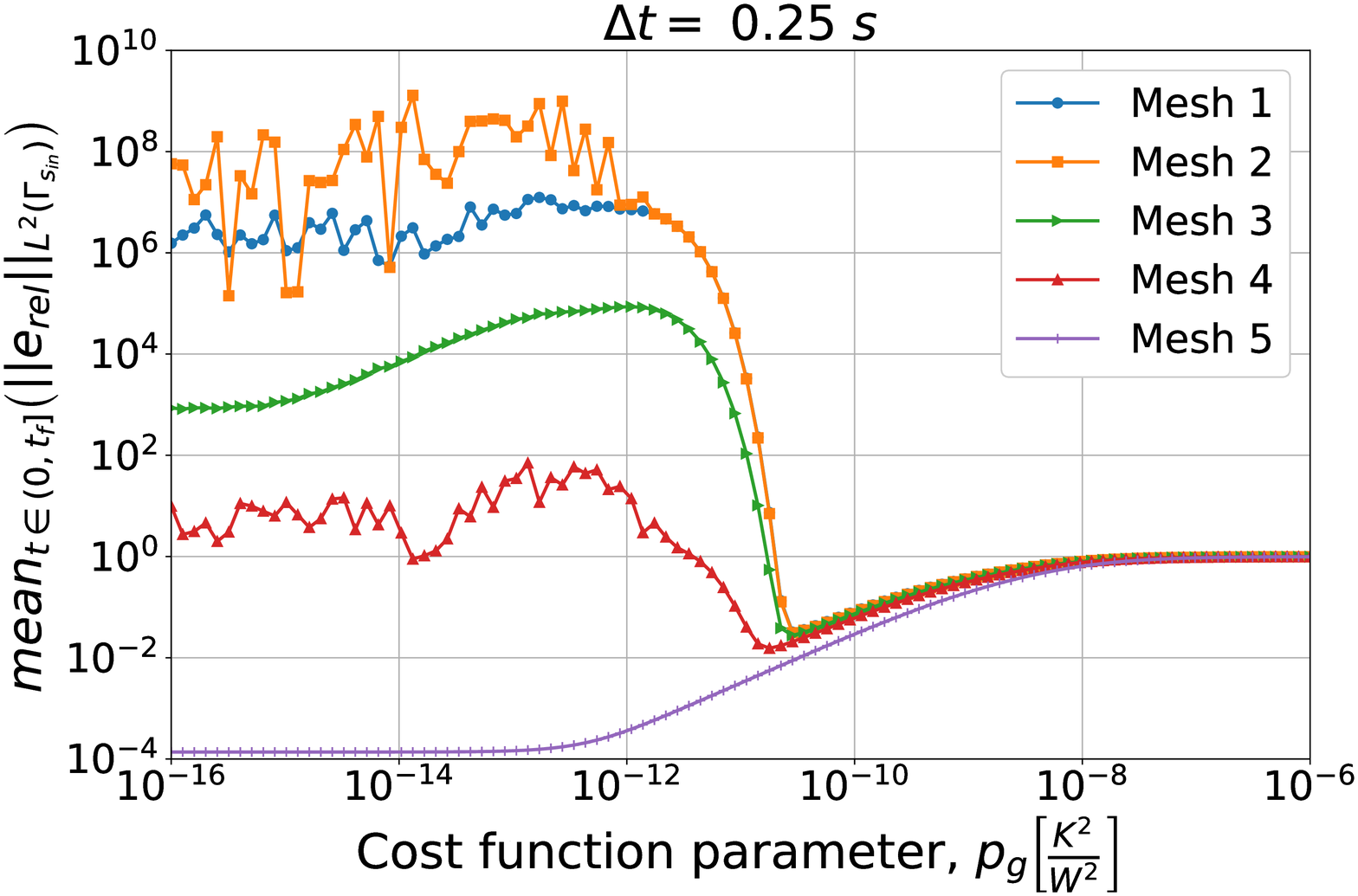}
            \caption{Mean relative error norm.}
        \end{subfigure}%
        \begin{subfigure}[c]{.4\linewidth}
            \captionsetup{width=.85\linewidth}
            \includegraphics[width=.95\textwidth]{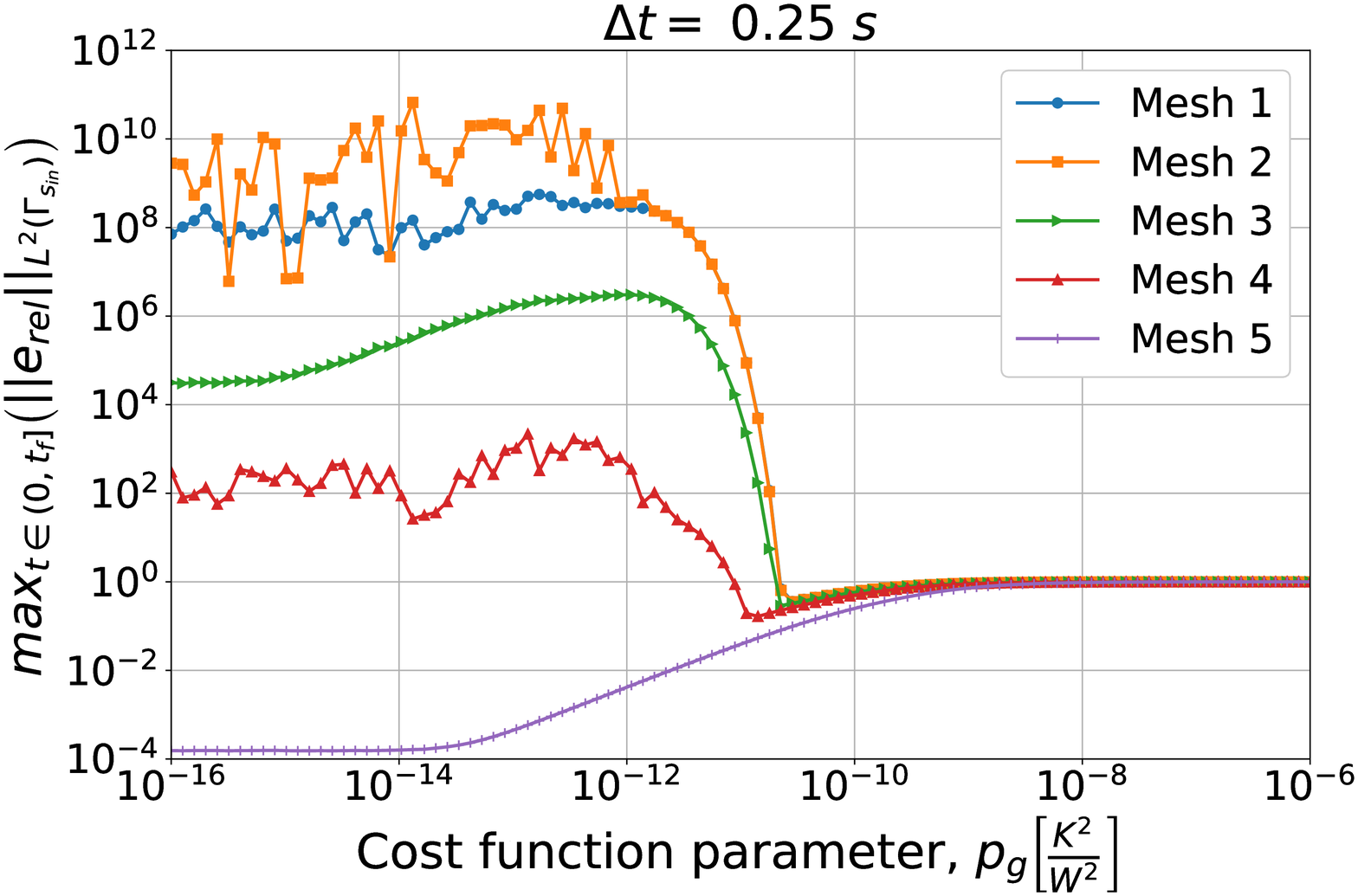}
            \caption{Max. relative error norm.}
        \end{subfigure}%
    \end{subfigure}%
    \caption{Mean (a) and maximum (b) values of the $L^2$-norm of the relative error, $e_{rel}$, in the interval $(0, t_f]$, for Benchmark 1 as the value of the cost function parameter, $p_g$, changes for Algorithm~\ref{alg:inverseSolver_linear_heatNorm} (piecewise linear time approximation of the heat flux).
    We show the results for $\Delta t = 0.25s$ and different meshes.}
    \label{fig:unsteadyNumericalBenchmarkInverseLinear_linear_costFunction_differentMeshes}
\end{figure}

These statements are remarked by the results shown in Figure~\ref{fig:unsteadyNumericalBenchmarkInverseLinear_linear_timeSpaceRefinement_costFunc5e-11}, where we show the results of the same test as in Figure~\ref{fig:unsteadyNumericalBenchmarkInverseLinear_linear_timeSpaceRefinement} but for $p_g = 5e-11 \frac{K^2}{W^2}$.
The obtained results confirm that for some values of $p_g$ we can obtain a stable solver with a moderate dependency on the discretization refinement.

\begin{figure}[!htb]
    \begin{subfigure}{\linewidth}
        \centering
        \begin{subfigure}[c]{.4\linewidth}
            \captionsetup{width=.85\linewidth}
            \includegraphics[width=.95\textwidth]{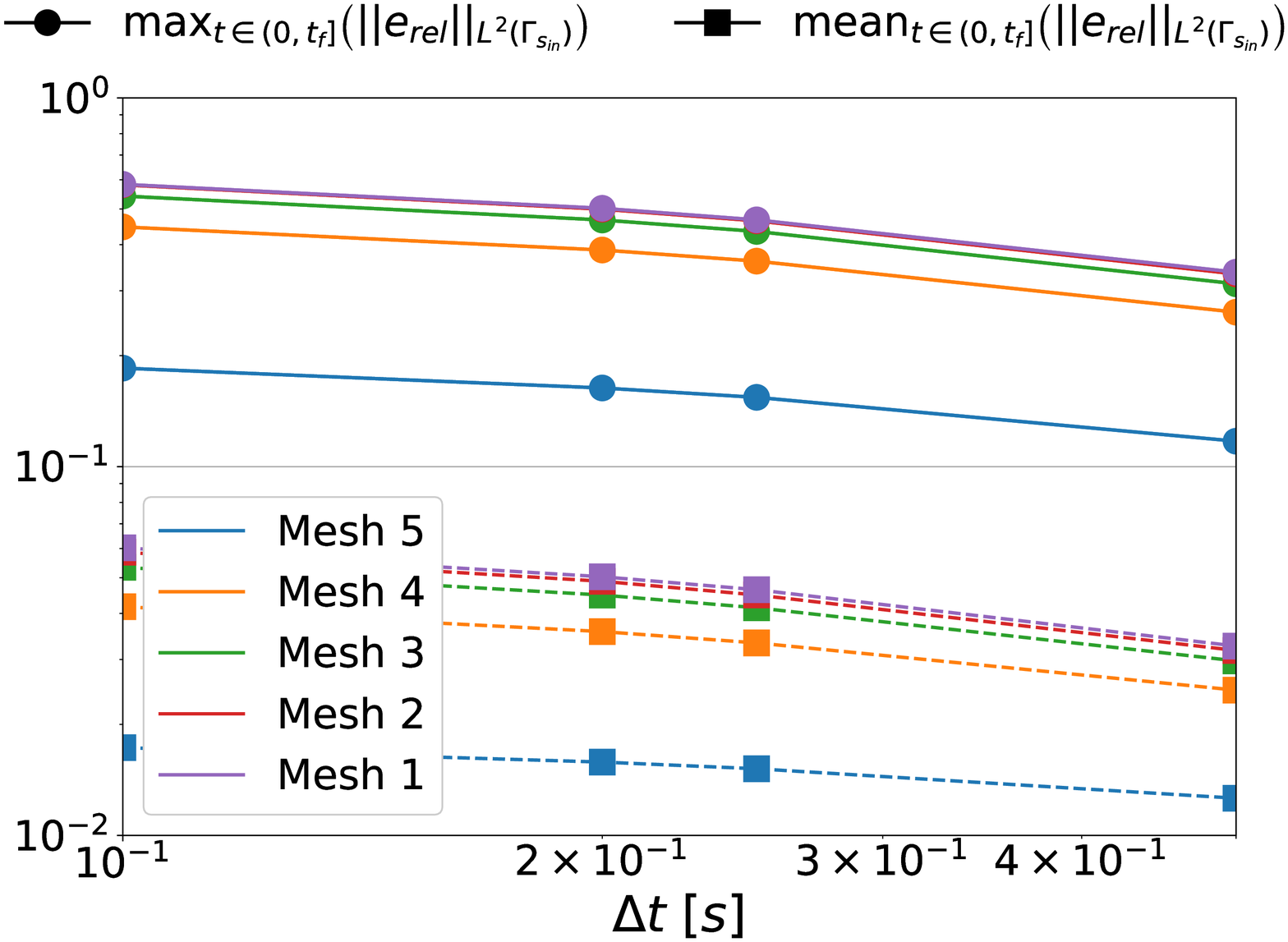}
            \caption{$L^2$-norm of the relative error as a function of $\Delta t$}
        \end{subfigure}%
        \begin{subfigure}[c]{.4\linewidth}
            \captionsetup{width=.85\linewidth}
            \includegraphics[width=.95\textwidth]{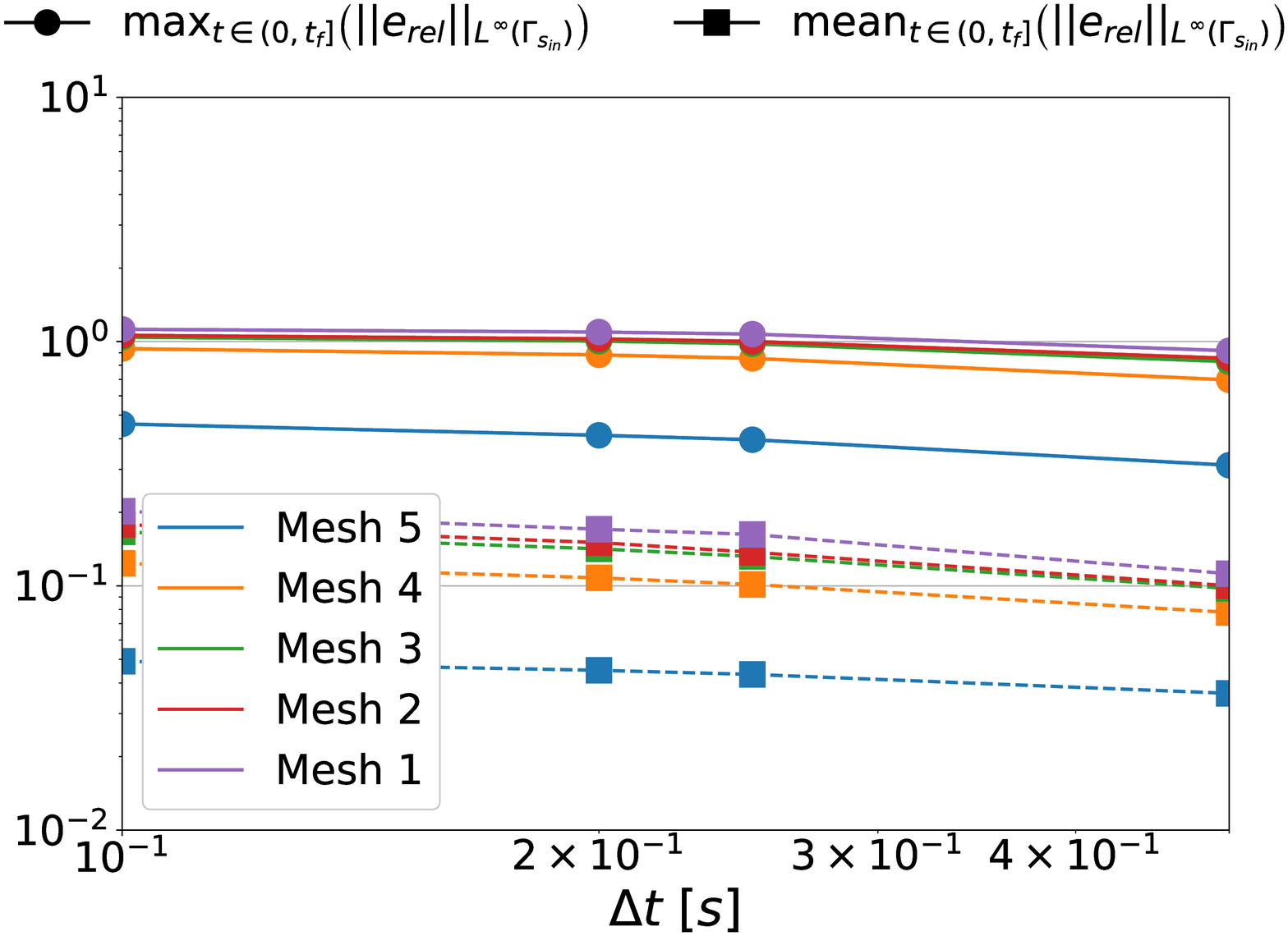}
            \caption{$L^\infty$-norm of the relative error as a function of $\Delta t$}
        \end{subfigure}%
    \end{subfigure}%
    \\
    \begin{subfigure}{\linewidth}
        \centering
        \begin{subfigure}[c]{.4\linewidth}
            \captionsetup{width=.85\linewidth}
            \includegraphics[width=.95\textwidth]{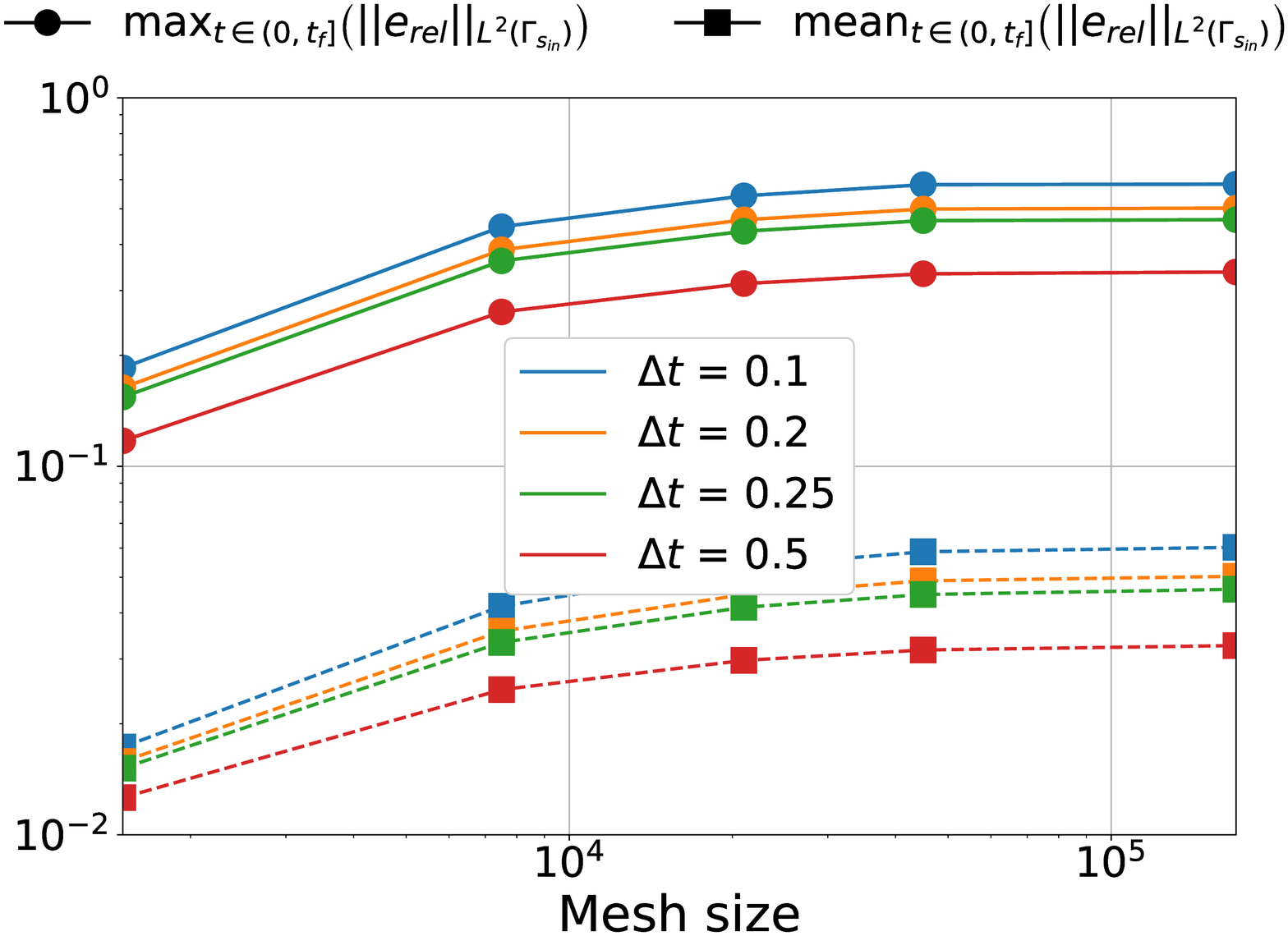}
            \caption{$L^2$-norm of the relative error as a function of the mesh size}
        \end{subfigure}%
        \begin{subfigure}[c]{.4\linewidth}
            \captionsetup{width=.85\linewidth}
            \includegraphics[width=.95\textwidth]{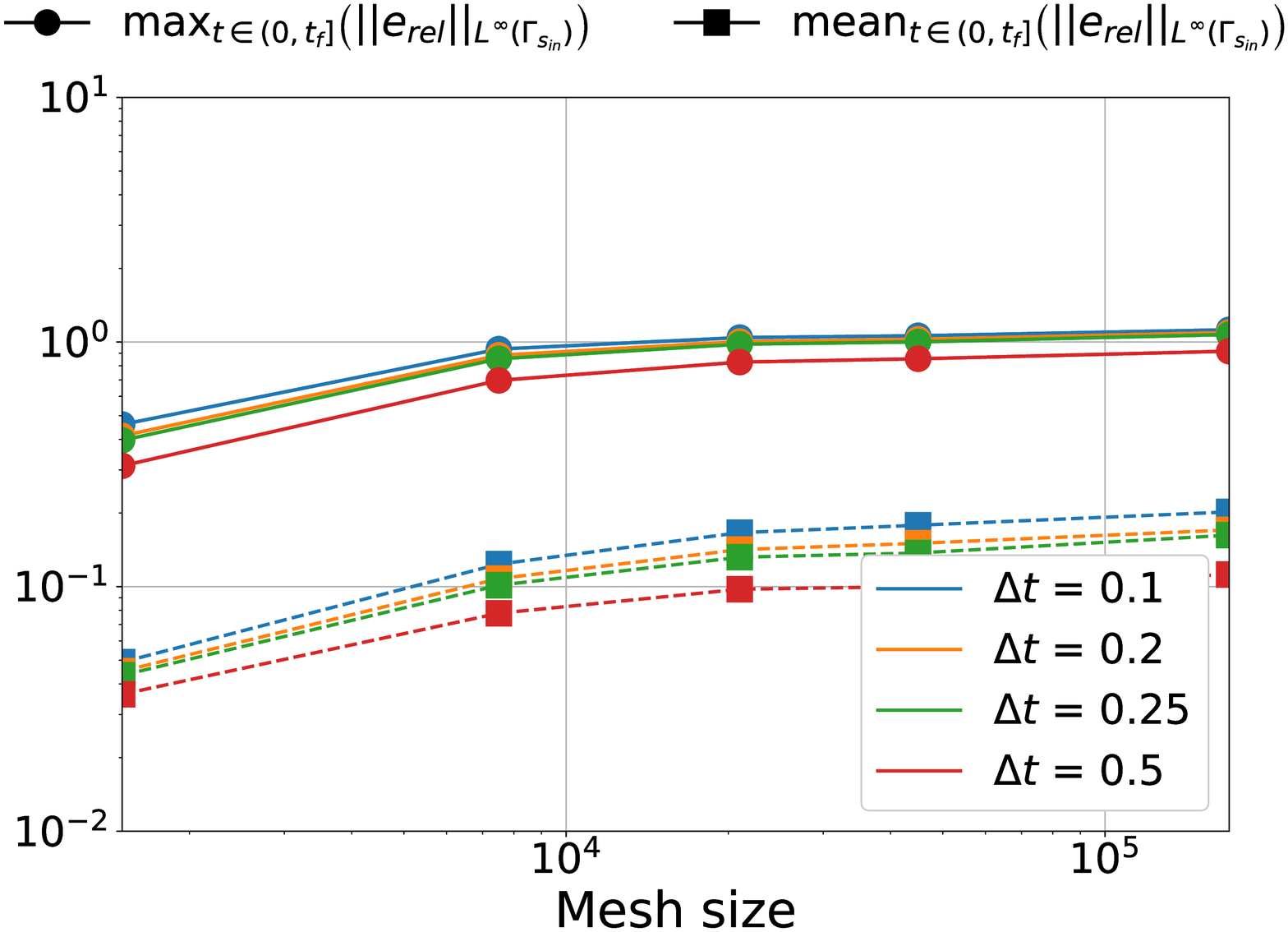}
            \caption{$L^\infty$-norm of the relative error as a function of the mesh size}
        \end{subfigure}%
    \end{subfigure}%
    \caption{Maximum (circles) and mean (squares) values of the $L^2$- and $L^\infty$-norm of the relative error, $e_{rel}$, in the interval $(0, t_f]$, for Benchmark 1 as the time and space discretization changes for Algorithm~\ref{alg:inverseSolver_linear_heatNorm} (piecewise linear time approximation of the heat flux and $p_g = 5e-11 \frac{K^2}{W^2}$).}
    \label{fig:unsteadyNumericalBenchmarkInverseLinear_linear_timeSpaceRefinement_costFunc5e-11}
\end{figure}

To conclude this analysis, we test the discretization and $p_g$ selection method of Algorithm~\ref{alg:meshSelection} in this benchmark case.
We use the virtual thermocouples measurements as input training dataset for the algorithm, $\hat{T}_{train}$.
In the test, the algorithm has to select a combination of mesh, timestep size and $p_g$ that provides stable and accurate solutions to this inverse problem.
Before presenting the results of Algorithm~\ref{alg:meshSelection}, we show in Figure~\ref{fig:unsteadyNumericalBenchmarkInverseLinear_linear_costFunction_measDiscrepancy} the mean value of the temperature discrepancy functional $S_1^k$, defined in (\ref{eq:unsteadyInverseFunctional1}), as function of $p_g$ for different meshes and $\Delta t$.

\begin{figure}[!htb]
    \begin{subfigure}{\linewidth}
        \centering
        \begin{subfigure}[c]{.4\linewidth}
            \captionsetup{width=.85\linewidth}
            \includegraphics[width=.95\textwidth]{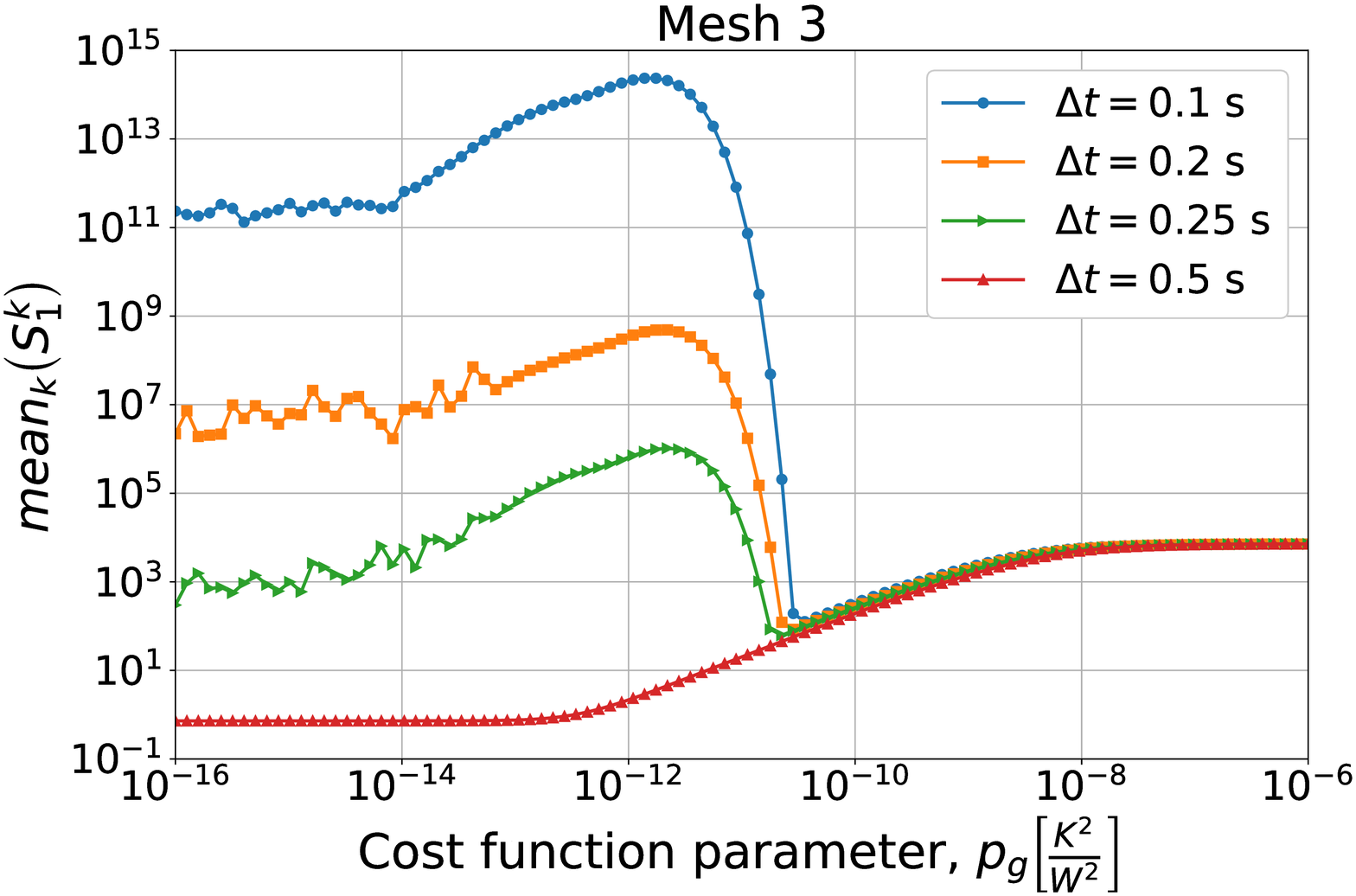}
            \caption{Mean of $S_1^k$ for Mesh 3.}
        \end{subfigure}%
        \begin{subfigure}[c]{.4\linewidth}
            \captionsetup{width=.85\linewidth}
            \includegraphics[width=.95\textwidth]{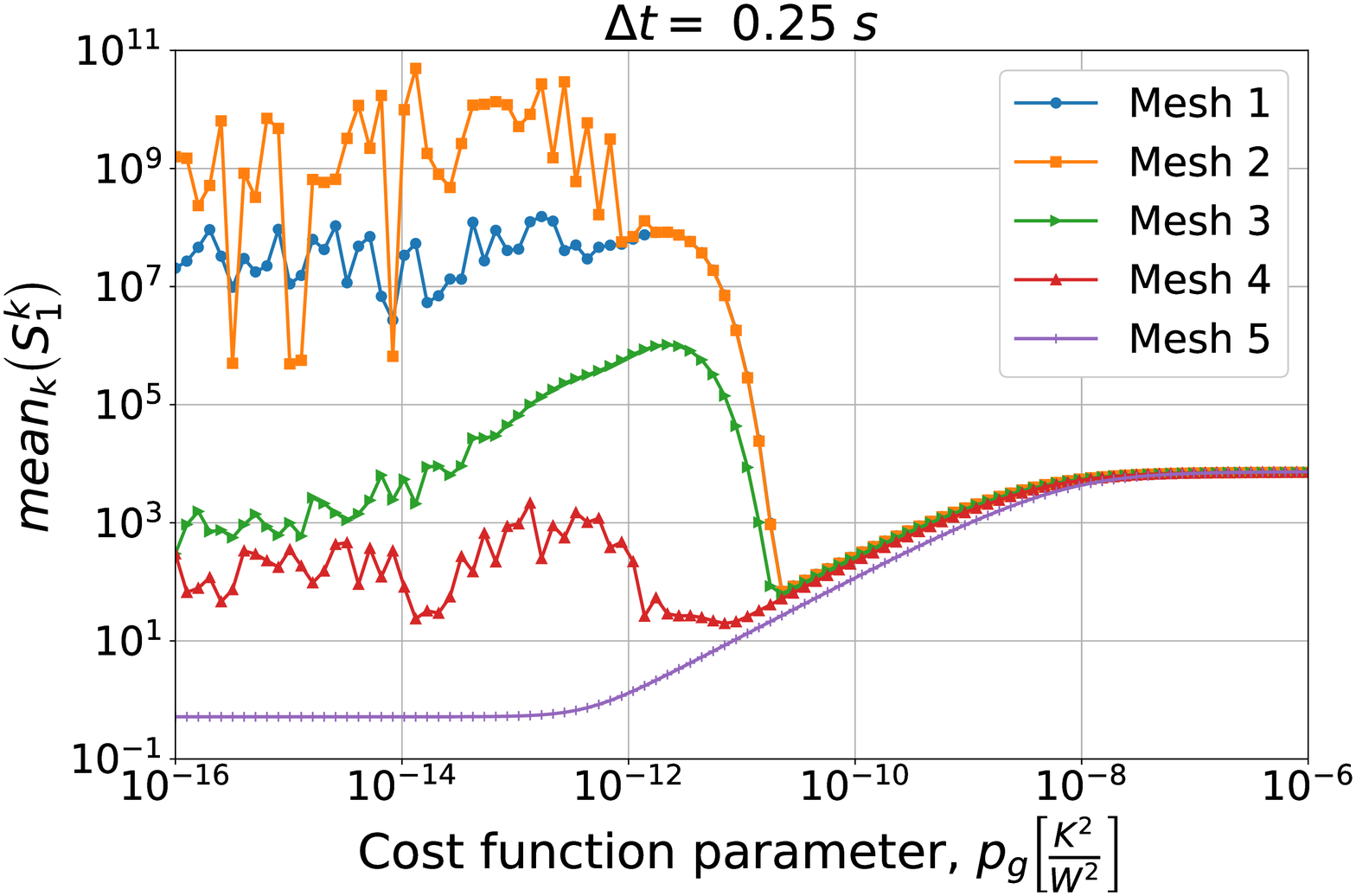}
            \caption{Mean of $S_1^k$ for $\Delta t = 0.25s$.}
        \end{subfigure}%
    \end{subfigure}%
    \caption{Mean values of $S_1^k$ for $1 \leq k \leq P_t$, for Benchmark 1 as the value of the cost function parameter, $p_g$, changes.
    The results are obtained using Algorithm~\ref{alg:inverseSolver_linear_heatNorm} (piecewise linear time approximation of the heat flux).
    We show the results for Mesh 3 and different $\Delta t$ in (a), and for $\Delta t = 0.25$ and different meshes in (b).}
    \label{fig:unsteadyNumericalBenchmarkInverseLinear_linear_costFunction_measDiscrepancy}
\end{figure}

These results show that $m_S$ in (\ref{eq:discrSelection_mesurementsDiscrepancy}) behaves similarly to the relative error (\ref{eq:unsteadyBenchmark_relativeError}) as function of $p_g$.
Notice that it presents the same two behaviors that we previously described for the relative error depending on the used discretization setup.
Moreover, the $m_S$ minima are close to those of the relative error.
For these reasons, in Section~\ref{sec:unsteadyDiscretizationSelectionAlgorithm}, we used this result in the creation of the selection criteria for the $p_g$ as well as for the mesh and the $\Delta t$.

That said, we are now ready to test Algorithm~\ref{alg:meshSelection}.
With respect to its implementation, in step~\ref{alg:meshSelection_minimization} of the algorithm, we use the Nelder-Mead method to find the $p_g$ that minimizes $S_1^k$\cite{Donald1975}.
To start the algorithm, we set $p_g^0 = 1e-7~\frac{K^2}{W^2}$.
Table~\ref{tab:unsteadyNumericalBenchmarkInverse_linear_meshSelection} summarizes the algorithm behavior.

\begin{table}[htb]
\centering
    \caption{Test of Algorithm~\ref{alg:meshSelection} for Benchmark 1.}
    \label{tab:unsteadyNumericalBenchmarkInverse_linear_meshSelection}
    \begin{tabular}{ |l|c|c|c|c|}
        \hline
        \textbf{Iteration}    &   \textbf{Mesh}    &   $\Delta t~[s]$    &   $p_g~\left[ \frac{K^2}{W^2} \right]$    &   $mean_k \left( S_1^k \right)~\left[ K^2 \right]$\\
        \hline
        0   &   5   &   $0.5$ & $1e-7$    &    $6.9e3$\\   
        1   &   5   &   $0.5$ & $3.2e-21$    &    $3.9e-1$\\   
        \hline
    \end{tabular}
\end{table}

From the results in the table, we appreciate that the algorithm chooses the coarsest discretization since the first iteration.
Then, it looks for the $p_g$ that minimizes $S_1^k$ for this discretization and, not finding  a better discretization setup for this value of $p_g$, exits the process.
Comparing the obtained results to the relative error plots of Figures~\ref{fig:unsteadyNumericalBenchmarkInverseLinear_linear_costFunction_differentDt} and \ref{fig:unsteadyNumericalBenchmarkInverseLinear_linear_costFunction_differentMeshes}, we confirm that the algorithm is selecting the best configuration in between all the available.

%%%%%%%%%%%%%%%%%%%%%%%%%%%%%%%%%%%%%%%%%%%%%%%%%%%%%%%%%%%%%%%%%%%%%%%%%%%%%%%%%
\subsection{Benchmark 2}\label{sec:unsteadyBenchmark_nonlinear}

In designing this benchmark case, we use the same geometrical and physical parameters as in Benchmark 1, but we choose a non linear in time true heat flux, $g_{tr}$, as in Table~\ref{tab:unsteadyNumericalBenchmarkInverse_nonlinear_parameters}.

\begin{table}[htb]
\centering
    \caption{Parameters used for the unsteady Benchmark 2.}
    \label{tab:unsteadyNumericalBenchmarkInverse_nonlinear_parameters}
    \begin{tabular}{ ll }
        \hline
        \textbf{Parameter}    &   \textbf{Value}\\
        \hline
        $g_1(\mathbf{x})$                           & $bz^2 + c$\\
        $g_2(\mathbf{x})$                           & $\frac{10 c}{1 + (x - 1)^2 + z^2}$\\
        Heat flux, $g_{tr}(\mathbf{x},t)$                & $- k_s \left[ g_1  + \frac{g_1}{2} \sin{\left(2 \pi f_{max} \frac{t^2}{t_f}\right)} + g_2 e^{-0.1t} \right]$ W/m\textsuperscript{2} \\
        Maximum frequency, $f_{max}$    & $0.1$ Hz\\
        \hline
    \end{tabular}
\end{table}

Also for this benchmark, we test both the piecewise constant and linear approximation algorithms.
In particular, we investigate the regularization properties of the discretization coarsening, the effects of the $p_g$ parameter and the robustness of the algorithms to noise in the measurements.

\subsubsection{Effect of Time and Space Discretization Refinement}

The first test we do is related to the space and time discretization refinement.
To test the effect of changing the space and time discretization sizes, we reconstruct the heat flux $g_{tr}(\mathbf{x},t)$ for all the meshes in Table~\ref{tab:unsteadyNumericalBenchmark_meshes} and $\Delta t = 0.1$, $0.2$, $0.25$ and $0.5~s$.
In these tests, we use the cost functional (\ref{eq:unsteadyInverseFunctional1}) (i.e. $p_g = 0~\frac{K^2}{W^2}$) and do not add noise to the measurements.

First, we test Algorithm~\ref{alg:inverseSolver_constant}.
Figure~\ref{fig:unsteadyNumericalBenchmarkInverseNonLinear_constant_timeSpaceRefinement} illustrates the mean and maximum values in $(0, t_f]$ of the $L^2$- and $L^\infty$-norm of the relative error (\ref{eq:unsteadyBenchmark_relativeError}).
The results show a entirely similar behavior to that of the previous benchmark as described in Section~\ref{sec:benchmark_spaceTimeRef1}.

\begin{figure}[!htb]
    \begin{subfigure}{\linewidth}
        \centering
        \begin{subfigure}[c]{.4\linewidth}
            \captionsetup{width=.85\linewidth}
            \includegraphics[width=.95\textwidth]{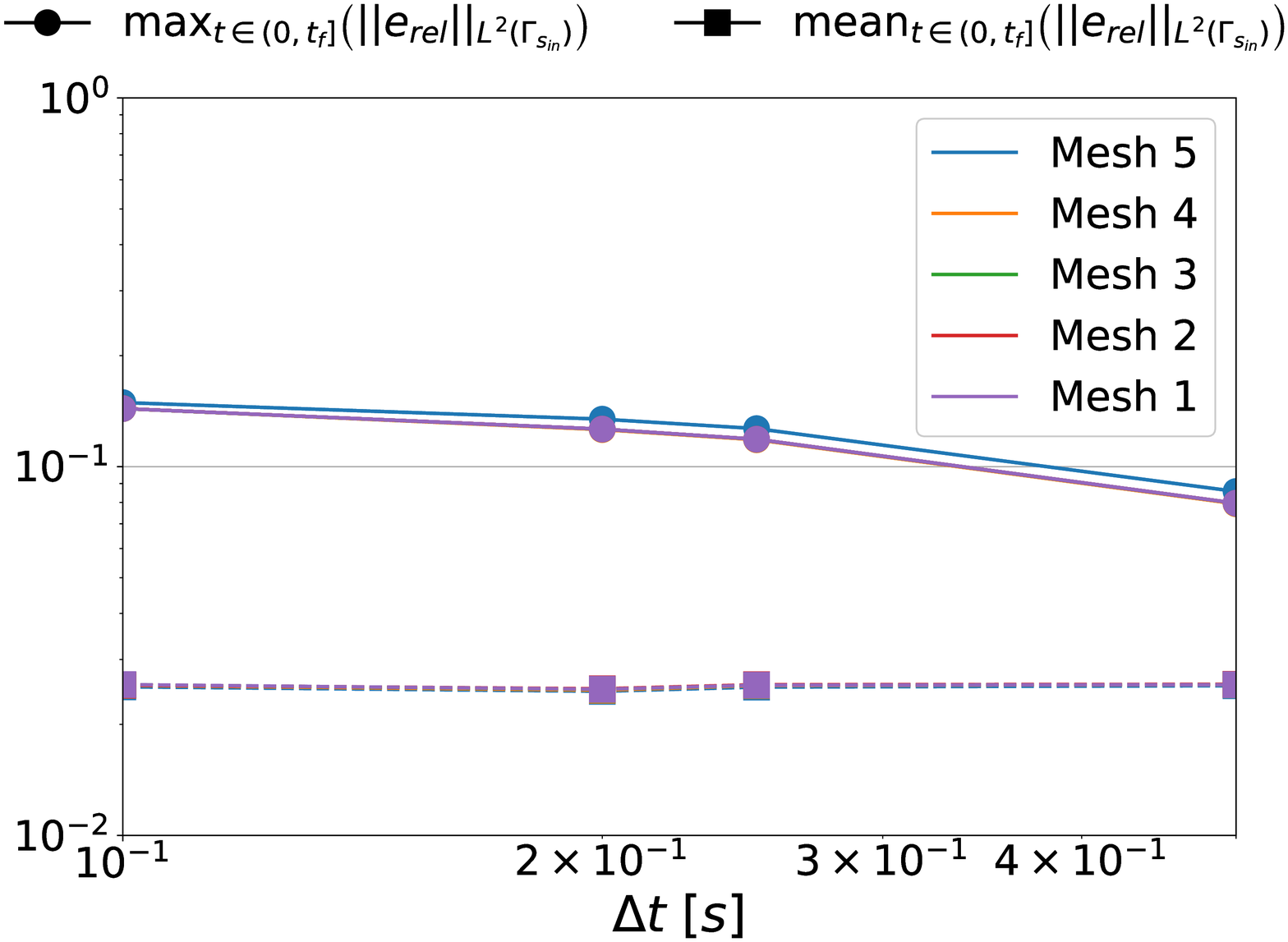}
            \caption{$L^2$-norm of the relative error as a function of $\Delta t$}
        \end{subfigure}%
        \begin{subfigure}[c]{.4\linewidth}
            \captionsetup{width=.85\linewidth}
            \includegraphics[width=.95\textwidth]{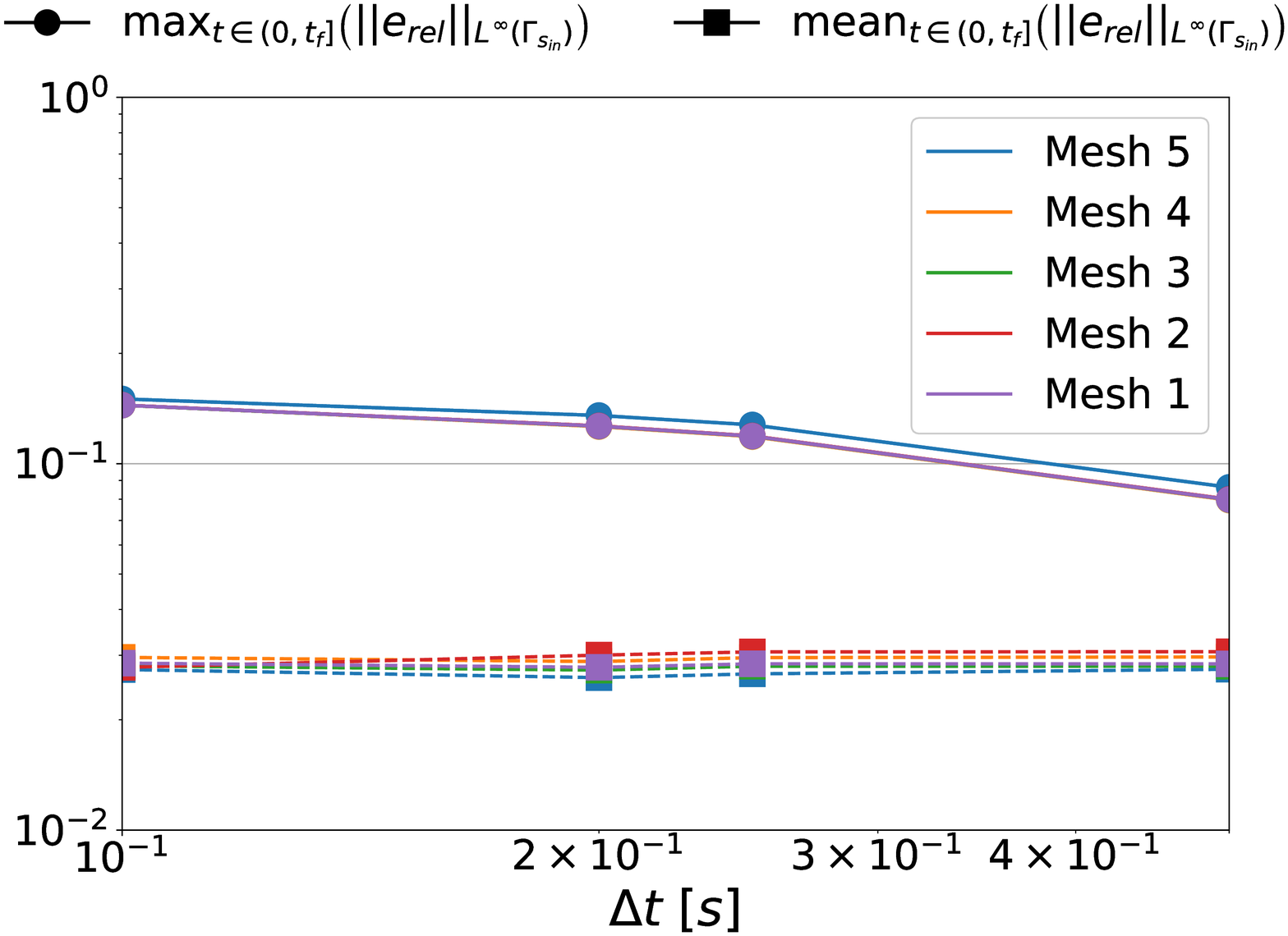}
            \caption{$L^\infty$-norm of the relative error as a function of $\Delta t$}
        \end{subfigure}%
    \end{subfigure}%
    \\
    \begin{subfigure}{\linewidth}
        \centering
        \begin{subfigure}[c]{.4\linewidth}
            \captionsetup{width=.85\linewidth}
            \includegraphics[width=.95\textwidth]{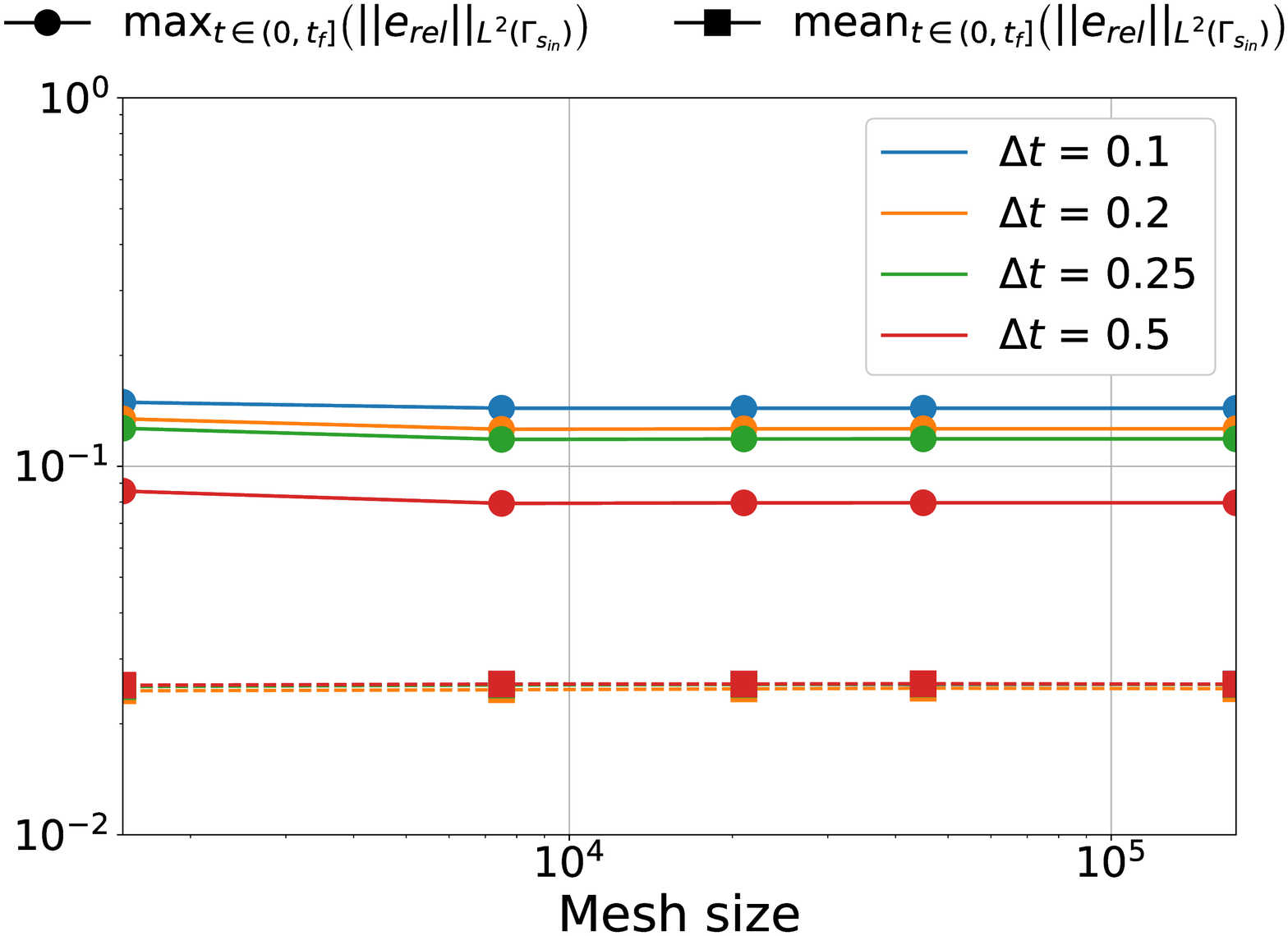}
            \caption{$L^2$-norm of the relative error as a function of the mesh size}
        \end{subfigure}%
        \begin{subfigure}[c]{.4\linewidth}
            \captionsetup{width=.85\linewidth}
            \includegraphics[width=.95\textwidth]{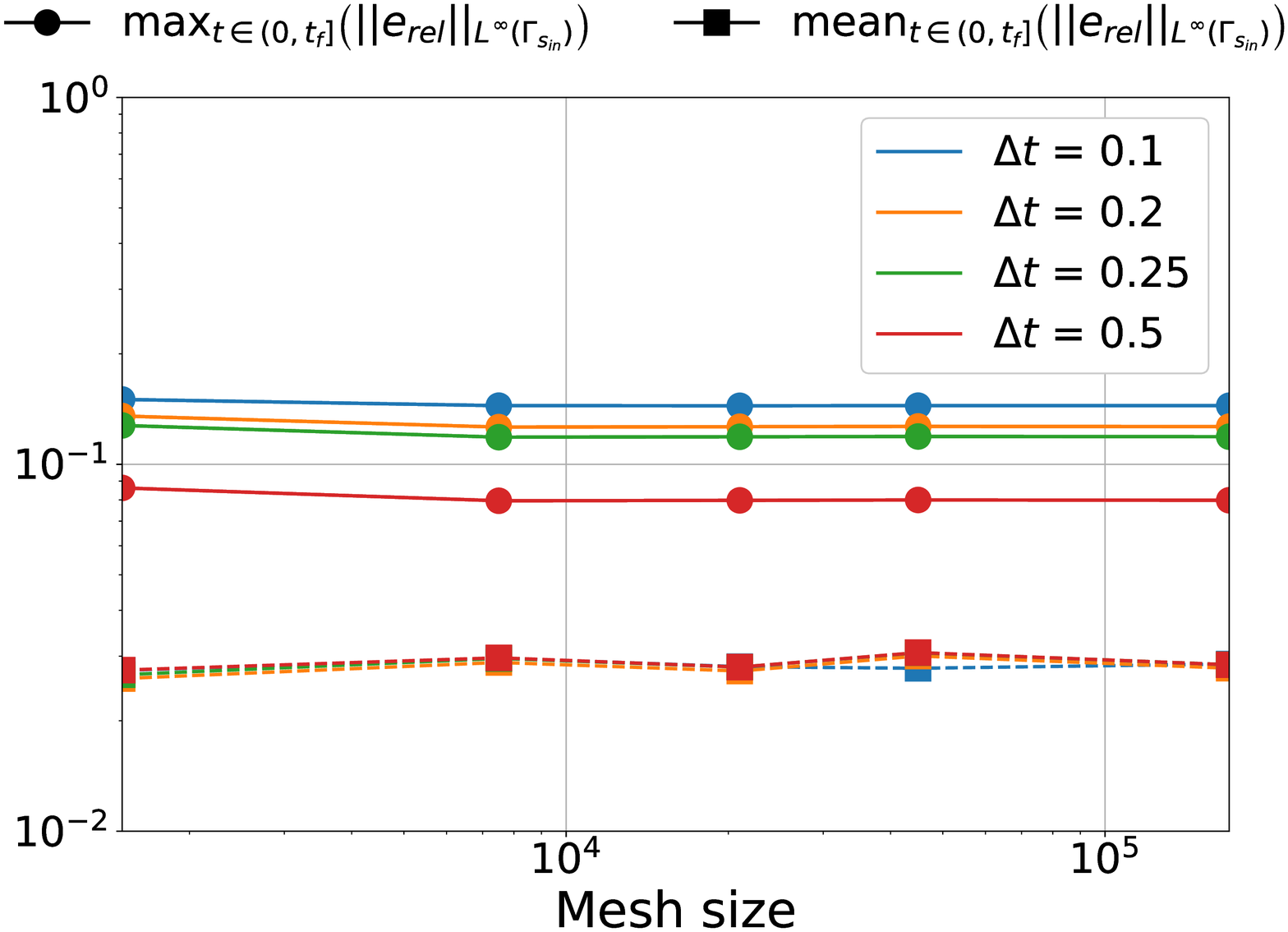}
            \caption{$L^\infty$-norm of the relative error as a function of the mesh size}
        \end{subfigure}%
    \end{subfigure}%
    \caption{Maximum (circles) and mean (squares) values of the $L^2$- and $L^\infty$-norm of the relative error, $e_{rel}$, in the interval $(0, t_f]$, for Benchmark 2 as the time and space discretization changes for Algorithm~\ref{alg:inverseSolver_constant} (piecewise constant time approximation of the heat flux and $p_g = 0 \frac{K^2}{W^2}$).}
    \label{fig:unsteadyNumericalBenchmarkInverseNonLinear_constant_timeSpaceRefinement}
\end{figure}

\begin{figure}[!htb]
    \begin{subfigure}{\linewidth}
        \centering
        \begin{subfigure}[c]{.4\linewidth}
            \captionsetup{width=.85\linewidth}
            \includegraphics[width=.95\textwidth]{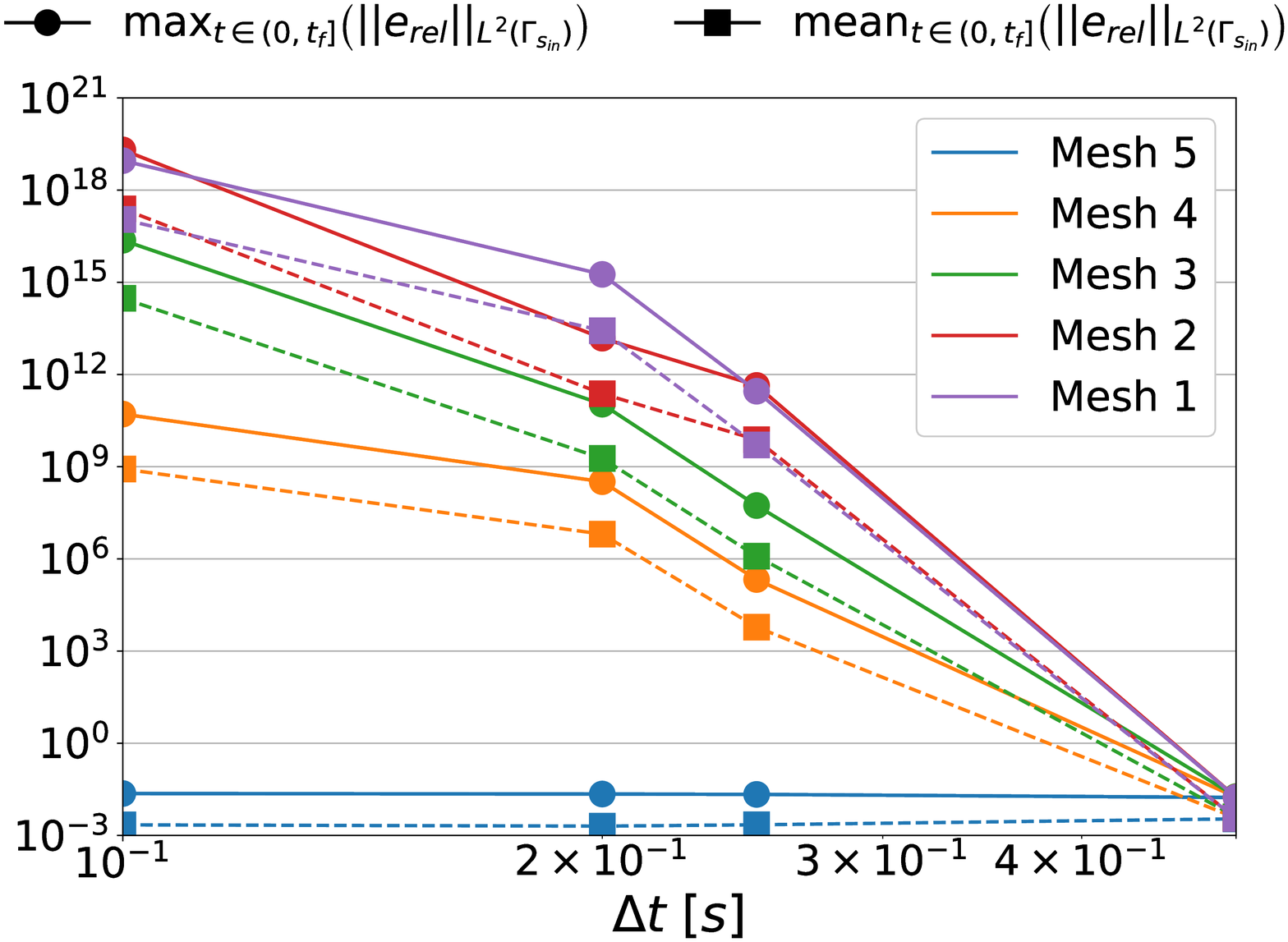}
            \caption{$L^2$-norm of the relative error as a function of $\Delta t$}
        \end{subfigure}%
        \begin{subfigure}[c]{.4\linewidth}
            \captionsetup{width=.85\linewidth}
            \includegraphics[width=.95\textwidth]{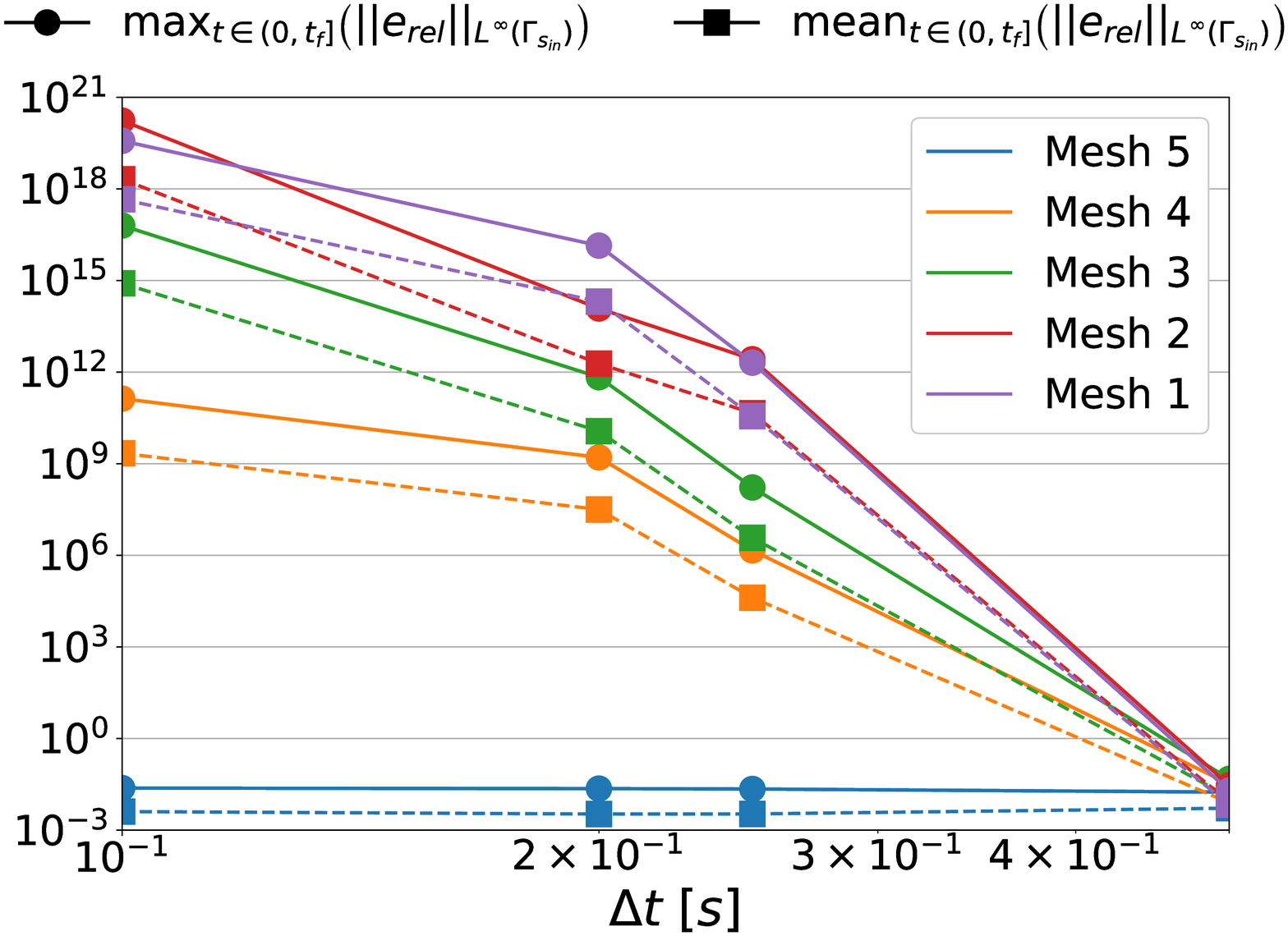}
            \caption{$L^\infty$-norm of the relative error as a function of $\Delta t$}
        \end{subfigure}%
    \end{subfigure}%
    \\
    \begin{subfigure}{\linewidth}
        \centering
        \begin{subfigure}[c]{.4\linewidth}
            \captionsetup{width=.85\linewidth}
            \includegraphics[width=.95\textwidth]{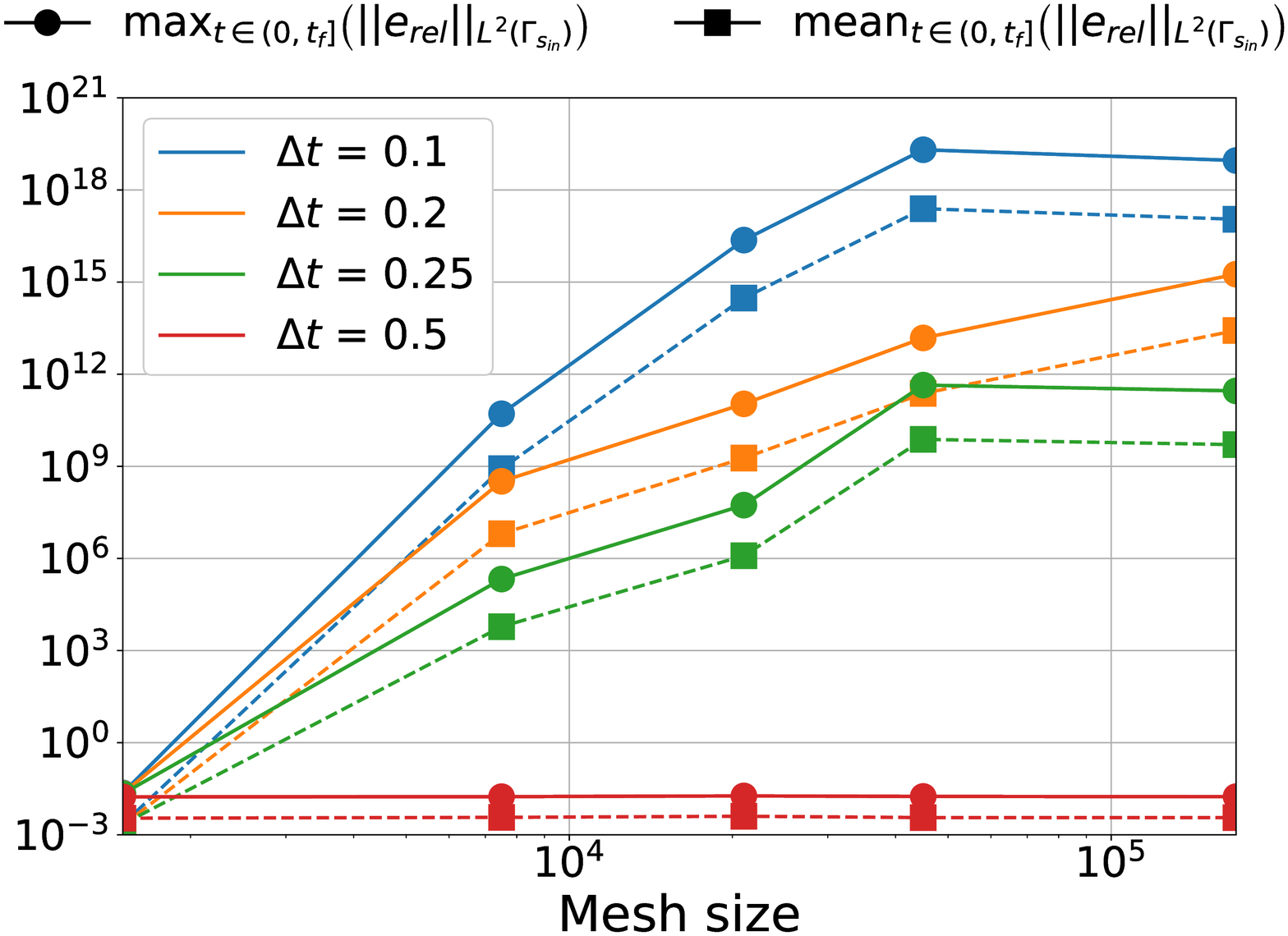}
            \caption{$L^2$-norm of the relative error as a function of the mesh size}
        \end{subfigure}%
        \begin{subfigure}[c]{.4\linewidth}
            \captionsetup{width=.85\linewidth}
            \includegraphics[width=.95\textwidth]{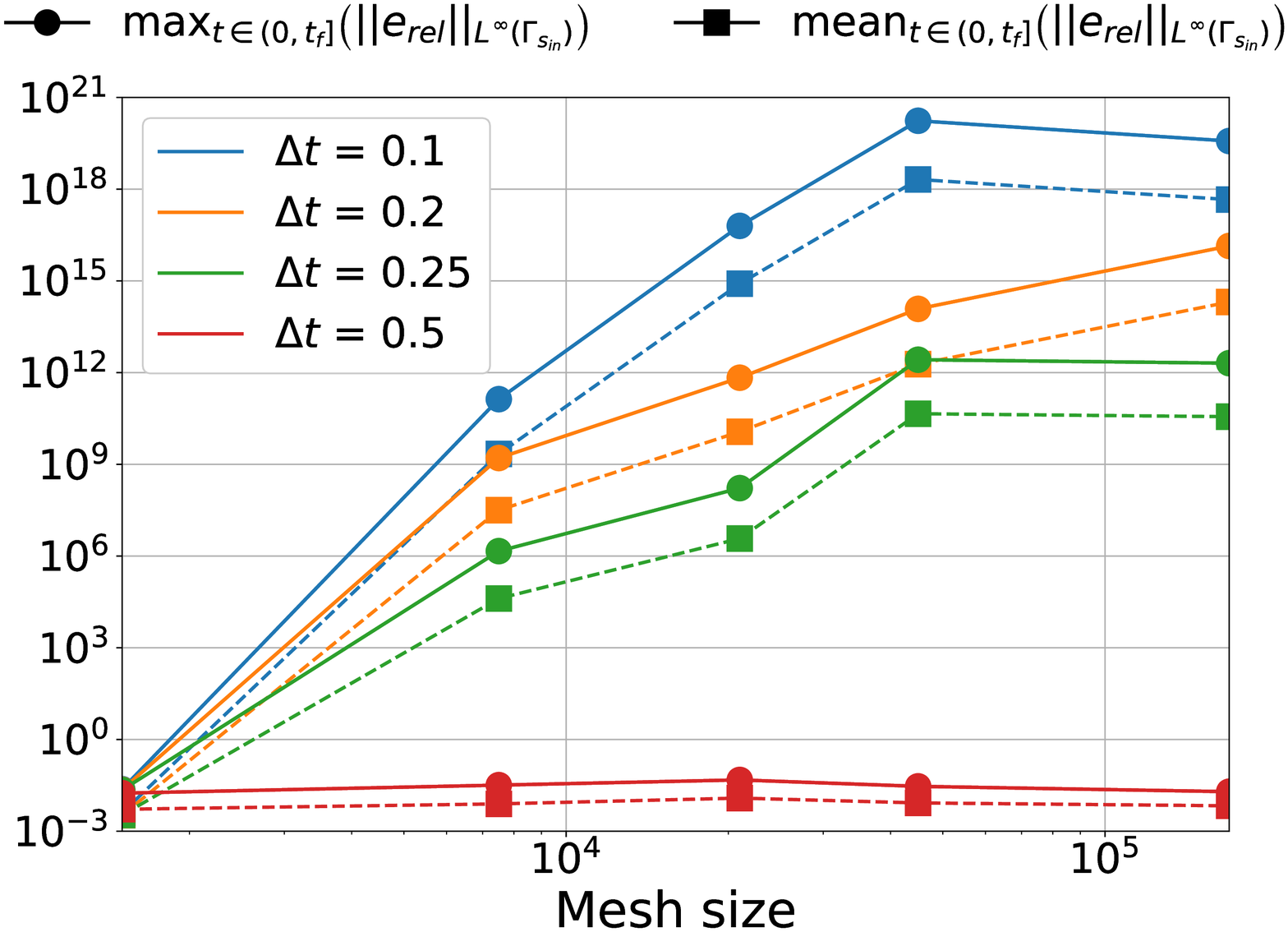}
            \caption{$L^\infty$-norm of the relative error as a function of the mesh size}
        \end{subfigure}%
    \end{subfigure}%
    \caption{Maximum (circles) and mean (squares) values of the $L^2$- and $L^\infty$-norm of the relative error, $e_{rel}$, in the interval $(0, t_f]$, for Benchmark 2 as the time and space discretization changes for Algorithm~\ref{alg:inverseSolver_linear} (piecewise linear time approximation of the heat flux and $p_g = 0 \frac{K^2}{W^2}$).}
    \label{fig:unsteadyNumericalBenchmarkInverseNonLinear_linear_timeSpaceRefinement}
\end{figure}

%%%%%%%%%%%%%%%%%%%%%%%%%%%%%%%%%%%%%%%%%%%%%%%%%%%%%%%%%%%%%%%%%%%%%%%%%%%%%%%%%
\subsubsection{Effect of Cost Functional Parameter, $p_g$}

Also for this benchmark case, we test the effects that the parameter $p_g$ in (\ref{eq:sequentialUnsteadyInverseProblem_heatNorm_functional}) has on the inverse solvers performances.
We do it by performing the same tests of Section~\ref{sec:unsteadyBenchmarkLinear_pgEffects} but for the present test case.
In particular, we solve this inverse problem using all the meshes in Table~\ref{tab:unsteadyNumericalBenchmark_meshes} and $\Delta t= 0.1, 0.2, 0.25$ and $0.5~s$, for $1e-16 \leq p_g \leq 1e-6 \frac{K^2}{W^2}$. 
We use both the piecewise constant and linear approximation of the heat flux in Algorithm~\ref{alg:inverseSolver_constant_heatNorm} and \ref{alg:inverseSolver_linear_heatNorm}, respectively.
Notice that in these tests, we use the full order algorithms.

First, we present in Figures~\ref{fig:unsteadyNumericalBenchmarkInverseNonLinear_constant_costFunction} and \ref{fig:unsteadyNumericalBenchmarkInverseNonLinear_constant_costFunction_meshes} the results obtained using the piecewise constant approximation algorithm.
These plots show the mean and maximum values of the $L^2$-norm of the relative error, $e_{rel}$, in the interval $(0, t_f]$ as functions of $p_g$ changing the mesh and the $\Delta t$, respectively.

\begin{figure}[!htb]
    \begin{subfigure}{\linewidth}
        \centering
        \begin{subfigure}[c]{.4\linewidth}
            \captionsetup{width=.85\linewidth}
            \includegraphics[width=.95\textwidth]{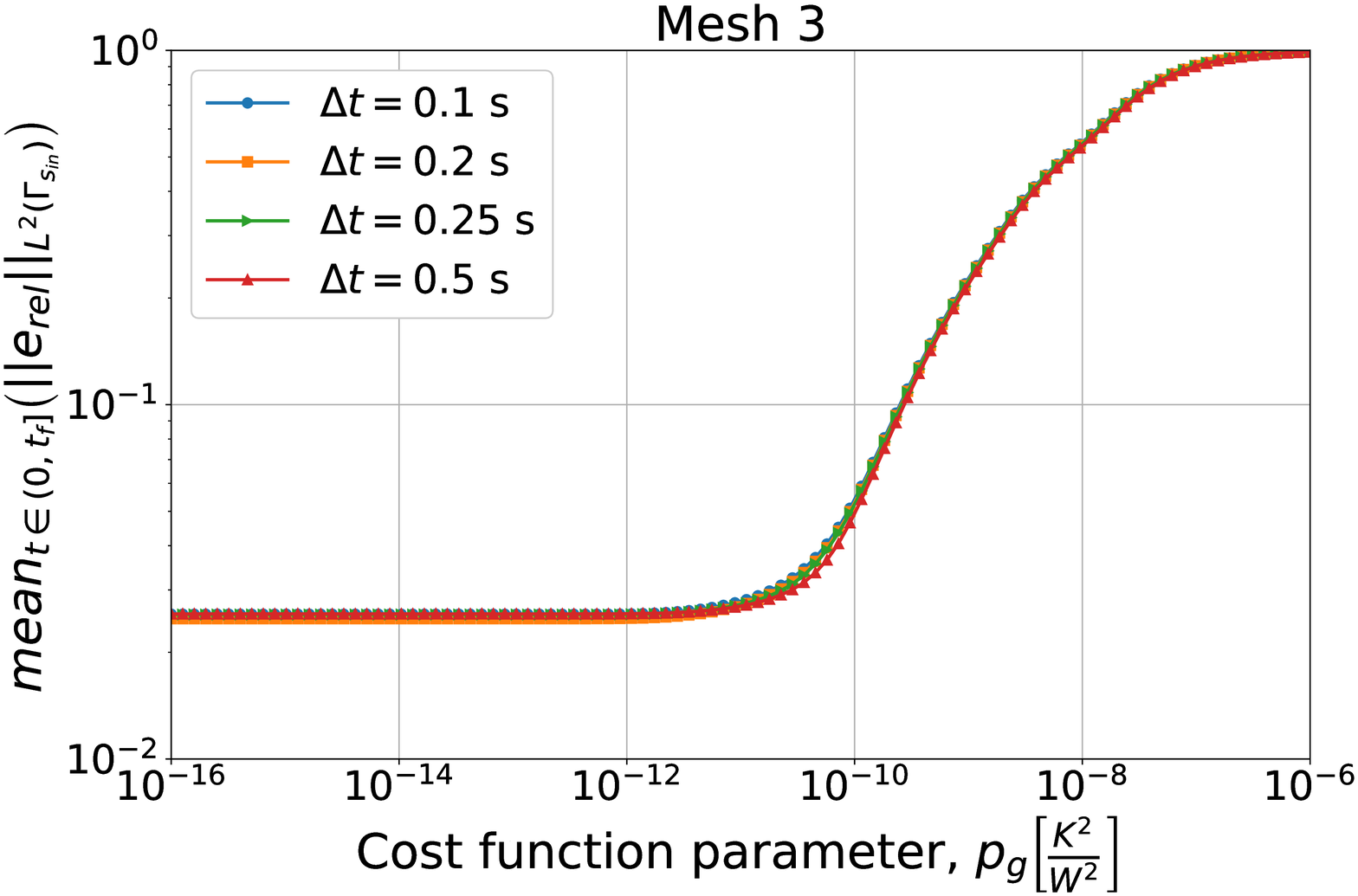}
            \caption{Mean relative error norm.}
        \end{subfigure}%
        \begin{subfigure}[c]{.4\linewidth}
            \captionsetup{width=.85\linewidth}
            \includegraphics[width=.95\textwidth]{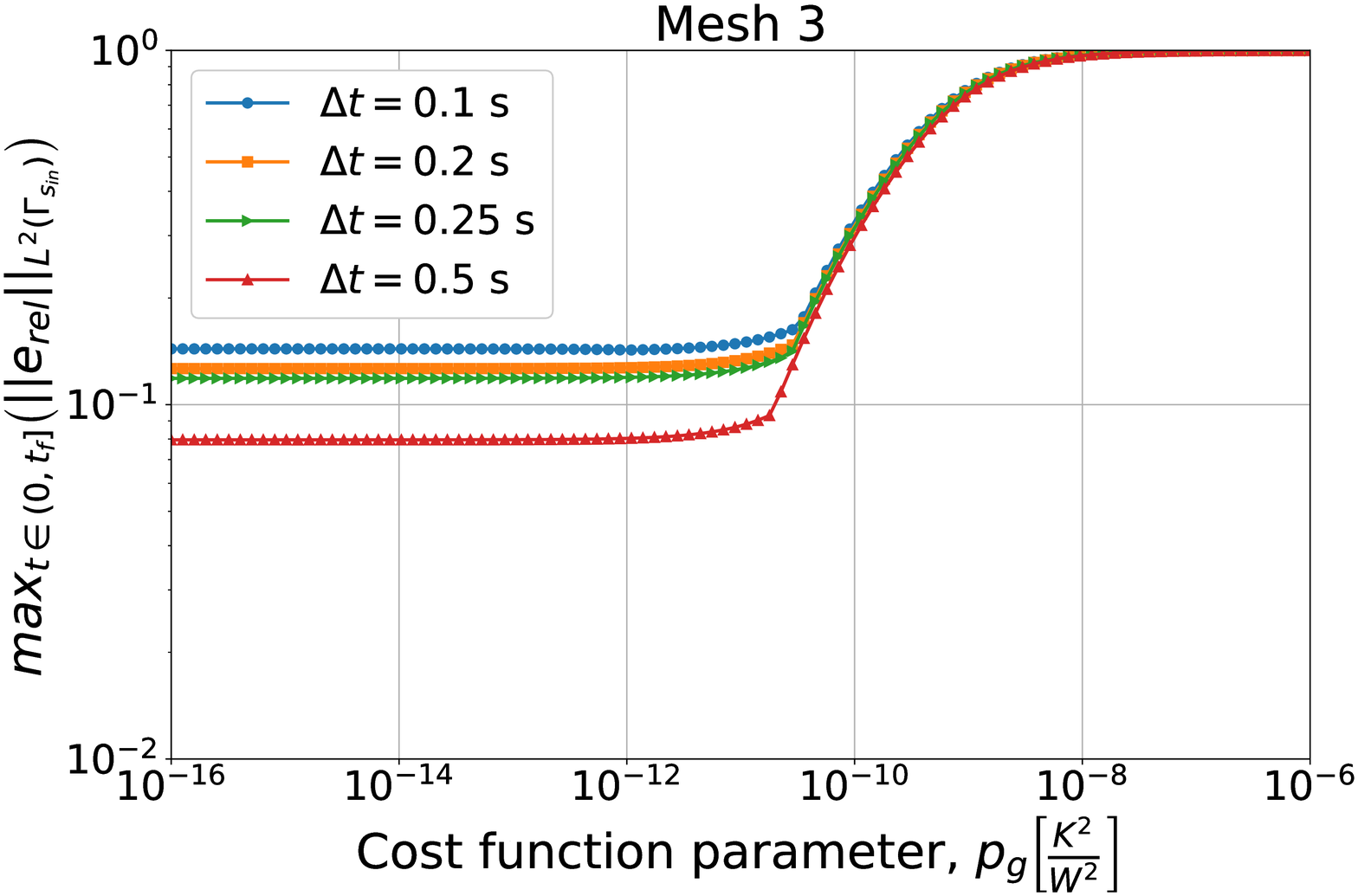}
            \caption{Max. relative error norm.}
        \end{subfigure}%
    \end{subfigure}%
    \caption{Mean (a) and maximum (b) values of the $L^2$-norm of the relative error, $e_{rel}$, in the interval $(0, t_f]$, for Benchmark 2 as the value of the cost function parameter, $p_g$, changes for Algorithm~\ref{alg:inverseSolver_constant_heatNorm} (piecewise constant time approximation of the heat flux).
    We show the results for Mesh 3 and different $\Delta t$.}
    \label{fig:unsteadyNumericalBenchmarkInverseNonLinear_constant_costFunction}
\end{figure}

\begin{figure}[!htb]
    \begin{subfigure}{\linewidth}
        \centering
        \begin{subfigure}[c]{.4\linewidth}
            \captionsetup{width=.85\linewidth}
            \includegraphics[width=.95\textwidth]{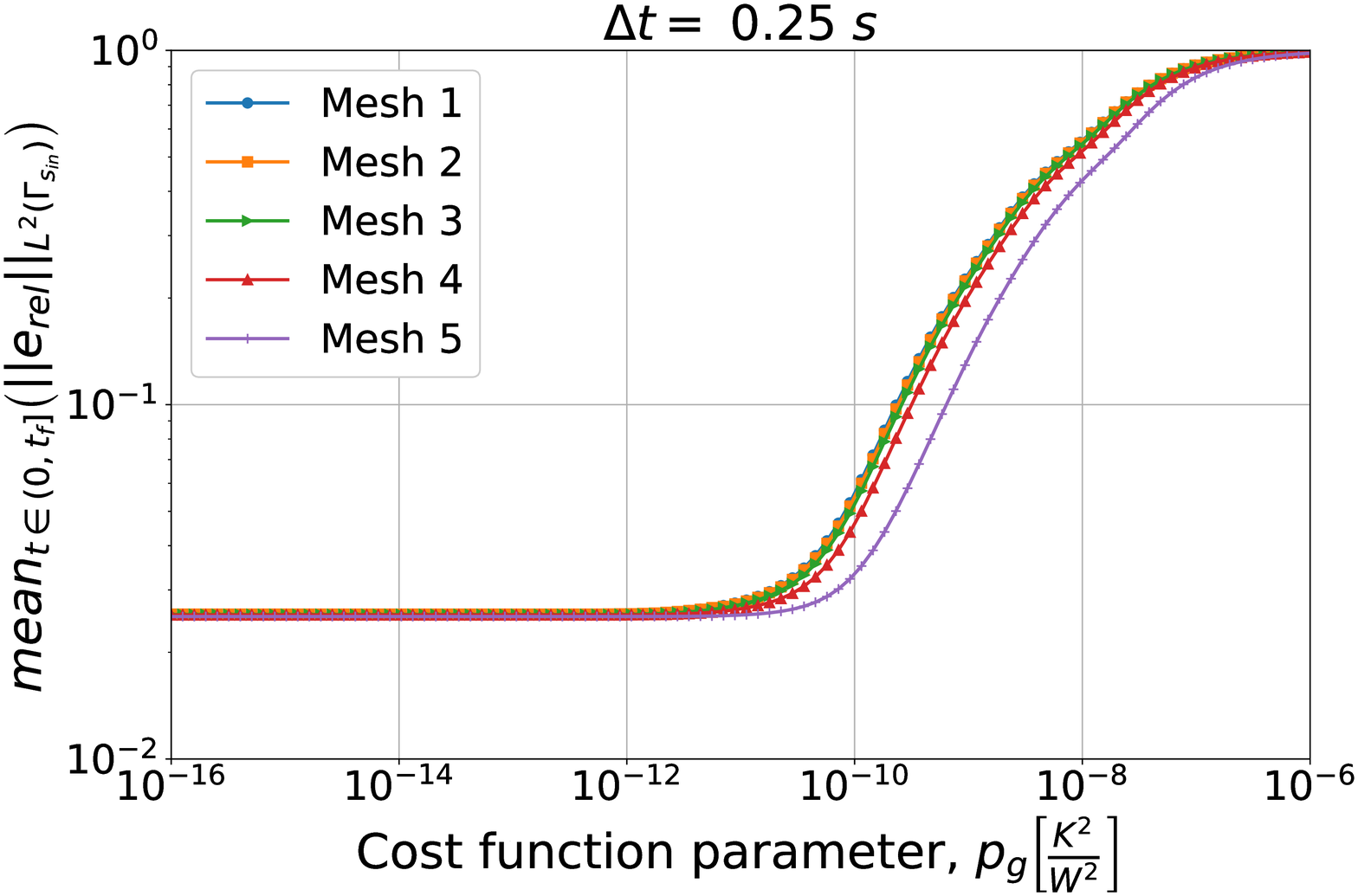}
            \caption{Mean relative error norm.}
        \end{subfigure}%
        \begin{subfigure}[c]{.4\linewidth}
            \captionsetup{width=.85\linewidth}
            \includegraphics[width=.95\textwidth]{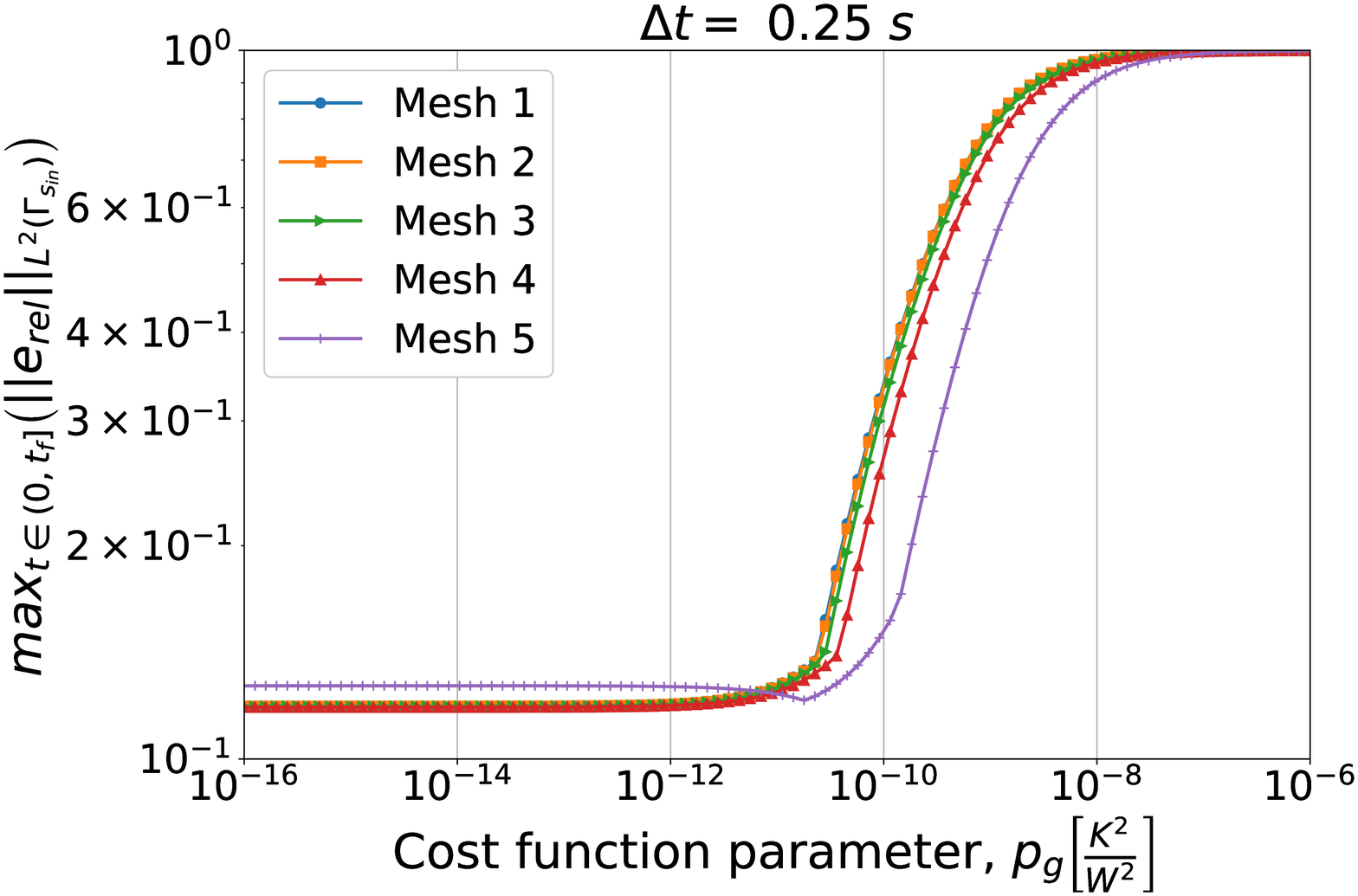}
            \caption{Max. relative error norm.}
        \end{subfigure}%
    \end{subfigure}%
    \caption{Mean (a) and maximum (b) values of the $L^2$-norm of the relative error, $e_{rel}$, in the interval $(0, t_f]$, for Benchmark 2 as the value of the cost function parameter, $p_g$, changes for Algorithm~\ref{alg:inverseSolver_constant_heatNorm} (piecewise constant time approximation of the heat flux).
    We show the results for $\Delta t = 0.25s$ and different meshes.}
    \label{fig:unsteadyNumericalBenchmarkInverseNonLinear_constant_costFunction_meshes}
\end{figure}

From the presented results, we notice that Algorithm~\ref{alg:inverseSolver_constant_heatNorm} has the same behavior shown in the previous benchmark case.
In particular, this inverse solver confirms to be badly affected by $p_g > 0 \frac{K^2}{W^2}$.
The effect of $p_g$ on its performance is very non linear with a first region of no effects for $p_g \lesssim 1e-11 \frac{K^2}{W^2}$ followed by a steep degradation and a plateau at $100\%$ relative error for $p_g \gtrsim 1e-7 \frac{K^2}{W^2}$.
Moreover, the different meshes and $\Delta t$ have the same behavior and similar values of the relative error.
These results confirm once more the almost insensibility of this inverse solver to the discretization refinement.

Now, we consider the piecewise linear approximation of Algorithm~\ref{alg:inverseSolver_linear_heatNorm}.
For this inverse solver, Figures~\ref{fig:unsteadyNumericalBenchmarkInverseNonLinear_linear_costFunction} and \ref{fig:unsteadyNumericalBenchmarkInverseNonLinear_linear_costFunction_meshes} show the mean and maximum values of the $L^2$-norm of the relative error, $e_{rel}$, in the interval $(0, t_f]$ as functions of $p_g$ changing the mesh and the $\Delta t$, respectively.

\begin{figure}[!htb]
    \begin{subfigure}{\linewidth}
        \centering
        \begin{subfigure}[c]{.4\linewidth}
            \captionsetup{width=.85\linewidth}
            \includegraphics[width=.95\textwidth]{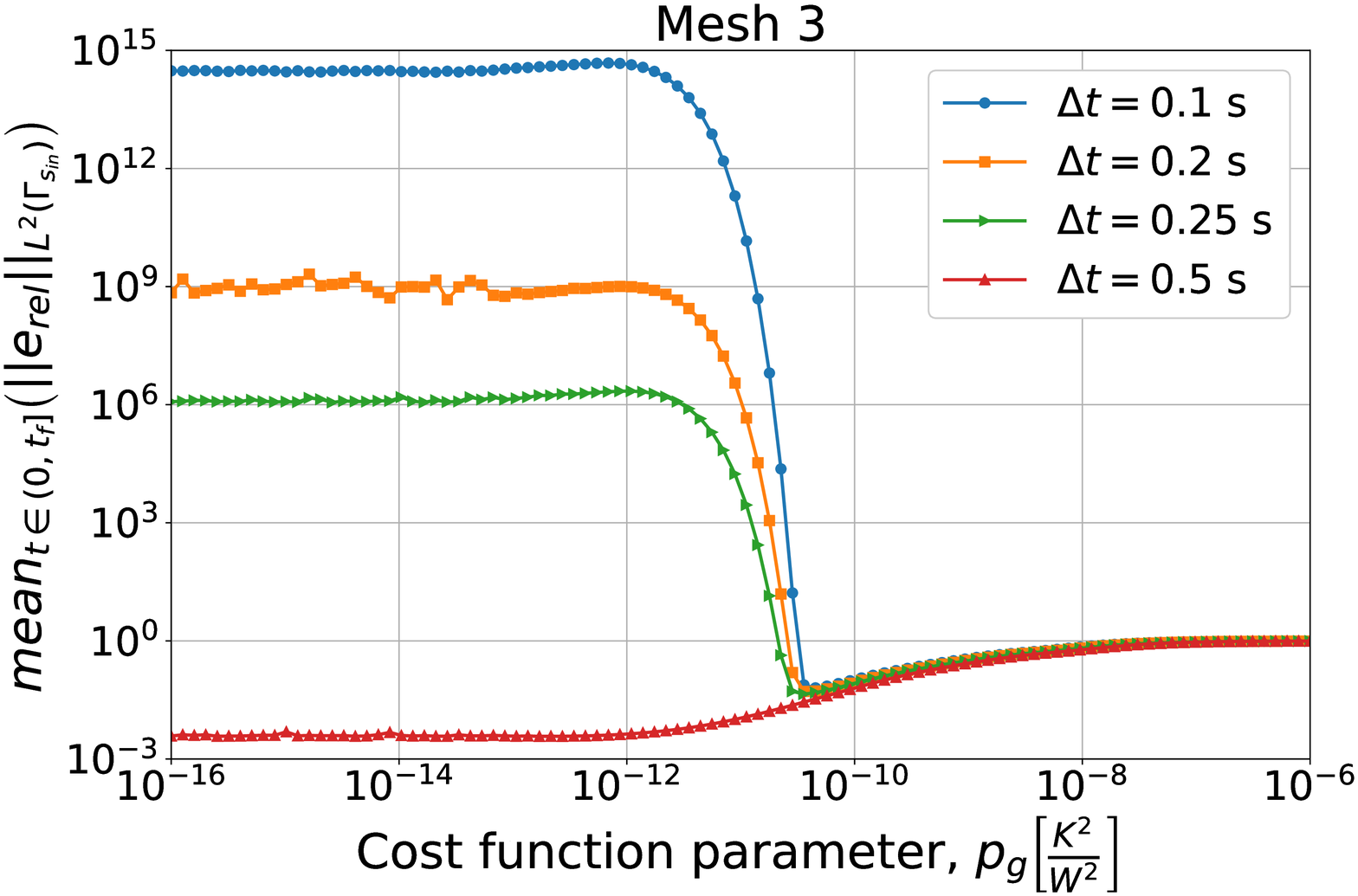}
            \caption{Mean relative error norm.}
        \end{subfigure}%
        \begin{subfigure}[c]{.4\linewidth}
            \captionsetup{width=.85\linewidth}
            \includegraphics[width=.95\textwidth]{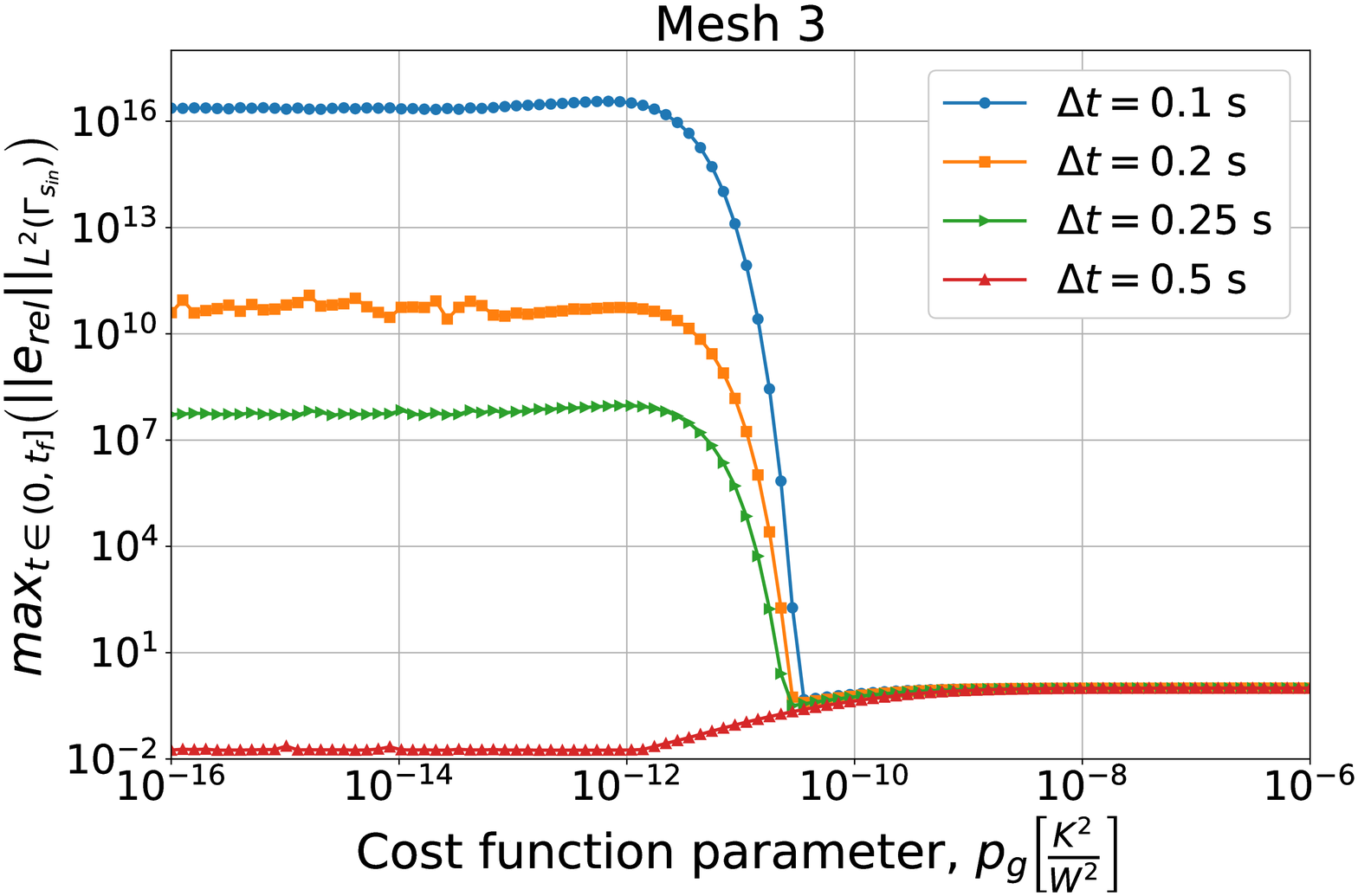}
            \caption{Max. relative error norm.}
        \end{subfigure}%
    \end{subfigure}%
    \caption{Mean (a) and maximum (b) values of the $L^2$-norm of the relative error, $e_{rel}$, in the interval $(0, t_f]$, for Benchmark 2 as the value of the cost function parameter, $p_g$, changes for Algorithm~\ref{alg:inverseSolver_linear_heatNorm} (piecewise linear time approximation of the heat flux).
    We show the results for Mesh 3 and different $\Delta t$.}
    \label{fig:unsteadyNumericalBenchmarkInverseNonLinear_linear_costFunction}
\end{figure}

\begin{figure}[!htb]
    \begin{subfigure}{\linewidth}
        \centering
        \begin{subfigure}[c]{.4\linewidth}
            \captionsetup{width=.85\linewidth}
            \includegraphics[width=.95\textwidth]{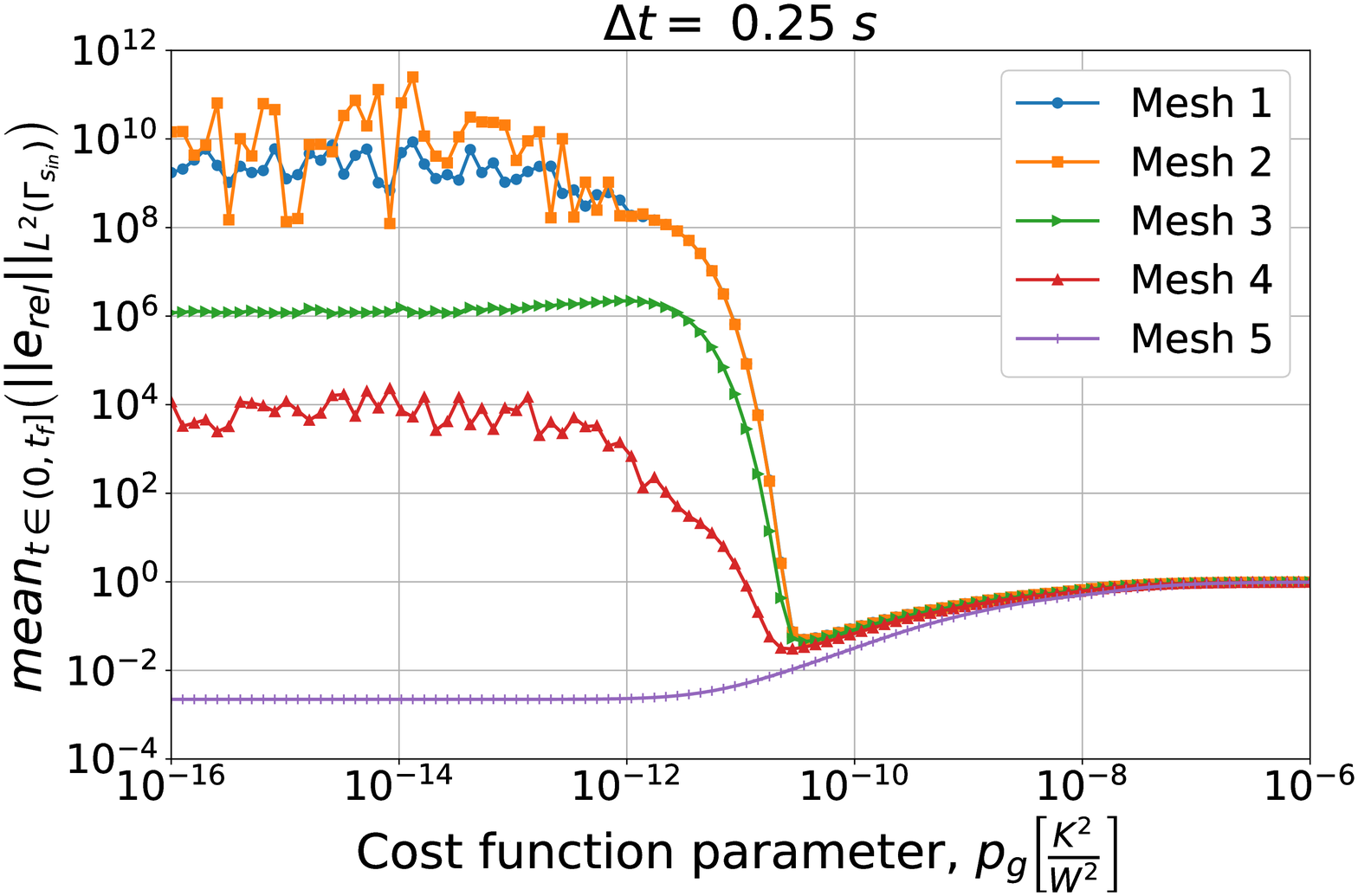}
            \caption{Mean relative error norm.}
        \end{subfigure}%
        \begin{subfigure}[c]{.4\linewidth}
            \captionsetup{width=.85\linewidth}
            \includegraphics[width=.95\textwidth]{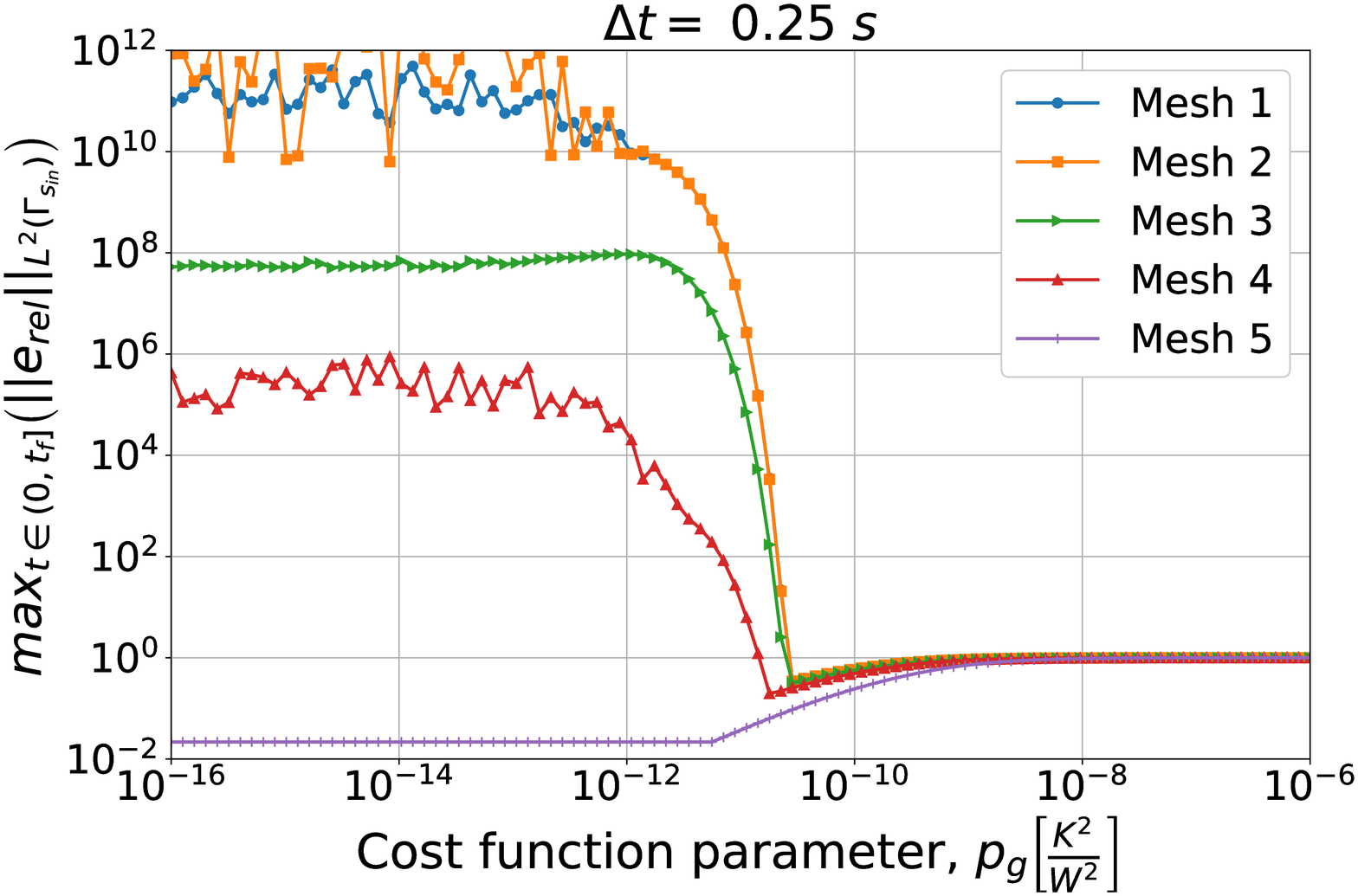}
            \caption{Max. relative error norm.}
        \end{subfigure}%
    \end{subfigure}%
    \caption{Mean (a) and maximum (b) values of the $L^2$-norm of the relative error, $e_{rel}$, in the interval $(0, t_f]$, for Benchmark 2 as the value of the cost function parameter, $p_g$, changes.
    The results are obtained using Algorithm~\ref{alg:inverseSolver_linear_heatNorm} (piecewise linear time approximation of the heat flux).
    We show the results for $\Delta t = 0.25s$ and different meshes.}
    \label{fig:unsteadyNumericalBenchmarkInverseNonLinear_linear_costFunction_meshes}
\end{figure}

Again, the results are very similar to those of the previous benchmark.
On one hand, the coarsest discretizations show a similar behavior to the piecewise constant approximation case with a monotonic degradation of the performance as $p_g$ increases.
On the other hand, we have unstable solutions for small $p_g$ that are stabilized by $p_g \gtrsim 5e-11 \frac{K^2}{W^2}$.
However, the accuracy of these solution rapidly decreases as we further increase $p_g$ until we reach the $100\%$ relative error plateau.

To conclude, we test the mesh, $\Delta t$ and $p_g$ selection method of Algorithm~\ref{alg:meshSelection}.
In this test, we input to the algorithm the virtual thermocouples measurements that we compute for this benchmark case.
Before presenting the results of this algorithm, we show in Figure~\ref{fig:unsteadyNumericalBenchmarkInverseNonLinear_linear_costFunction_measDiscrepancy} the mean value of $S_1^k$, $m_S$, as function of $p_g$ for different meshes and $\Delta t$.

\begin{figure}[!htb]
    \begin{subfigure}{\linewidth}
        \centering
        \begin{subfigure}[c]{.4\linewidth}
            \captionsetup{width=.85\linewidth}
            \includegraphics[width=.95\textwidth]{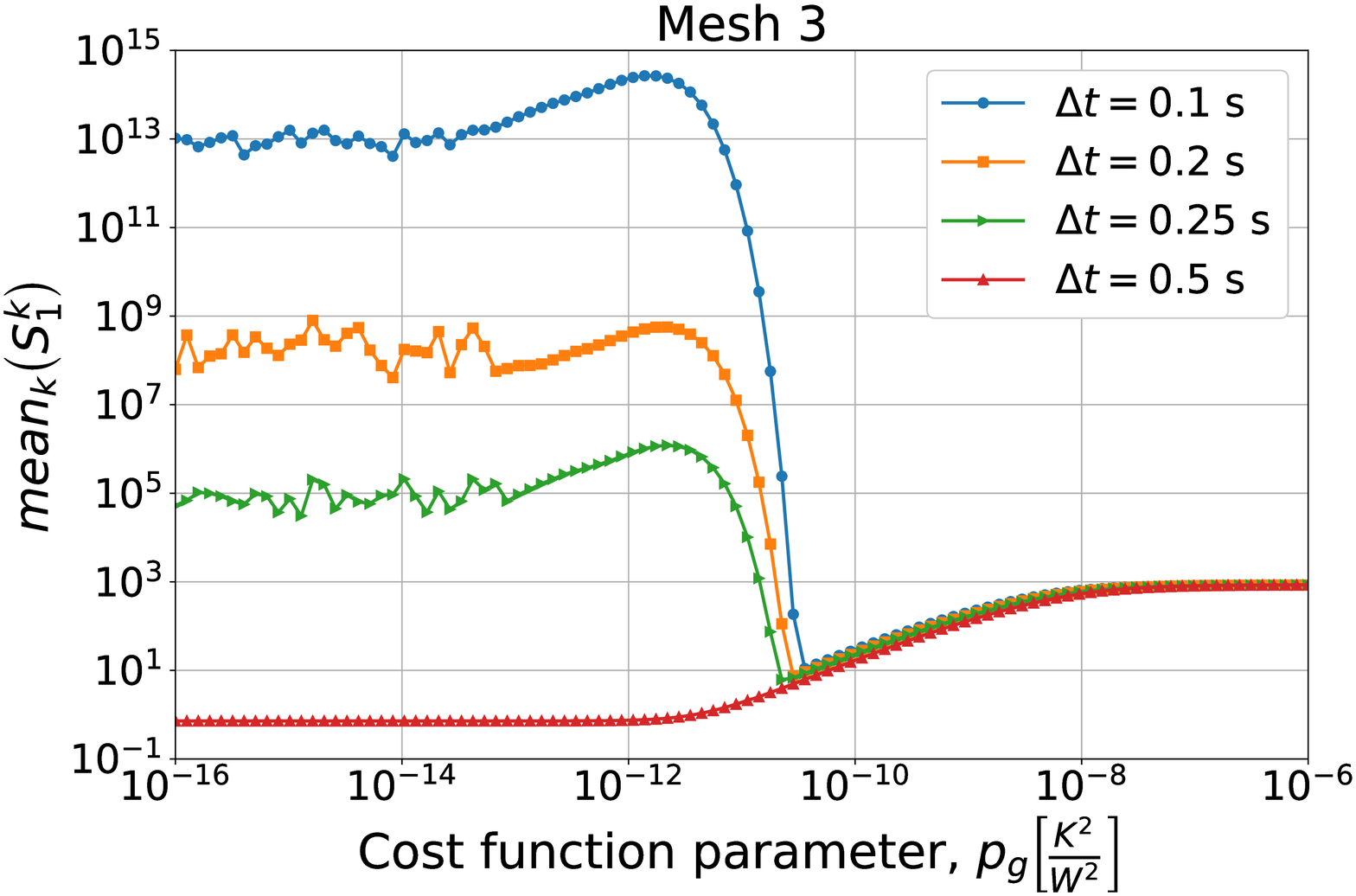}
            \caption{Mean of $S_1^k$ for Mesh 3.}
        \end{subfigure}%
        \begin{subfigure}[c]{.4\linewidth}
            \captionsetup{width=.85\linewidth}
            \includegraphics[width=.95\textwidth]{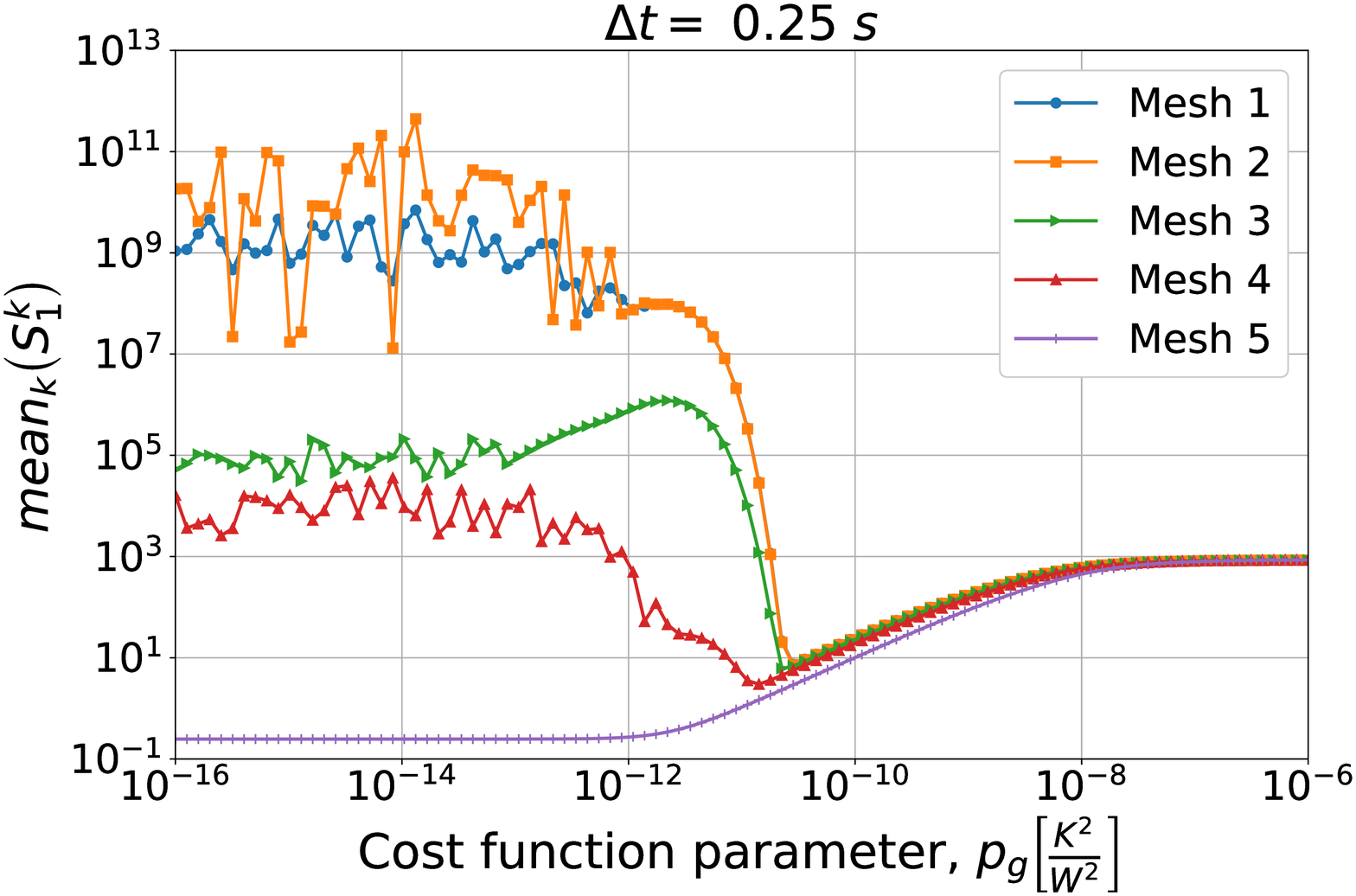}
            \caption{Mean of $S_1^k$ for $\Delta t = 0.25s$.}
        \end{subfigure}%
    \end{subfigure}%
    \caption{Mean values of $S_1^k$ for $1 \leq k \leq P_t$, for Benchmark 2 as the value of the cost function parameter, $p_g$, changes.
    The results are obtained using Algorithm~\ref{alg:inverseSolver_linear_heatNorm} (piecewise linear time approximation of the heat flux).
    We show the results for Mesh 3 and different $\Delta t$ in (a), and for $\Delta t = 0.25$ and different meshes in (b).}
    \label{fig:unsteadyNumericalBenchmarkInverseNonLinear_linear_costFunction_measDiscrepancy}
\end{figure}

Notice that $m_S$ has a behavior that is very similar to the relative error norm shown in Figures~\ref{fig:unsteadyNumericalBenchmarkInverseNonLinear_linear_costFunction} and \ref{fig:unsteadyNumericalBenchmarkInverseNonLinear_linear_costFunction_meshes}.
Moreover, its minima are almost correspondent to the relative error ones.
As already mentioned, we used these results as a guideline in developing Algorithm~\ref{alg:meshSelection}.
We present its behavior for the present benchmark case in Table~\ref{tab:unsteadyNumericalBenchmarkInverse_nonLinear_meshSelection}.

\begin{table}[htb]
\centering
    \caption{Test of Algorithm~\ref{alg:meshSelection} for Benchmark 2.}
    \label{tab:unsteadyNumericalBenchmarkInverse_nonLinear_meshSelection}
    \begin{tabular}{ |l|c|c|c|c|}
        \hline
        \textbf{Iteration}    &   \textbf{Mesh}    &   $\Delta t~[s]$    &   $p_g~\left[ \frac{K^2}{W^2} \right]$    &   $mean_k \left( S_1^k \right)~\left[ K^2 \right]$\\
        \hline
        0   &   4   &   $0.1$ & $8e-7$      &    $8.46e2 $\\   
        1   &   5   &   $0.5$ & $3.6e-11$   &    $6.2e0$\\   
        2   &   5   &   $0.5$ & $3.2e-13$   &    $5.5e-1$\\   
        \hline
    \end{tabular}
\end{table}

In this case, the algorithm does not select since the beginning the coarsest discretization.
In the first iteration, it finds the $p_g$ that minimizes $m_S$ for this setup.
Then, when comparing it to the other discretizations in step~\ref{alg:meshSelection_chooseMeshDt}, it selects again the coarsest one.
Also in this benchmark case, this algorithm showed to be able  to select the discretization setup and the value of $p_g$ corresponding to the best performance of the inverse solver.

%%%%%%%%%%%%%%%%%%%%%%%%%%%%%%%%%%%%%%%%%%%%%%%%%%%%%%%%%%%%%%%%%%%%%%%%%%%%%%%%%
\subsubsection{Effects of Measurements Noise and Regularization}\label{sec:unsteadyBenchmark_nonlinear_FOM_noise}

In this section, we test the effect that adding noise to the measurements vector, $\hat{\mathbf{T}}$, has in the performances of Algorithm~\ref{alg:inverseSolver_constant} and \ref{alg:inverseSolver_linear}.
From the industrial point of view, this analysis is of particular interest for our application since in the real case, thermocouples measurements are affected by noise.

We perform this analysis by adding to the measurements vector the Gaussian random noise $\pmb{\eta} = \mathcal{N}(\pmb{\mu}, \Sigma)$, where $\pmb{\mu} \in \R^M$ is the mean vector and $\Sigma \in \M^{M \times M}$ is the covariance matrix.
Then, we have
\begin{equation}
    \hat{\mathbf{T}}^k_\eta = \hat{\mathbf{T}}^k + \pmb{\eta}.
\end{equation}
In particular, we choose $\pmb{\eta}$ to be a independent and identically distributed random variable with zero mean, i.e. $\pmb{\eta} = \mathcal{N}(\mathbf{0}, \omega^2 I)$ , where $\omega$ denotes the noise standard deviation.
To study the effect of noise, we perform several solutions of the inverse problem using $\hat{\mathbf{T}}^k_\eta$ as thermocouples measurements. 
For each test, we compute 200 samples.

We show in Figures~\ref{fig:benchmark2_nonlinear_noiseEffect_constant} and \ref{fig:benchmark2_nonlinear_noiseEffect_linear} the obtained results for the piecewise constant and linear algorithm, respectively. 
In particular, we illustrate for each of them the mean values over the samples of the mean and maximum of the relative error (\ref{eq:unsteadyBenchmark_relativeError}) in $(0,t_f]$ (with 90\% quantile bars) for different values of the noise standard deviation, $\omega$.
The figure compares the results obtained using LU with full pivoting and TSVD with different values for the regularization parameter $\alpha_{TSVD}$.

\begin{figure}[!htb]
    \begin{subfigure}{\linewidth}
        \centering
        \begin{subfigure}[c]{.4\linewidth}
            \includegraphics[width=.95\textwidth]{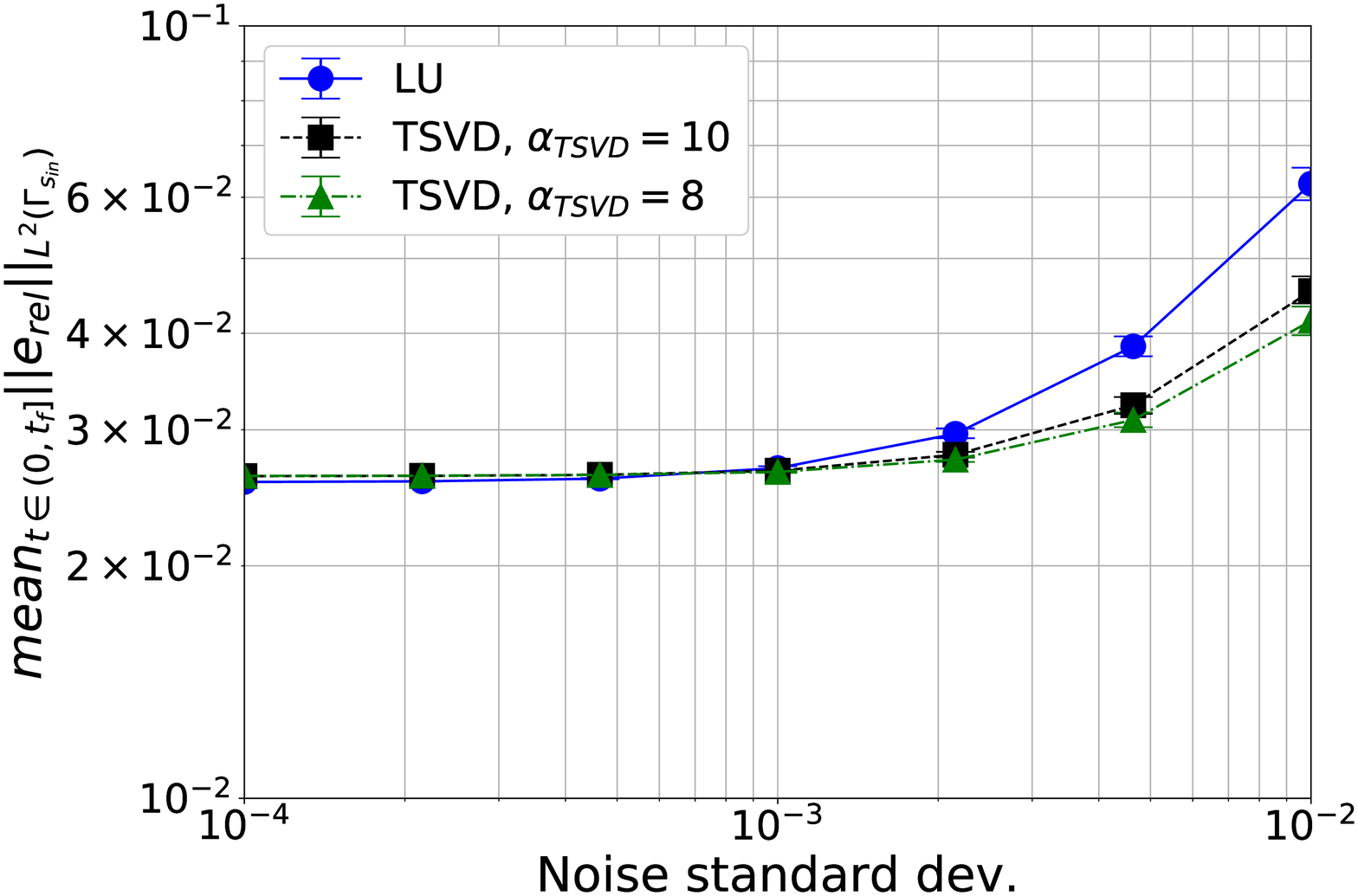}
            \caption{Mean.}
        \end{subfigure}%
        \begin{subfigure}[c]{.4\linewidth}
            \includegraphics[width=.95\textwidth]{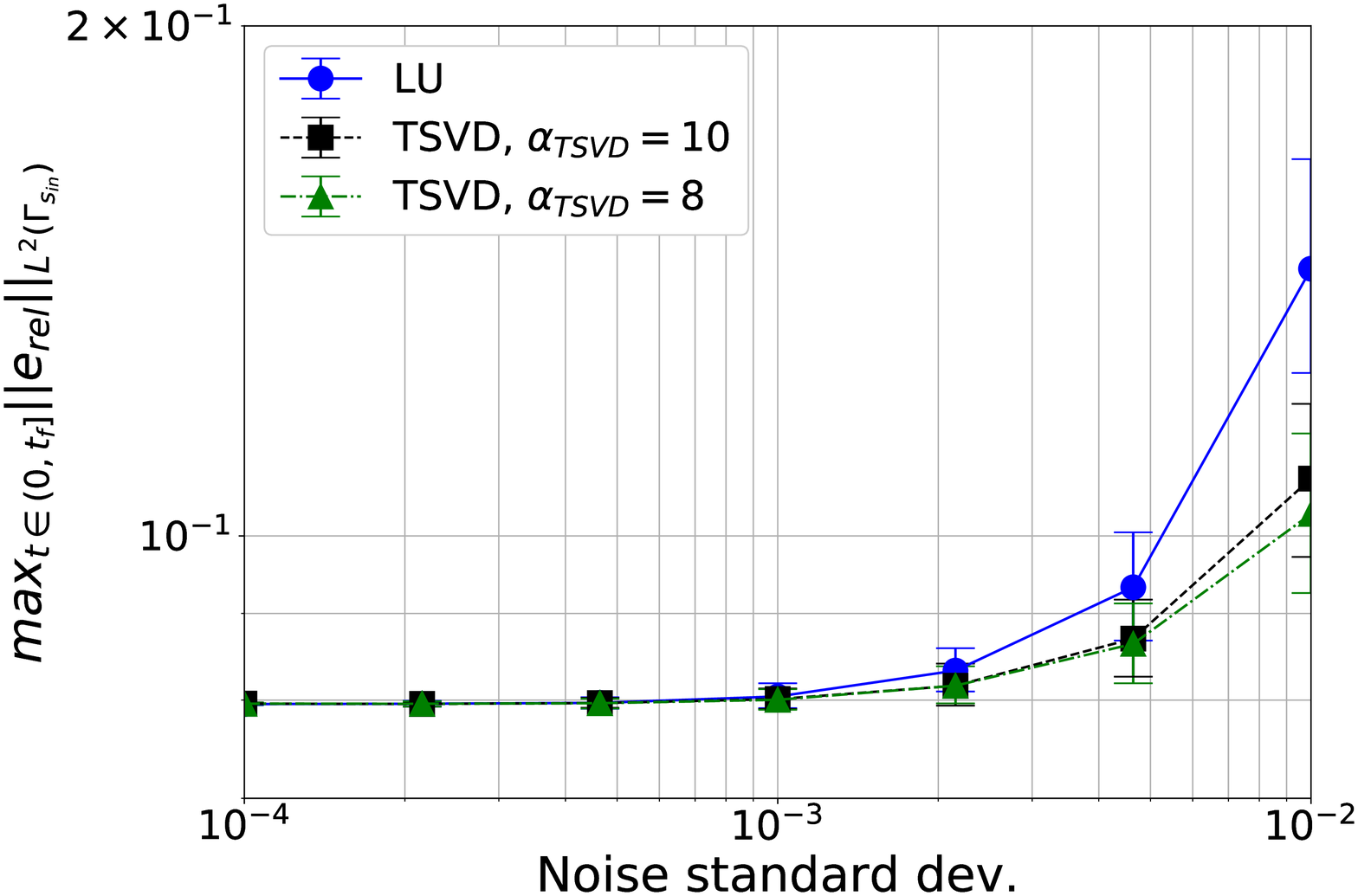}
            \caption{Maximum.}
        \end{subfigure}%
    \end{subfigure}%
    \caption{Effect of the noise in the temperature measurements for Algorithm~\ref{alg:inverseSolver_constant} in Benchmark 2.
    In the figures, we show the mean (a) and maximum (b) values of the relative error (\ref{eq:unsteadyBenchmark_relativeError}) in $(0,t_f]$ for different values of the noise standard deviation and using both LU with full pivoting and TSVD for the solution of inverse problem linear system (\ref{eq:linSys_parametrizedBC_sequential_constant}).
    For each case, we performed 200 runs.
    The markers show the mean values while the bars are the 90\% quantiles. 
    In these computations, we considered $p_g = 0 \frac{K^2}{W^2}$.}
    \label{fig:benchmark2_nonlinear_noiseEffect_constant}
\end{figure}

\begin{figure}[!htb]
    \begin{subfigure}{\linewidth}
        \centering
        \begin{subfigure}[c]{.4\linewidth}
            \includegraphics[width=.95\textwidth]{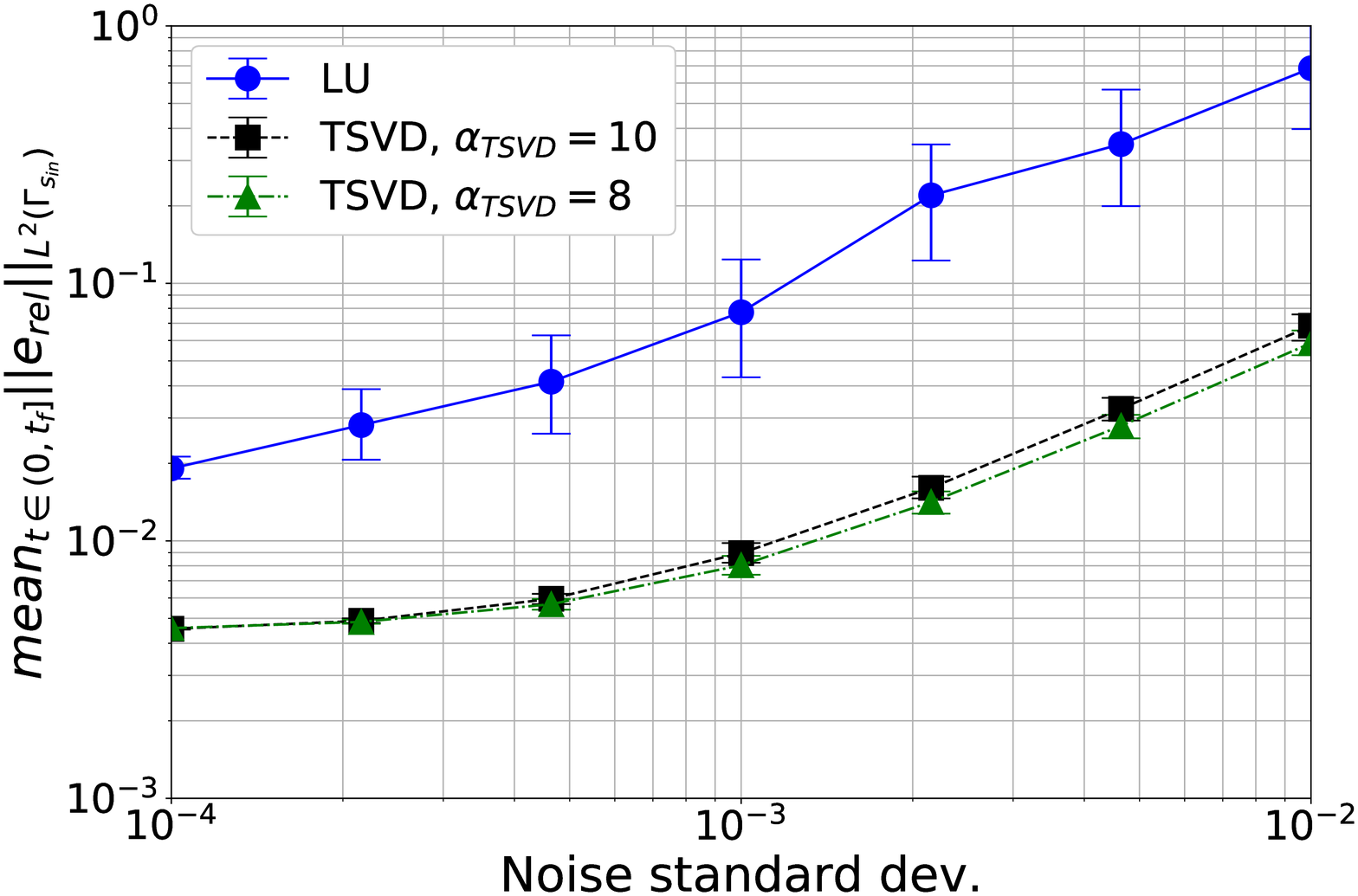}
            \caption{Mean.}
        \end{subfigure}%
        \begin{subfigure}[c]{.4\linewidth}
            \includegraphics[width=.95\textwidth]{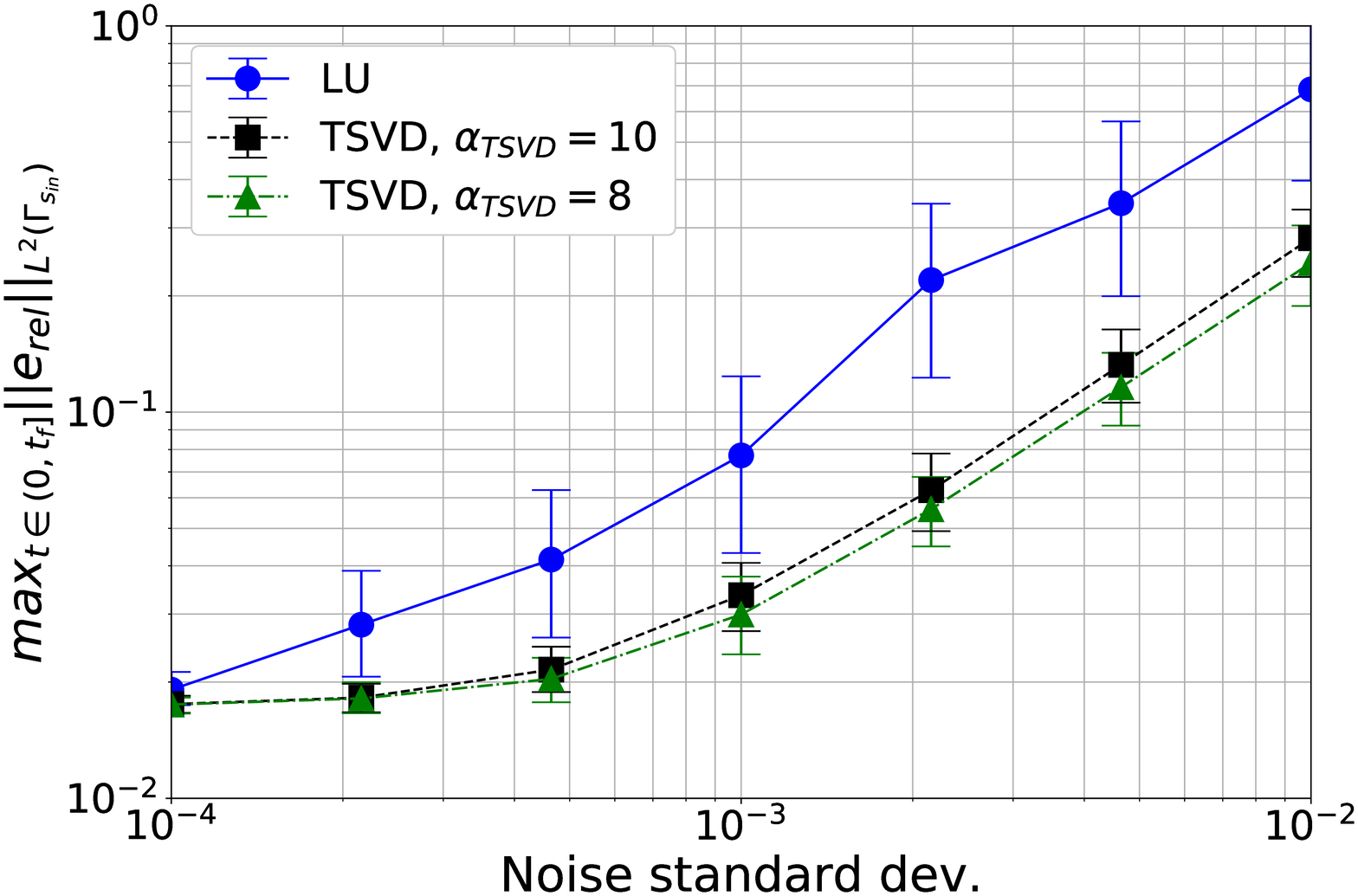}
            \caption{Maximum.}
        \end{subfigure}%
    \end{subfigure}%
    \caption{Effect of the noise in the temperature measurements for Algorithm~\ref{alg:inverseSolver_linear} in Benchmark 2.
    In the figures, we show the mean (a) and maximum (b) values of the relative error (\ref{eq:unsteadyBenchmark_relativeError}) in $(0,t_f]$ for different values of the noise standard deviation and using both LU with full pivoting and TSVD for the solution of inverse problem linear system (\ref{eq:linSys_parametrizedBC_sequential_linear}).
    For each case, we performed 200 runs.
    The markers show the mean values while the bars are the 90\% quantiles. 
    In these computations, we considered $p_g = 0 \frac{K^2}{W^2}$.}
    \label{fig:benchmark2_nonlinear_noiseEffect_linear}
\end{figure}

The results show a very different dependency from the measurement noise in the two algorithm.
The piecewise constant Algorithm~\ref{alg:inverseSolver_constant} shows in Figure~\ref{fig:benchmark2_nonlinear_noiseEffect_constant} to be quite robust with respect to these levels of noise.
Moreover, the TSVD is effective in reducing the noise propagation and we are able to keep a reasonable level of accuracy.

On the other hand, the piecewise linear Algorithm~\ref{alg:inverseSolver_linear} is much more affected by the noise.
By using the TSVD regularization, we have an improvement of the noise robustness.
However, the error rate of increase is much higher than for the piecewise constant solver.

%%%%%%%%%%%%%%%%%%%%%%%%%%%%%%%%%%%%%%%%%%%%%%%%%%%%%%%%%%%%%%%%%%%%%%%%%%%%%%%%%
\subsubsection{Computational Cost}

To conclude this analysis, we present in Tables~\ref{tab:computationalCostConstant} and \ref{tab:computationalCostLinear} the numerical cost of performing one iteration of the proposed inverse solvers.
Notice that all the computations were performed in serial on a Intel\textsuperscript{\textregistered} Core\textsuperscript{\texttrademark} i7-8550U CPU processor.
As expected the required CPU time increases with the refinement of the discretization.
Since we are using relative coarse meshes due to the simplified geometry, the computational cost in many cases meets the real-time requirement for this application (i.e. $1$ s).
However, the meshes required for the discretization of the real mold geometry are such that we cannot ensure real-time performances in these cases.

\begin{table}[htb]
\centering
    \caption{Average computational cost for one iteration of Algorithm~\ref{alg:inverseSolver_constant} in Benchmark 2.}
    \label{tab:computationalCostConstant}
    \begin{tabular}{ |l|c|c|c|c|c|}
        \hline
        \diagbox{$\Delta t$}{\textbf{Mesh}}   &   \textbf{Mesh 1}    &   \textbf{Mesh 2}    &   \textbf{Mesh 3}    &   \textbf{Mesh 4}    &   \textbf{Mesh 5}    \\    
        \hline
        $0.1$ s &   $5184.0$ ms    &   $1402.6$ ms   &  $815.6$ ms    &  $482.8$ ms    &  $343.4$ ms    \\   
        $0.2$ s &   $2520.1$ ms    &   $741.2$ ms    &  $427.4$ ms    &  $260.2$ ms    &  $192.6$ ms    \\
        $0.25$ s &  $2052.9$ ms    &   $600.6$ ms    &  $350.9$ ms    &  $215.8$ ms    &  $162.1$ ms    \\
        $0.5$ s &   $1115.5$ ms    &   $333.4$ ms    &  $197.6$ ms    &  $128.2$ ms    &  $101.0$ ms    \\
        \hline
    \end{tabular}
\end{table}

\begin{table}[htb]
\centering
    \caption{Average computational cost for one iteration of Algorithm~\ref{alg:inverseSolver_linear} in Benchmark 2.}
    \label{tab:computationalCostLinear}
    \begin{tabular}{ |l|c|c|c|c|c|}
        \hline
        \diagbox{$\Delta t$}{\textbf{Mesh}}   &   \textbf{Mesh 1}    &   \textbf{Mesh 2}    &   \textbf{Mesh 3}    &   \textbf{Mesh 4}    &   \textbf{Mesh 5}    \\    
        \hline
        $0.1$ s &   $5966.4$ ms    &   $1621.2$ ms   &  $1442.5$ ms   &  $517.3$ ms    &  $400.1$ ms    \\   
        $0.2$ s &   $2836.2$ ms    &   $907.4$ ms    &  $518.1$ ms    &  $293.4$ ms    &  $206.2$ ms    \\
        $0.25$ s &  $2258.8$ ms    &   $707.6$ ms    &  $382.5$ ms    &  $232.4$ ms    &  $182.9$ ms    \\
        $0.5$ s &   $1195.8$ ms    &   $435.7$ ms    &  $216.8$ ms    &  $140.2$ ms    &  $108.0$ ms    \\
        \hline
    \end{tabular}
\end{table}

%%%%%%%%%%%%%%%%%%%%%%%%%%%%%%%%%%%%%%%%%%%%%%%%%%%%%%%%%%%%%%%%%%%%%%%%%%%%%%%%%
\section{Conclusions and Future Works}
\label{section:conclusions}
%%%%%%%%%%%%%%%%%%%%%%%%%%%%%%%%%%%%%%%%%%%%%%%%%%%%%%%%%%%%%%%%%%%%%%%%%%%%%%%%%

The goal of the present investigation was to develop mathematical tools to monitor the mold behavior in CC machineries. 
At industrial experts' suggestion, we identified the mold-steel heat flux as the quantity of interest for the mold monitoring. 
Then, our objective has been to investigate a methodology for the real-time estimation of this heat flux having as data the physical properties of the mold, its geometry, the cooling
water temperature and some pointwise temperature measurements in the interior of the mold plates.

We opted for stating the problem in a data assimilation, optimal control setting in which we look for the heat flux that minimizes a functional that includes a measure
of the distance between the computed and measured temperature at the measurement points. 
However, given a mold-steel heat flux, we required a mold model to compute the temperature at these points.

In deriving the mold model, we had to take into account the real-time requirement of this mold monitoring task. 
Then, we avoided to model all the mold region, including the steel pool. 
In fact, it would require us to model several complex and coupled physical phenomena (heat transfer, steel solidification, steel and cooling water fluid flows, etc.) happening on different time and space scales. 
With this level of complexity, the computational cost of such simulation would have been unsustainable for real-time computations. 
Thus, in modeling the mold, we considered as computational domain the mold plates only. 
Hence, our first task has been the derivation of the three dimensional unsteady-state heat conduction mold model.

Once the mold model has been established, we focused on the mold-steel heat flux estimation. 
Notice that, in this setting, this flux is a Neumann BC on a portion of the boundary of our domain. 
Then, we can generalize this mathematical inverse problem as the estimation of a Neumann BC given pointwise state measurements in the interior of the domain.

In this unsteady-state setting, we used a sequential approach to the inverse problem. 
In fact, to provide a real-time solution in this setting means to stay always at the front of the time line as it stretches. 
Then, since our measurements come equally spaced in time by one second, we considered the problem of estimating the heat flux only in between the last measurement and the previous one, assuming to have already the solution for older times.

In this framework, we stated two different inverse problems. 
One looking for the heat flux that minimizes a measure of the distance between computed and measured temperature only. 
While, in the other, we want to minimize this distance plus an heat flux norm.

For both these inverse problems, we developed novel methodologies for their solution that exploit a RBFs parameterization of the heat flux in space with time dependent coefficients.
With respect to these coefficients, we considered both the piecewise constant and the piecewise linear case. 
It means that the estimated heat flux is constant or linear in between two contiguous measurement instants.

These novel methodologies are direct methods that benefit from an offline-online decomposition. 
Thanks to this decomposition, we have a first computationally expensive offline phase, in which we solve several direct problems. 
This offline phase is computed once and for all and does not require any measurement. 
Then, when the caster starts to work, we only have to collect the thermocouples measurements and run the online phase which is computationally much cheaper.

To conclude, we tested the proposed inverse solvers on some benchmark cases.
We designed two benchmark cases for the inverse problem. 
To design an inverse problem test, we arbitrarily selected a mold-steel heat flux. 
We solved the direct problem for this heat flux obtaining the corresponding temperature field in the mold. 
Finally, we located some virtual thermocouples and used the computed temperature at these points as input data for the inverse solvers. 
Then, we tested their ability to reconstruct the previously selected heat flux.
The two benchmark cases share the same domain and physical parameters which were chosen to resemble the industrial setting. 
The difference is in the selected heat fluxes.

From the obtained results, we noticed a great difference in the behavior of the piecewise constant and linear inverse solvers. 
The former showed a very stable behavior and insensitivity to the time and space discretization used. 
The latter, rather, is very much influenced by the discretization used. 
In particular, it can be very unstable when using fine discretizations but this instability is reduced by coarsening the time and/or space discretization. 
In fact, for some discretizations, we achieved very stable and accurate solutions, eventually.

We also tested the effects that adding the heat flux norm to the minimization functional has on these inverse solvers. 
We implemented this new term multiplying it by a parameter. 
Then, we tested the effect that its value has on the solvers performance.

We noticed that the piecewise constant algorithm performance monotonically deteriorates as this parameter increases. 
The same goes for the piecewise linear solver when using the coarsest, stable discretizations. 
However, the unstable configurations showed to be positively affected by the addition of this new term and, for some values of this parameter, we were able to obtain stable and accurate solutions for all the tested discretizations. 
While, for too high values of the parameter, the solution is stable but inaccurate for all the meshes and timestep sizes. 
Moreover, we showed that, for values of the parameter above a threshold, the inverse solver performance is almost independent from the discretization.

Due to this dependency from the discretization and the functional parameter, we developed an algorithm for the automatic selection of these quantities.
In the numerical tests, it proved to be able of a selection that corresponds to a stable and accurate inverse solver.

Testing the inverse solvers for several noise levels showed again a different behavior between the piecewise constant and linear approximations.
The former is much less sensitive to the measurements noise than the latter.
For both, the TSVD regularization proved to be able to mitigate the noise propagation and we were able to obtain accurate and stable solutions also in the presence of noise.

To conclude, we recall that the online phases of the proposed algorithms require the solution of a full order problem whose computational cost depends on the mesh and timestep size.
As shown in the numerical tests section, it means that we cannot ensure real-time performance for these algorithms as they are.
Then, in our future work, we will develop model order reduction techniques that will allow us to reduce the online phase computational time and make it independent from the discretization.\cite{Strazzullo2022, Negri2013}

As a final remark, we discuss the application of the new proposed methodology to other problems. 
Recalling that the presented continuous casting problem is a Neumann BC estimation problem in a unsteady linear setting with pointwise state measurements in the interior of the domain, we can apply the proposed methodologies to any problem sharing these features. 
An example can be a boundary stress estimation problem in linear elasticity with pointwise deformation measurements.

Other possible future works on the subject could be related to the study of theoretical results that can ensure a priori the stability and accuracy of these inverse solvers with respect to the used discretization.
It would increase the potential of the proposed methodologies as well as their reliability. 
In particular, it would be useful for the final user to know a priori the time and space discretization to select as well as the minimization functional parameter. 
Notice that it is needed for the piecewise linear inverse solver because the piecewise constant one is almost insensible to the discretization refinement.

In the future, it would also be interesting to investigate the use of a completely different approach in the solution of this inverse problem. 
In particular, thinking about a more proper handling of the measurement noise, we could think of using a Bayesian approach.\cite{Stabile2020b} 
Techniques such as ensemble Kalman filter could be suitable for this problem given the sequentiality of the measurements. 
Moreover, considering the real-time requirement of the application, it would probably require an effective use of model order reduction techniques to reduce the demanding computational cost of these techniques.

\section{Acknowledge}
We would like to acknowledge the financial support of the European Union under the Marie Sklodowska-Curie Grant Agreement No. 765374.
We also acknowledge the partial support by the Ministry of Economy, Industry and Competitiveness through the Plan Nacional de I+D+i (MTM2015-68275-R), by the Agencia Estatal de Investigacion through project [PID2019-105615RB-I00/ AEI / 10.13039/501100011033], by the European Union Funding for Research and Innovation - Horizon 2020 Program - in the framework of European Research Council Executive Agency:  Consolidator Grant H2020 ERC CoG 2015 AROMA-CFD project 681447 "Advanced Reduced Order Methods with Applications in Computational Fluid Dynamics" and  INDAM-GNCS project "Advanced intrusive and non-intrusive model order reduction techniques and applications", 2019.
Moreover, we gratefully thank Gianfranco Marconi, Federico Bianco and Riccardo Conte from Danieli \& C.Officine Meccaniche SpA for helping us in better understanding the industrial problem and for the fruitful cooperation.

\bibliographystyle{ama}
\bibliography{references}%

\begin{thebibliography}{10}
\providecommand \doibase [0]{http://dx.doi.org/}%

\bibitem{WorldSteel2018}
{World Steel Association} . World steel in figures 2018. {\it World Steel
  Association: Brussels, Belgium} 2018.

\bibitem{Irving1993}
Irving WR. {\it Continuous casting of steel}.
\newblock The Institute of Materials (UK) .
\newblock 1993.

\bibitem{Morelli2021}
Morelli UE, Barral P, Quintela P, Rozza G, Stabile G. A numerical approach for
  heat flux estimation in thin slabs continuous casting molds using data
  assimilation. {\it International Journal for Numerical Methods in
  Engineering} 2021\string; 122(17)\string: 4541-4574.
\newblock \href {\doibase https://doi.org/10.1002/nme.6713} {doi:
  https://doi.org/10.1002/nme.6713}

\bibitem{Nittka2011b}
Nittka R. Inhomogeneous parabolic Neumann problems. {\it Czechoslovak
  Mathematical Journal} 2014\string; 64(3)\string: 703-742.
\newblock \href {\doibase https://doi.org/10.48550/arXiv.1108.6227} {doi:
  https://doi.org/10.48550/arXiv.1108.6227}

\bibitem{Eymard2000}
Eymard R, Gallou{\"{e}}t T, Herbin R. Finite volume methods. In: . 7 of {\it
  Handbook of Numerical Analysis}. Elsevier.  2000 (pp. 713 - 1018).

\bibitem{Ling2003}
Ling X, Keanini RG, Cherukuri HP. A non-iterative finite element method for
  inverse heat conduction problems. {\it International Journal for Numerical
  Methods in Engineering} 2003\string; 56(9)\string: 1315-1334.
\newblock \href {\doibase https://doi.org/10.1002/nme.614} {doi:
  https://doi.org/10.1002/nme.614}

\bibitem{Loulou2006}
Loulou T, Scott EP. An inverse heat conduction problem with heat flux
  measurements. {\it International Journal for Numerical Methods in
  Engineering} 2006\string; 67(11)\string: 1587-1616.
\newblock \href {\doibase https://doi.org/10.1002/nme.1674} {doi:
  https://doi.org/10.1002/nme.1674}

\bibitem{Jin2007}
Jin B. Conjugate gradient method for the Robin inverse problem associated with
  the Laplace equation. {\it International Journal for Numerical Methods in
  Engineering} 2007\string; 71(4)\string: 433-453.
\newblock \href {\doibase https://doi.org/10.1002/nme.1949} {doi:
  https://doi.org/10.1002/nme.1949}

\bibitem{Huang1996}
Huang CH, Yan JY. An Inverse Problem In Predicting Temperature Dependent Heat
  Capacity Per Unit Volume Without Internal Measurements. {\it International
  Journal for Numerical Methods in Engineering} 1996\string; 39(4)\string:
  605-618.
\newblock \href {\doibase
  https://doi.org/10.1002/(SICI)1097-0207(19960229)39:4<605::AID-NME872>3.0.CO;2-H}
  {doi:
  https://doi.org/10.1002/(SICI)1097-0207(19960229)39:4<605::AID-NME872>3.0.CO;2-H}

\bibitem{Alifanov1988}
Alifanov OM. {\it Inverse Heat Transfer Problems}.
\newblock Moscow Izdatel Mashinostroenie.
\newblock 1~ed. 1988.

\bibitem{Orlande2010}
Orlande HRB. {Inverse Problems in Heat Transfer: New Trends on Solution
  Methodologies and Applications}. {\it Journal of Heat Transfer} 2012\string;
  134(3)\string: 1-13.
\newblock 031011\href {\doibase 10.1115/1.4005131} {doi: 10.1115/1.4005131}

\bibitem{Beck1985}
Beck JV, Blackwell B, Clair~Jr. CRS. {\it Inverse heat conduction: Ill-posed
  problems}.
\newblock James Beck .
\newblock 1985.

\bibitem{Chang2017}
Chang CW, Liu CH, Wang CC. Review of Computational Schemes in Inverse Heat
  Conduction Problems. {\it Smart Science} 2018\string; 6(1)\string: 94-103.
\newblock \href {\doibase 10.1080/23080477.2017.1408987} {doi:
  10.1080/23080477.2017.1408987}

\bibitem{Ahin2006}
{\c{S}}ahin HM, Kocatepe K, Kay{\i}kc{\i} R, Akar N. Determination of
  unidirectional heat transfer coefficient during unsteady-state solidification
  at metal casting{\textendash}chill interface. {\it Energy Conversion and
  Management} 2006\string; 47(1)\string: 19--34.
\newblock \href {\doibase 10.1016/j.enconman.2005.03.021} {doi:
  10.1016/j.enconman.2005.03.021}

\bibitem{Ranut2012}
Ranut P. {\it Optimization and Inverse Problems in Heat Transfer}. PhD thesis.
  Universit\'a degli Studi di Udine, Via delle Scienze, 206, 33100 Udine UD,
  Italy;  2012.

\bibitem{Udayraj2017}
Udayraj , Chakraborty S, Ganguly S, Chacko E, Ajmani S, Talukdar P. Estimation
  of surface heat flux in continuous casting mould with limited measurement of
  temperature. {\it International Journal of Thermal Sciences} 2017\string;
  118\string: 435 - 447.
\newblock \href {\doibase https://doi.org/10.1016/j.ijthermalsci.2017.05.012}
  {doi: https://doi.org/10.1016/j.ijthermalsci.2017.05.012}

\bibitem{Mahapatra1991}
Mahapatra R, Brimacombe J, Samarasekera I. Mold behavior and its influence on
  quality in the continuous casting of steel slabs: Part II. Mold heat
  transfer, mold flux behavior, formation of oscillation marks, longitudinal
  off-corner depressions, and subsurface cracks. {\it Metallurgical and
  Materials Transactions B} 1991\string; 22(6)\string: 875--888.
\newblock \href {\doibase 10.1007/BF02651164} {doi: 10.1007/BF02651164}

\bibitem{Raymond2013}
Raymond JP. Optimal control of partial differential equations. {\it
  Universit{\'e} Paul Sabatier, Internet} 2013.

\bibitem{Vitale2012}
Vitale G, Preziosi L, Ambrosi D. Force traction microscopy: An inverse problem
  with pointwise observations. {\it Journal of Mathematical Analysis and
  Applications} 2012\string; 395(2)\string: 788 - 801.
\newblock \href {\doibase https://doi.org/10.1016/j.jmaa.2012.05.074} {doi:
  https://doi.org/10.1016/j.jmaa.2012.05.074}

\bibitem{Huang1998}
Huang CH, Chen CW. A boundary element-based inverse-problem in estimating
  transient boundary conditions with conjugate gradient method. {\it
  International Journal for Numerical Methods in Engineering} 1998\string;
  42(5)\string: 943-965.
\newblock \href {\doibase
  https://doi.org/10.1002/(SICI)1097-0207(19980715)42:5<943::AID-NME395>3.0.CO;2-V}
  {doi:
  https://doi.org/10.1002/(SICI)1097-0207(19980715)42:5<943::AID-NME395>3.0.CO;2-V}

\bibitem{Tikhonov1963}
Tikhonov AN. Solution of incorrectly formulated problems and the regularization
  method. {\it Soviet Math. Dokl.} 1963\string; 4\string: 1035--1038.

\bibitem{Beck1968}
Beck JV. Surface heat flux determination using an integral method. {\it Nuclear
  Engineering and Design} 1968\string; 7(2)\string: 170 - 178.

\bibitem{Beck1970}
Beck JV. Nonlinear estimation applied to the nonlinear inverse heat conduction
  problem. {\it International Journal of Heat and Mass Transfer} 1970\string;
  13(4)\string: 703-716.

\bibitem{Chen1976}
Chen CJ, Chiou JS. Prediction of surface temperature and heat flux from an
  interior temperature response. {\it Letters in Heat and Mass Transfer}
  1976\string; 3(6)\string: 539 - 548.

\bibitem{Pinheiro2000}
Pinheiro C, Samarasekera I, Brimacomb J, Walker B. Mould heat transfer and
  continuously cast billet quality with mould flux lubrication Part 1 Mould
  heat transfer. {\it Ironmaking \& Steelmaking} 2000\string; 27(1)\string:
  37-54.
\newblock \href {\doibase 10.1179/030192300677363} {doi:
  10.1179/030192300677363}

\bibitem{Samarasekera1982}
Samarasekera I, Brimacombe J. The influence of mold behavior on the production
  of continuously cast steel billets. {\it Metallurgical Transactions B}
  1982\string; 13(1)\string: 105--116.

\bibitem{Rauter2008}
Michelic S, Rauter W, Erker M, Brandl W, Bernhard C. Heat Transfer in a Round
  CC Mould: Measurement, Modelling and Validation. In: Steelmaking \& Plastic
  Deformation Study Groups. Associazione Italiana di Metallurgia; 2008;
  Italy\string: Paper--30.

\bibitem{Ranut2011}
Ranut P, Persi C, Nobile E, Spagnul S. {Estimation of Heat Flux Distribution in
  a Continuous Casting Mould by Inverse Heat Transfer Algorithms}. In: . Volume
  2: 31st Computers and Information in Engineering Conference, Parts A and B.
  The American Society of Mechanical Engineers. ; 2011\string: 389-398

\bibitem{Nelder1965}
Nelder JA, Mead R. A simplex method for function minimization. {\it The
  computer journal} 1965\string; 7(4)\string: 308--313.

\bibitem{Man2004}
Man Y, Hebi Y, Dacheng F. Real-time Analysis on Non-uniform Heat Transfer and
  Solidification in Mould of Continuous Casting Round Billets. {\it Isij
  International - ISIJ INT} 2004\string; 44\string: 1696-1704.
\newblock \href {\doibase 10.2355/isijinternational.44.1696} {doi:
  10.2355/isijinternational.44.1696}

\bibitem{Hebi2006}
Hebi Y, Man Y, Dacheng F. 3-D Inverse Problem Continuous Model for Thermal
  Behavior of Mould Process Based on the Temperature Measurements in Plant
  Trial. {\it Isij International - ISIJ INT} 2006\string; 46\string: 539-545.
\newblock \href {\doibase 10.2355/isijinternational.46.539} {doi:
  10.2355/isijinternational.46.539}

\bibitem{Gonzalez2003}
Gonzalez M, Goldschmit M, Assanelli A, Dvorkin E, Berdaguer E. Modeling of the
  solidification process in a continuous casting installation for steel slabs.
  {\it Metallurgical and Materials Transactions B} 2003\string; 34\string:
  455-473.
\newblock \href {\doibase 10.1007/s11663-003-0072-3} {doi:
  10.1007/s11663-003-0072-3}

\bibitem{Wang2016}
Wang X, Kong L, Du F, et al. Mathematical Modeling of Thermal Resistances of
  Mold Flux and Air Gap in Continuous Casting Mold Based on an Inverse Problem.
  {\it ISIJ International} 2016\string; 56\string: 803-811.
\newblock \href {\doibase 10.2355/isijinternational.ISIJINT-2015-601} {doi:
  10.2355/isijinternational.ISIJINT-2015-601}

\bibitem{ZhangWang2017}
Zhang H, Wang W. Mold Simulator Study of Heat Transfer Phenomenon During the
  Initial Solidification in Continuous Casting Mold. {\it Metallurgical and
  Materials Transactions B} 2017\string; 48\string: 779-793.
\newblock \href {\doibase 10.1007/s11663-016-0901-9} {doi:
  10.1007/s11663-016-0901-9}

\bibitem{Hu2018}
Hu P, Wang X, Wei J, Yao M, Guo Q. Investigation of Liquid/Solid Slag and Air
  Gap Behavior inside the Mold during Continuous Slab Casting. {\it ISIJ
  International} 2018\string; 58(5)\string: 892-898.
\newblock \href {\doibase 10.2355/isijinternational.ISIJINT-2017-393} {doi:
  10.2355/isijinternational.ISIJINT-2017-393}

\bibitem{Tang2012}
Tang L, Yao M, Wang X, Zhang X. Non-uniform thermal behavior and shell growth
  within mould for wide and thick slab continuous casting. {\it Steel Research
  International} 2012\string; 83(12)\string: 1203-1213.
\newblock \href {\doibase 10.1002/srin.201200075} {doi: 10.1002/srin.201200075}

\bibitem{Videcoq2008}
Videcoq E, Quemener O, Lazard M, Neveu A. Heat source identification and
  on-line temperature control by a Branch Eigenmodes Reduced Model. {\it
  International Journal of Heat and Mass Transfer} 2008\string; 51(19)\string:
  4743-4752.
\newblock \href {\doibase
  https://doi.org/10.1016/j.ijheatmasstransfer.2008.02.029} {doi:
  https://doi.org/10.1016/j.ijheatmasstransfer.2008.02.029}

\bibitem{Videcoq2006}
Videcoq E, Neveu A, Quemener O, Girault M, Petit D. Comparison of Two Nonlinear
  Model Reduction Techniques: The Modal Identification Method and the Branch
  Eigenmodes Reduction Method. {\it Numerical Heat Transfer, Part B:
  Fundamentals} 2006\string; 49(6)\string: 537-558.
\newblock \href {\doibase 10.1080/10407790500344035} {doi:
  10.1080/10407790500344035}

\bibitem{Girault2010}
Girault M, Videcoq E, Petit D. Estimation of time-varying heat sources through
  inversion of a low order model built with the Modal Identification Method
  from in-situ temperature measurements. {\it International Journal of Heat and
  Mass Transfer} 2010\string; 53(1)\string: 206-219.
\newblock \href {\doibase
  https://doi.org/10.1016/j.ijheatmasstransfer.2009.09.040} {doi:
  https://doi.org/10.1016/j.ijheatmasstransfer.2009.09.040}

\bibitem{Girault2005}
Girault M, Petit D. Identification methods in nonlinear heat conduction. Part
  II: inverse problem using a reduced model. {\it International Journal of Heat
  and Mass Transfer} 2005\string; 48(1)\string: 119-133.
\newblock \href {\doibase
  https://doi.org/10.1016/j.ijheatmasstransfer.2004.06.033} {doi:
  https://doi.org/10.1016/j.ijheatmasstransfer.2004.06.033}

\bibitem{Aguado2015}
Aguado JV, Huerta A, Chinesta F, Cueto E. Real-time monitoring of thermal
  processes by reduced-order modeling. {\it International Journal for Numerical
  Methods in Engineering} 2015\string; 102(5)\string: 991-1017.
\newblock \href {\doibase https://doi.org/10.1002/nme.4784} {doi:
  https://doi.org/10.1002/nme.4784}

\bibitem{WangYao2011}
Wang X, Yao M. Neural networks for solving the inverse heat transfer problem of
  continuous casting mould. In: . 2. {IEEE} Circuits and Systems Society. ;
  2011\string: 791-794.

\bibitem{Chen2014}
Chen H, Su L, Wang G, Shibin W, Zhang L, Luo Z. Fuzzy estimation for heat flux
  distribution at the slab continuous casting mold surface. {\it International
  Journal of Thermal Sciences} 2014\string; 83\string: 80-88.
\newblock \href {\doibase 10.1016/j.ijthermalsci.2014.04.012} {doi:
  10.1016/j.ijthermalsci.2014.04.012}

\bibitem{Buhmann2003}
Buhmann MD. {\it Radial basis functions: theory and implementations}. 12.
\newblock Cambridge university press .
\newblock 2003.

\bibitem{Prando2016}
Prando G. {\it Non-Parametric Bayesian Methods for Linear System
  Identification}. PhD thesis. Universit\'a di Padova, Via 8 Febbraio 1848, 2,
  35122 Padova PD, Italy;  2016.

\bibitem{Engl1996}
Engl H. {\it Regularization of inverse problems}.
\newblock Dordrecht Boston: Kluwer Academic Publishers .
\newblock 1996.

\bibitem{Hansen2010}
Hansen P. {\it Discrete inverse problems : insight and algorithms}.
\newblock Philadelphia: Society for Industrial and Applied Mathematics .
\newblock 2010.

\bibitem{Aster2019}
Aster R. {\it Parameter estimation and inverse problems}.
\newblock Amsterdam, Netherlands: Elsevier .
\newblock 2019.

\bibitem{Kirsch2011}
Kirsch A. {\it Regularization by Discretization}\string: 63--119; New York, NY:
  Springer New York .
\newblock 2011

\bibitem{Stabile2017}
Stabile G, Hijazi S, Mola A, Lorenzi S, Rozza G. {POD-Galerkin reduced order
  methods for CFD using Finite Volume Discretisation: vortex shedding around a
  circular cylinder}. {\it Communications in Applied and Industrial
  Mathematics} (2017)\string; 8(1)\string: 210-236.
\newblock \href {\doibase 10.1515/caim-2017-0011} {doi: 10.1515/caim-2017-0011}

\bibitem{ithaca}
{ITHACA-FV} . \url{https://mathlab.sissa.it/ithaca-fv}; .
\newblock Accessed: 2020-10-26.

\bibitem{Moukalled2015}
Moukalled F, Mangani L, Darwish M. {\it The Finite Volume Method in
  Computational Fluid Dynamics: An Advanced Introduction with OpenFOAM and
  Matlab}.
\newblock Springer Publishing Company, Incorporated.
\newblock 1st~ed. 2015.

\bibitem{Donald1975}
Olsson DM, Nelson LS. The Nelder-Mead Simplex Procedure for Function
  Minimization. {\it Technometrics} 1975\string; 17(1)\string: 45-51.
\newblock \href {\doibase 10.1080/00401706.1975.10489269} {doi:
  10.1080/00401706.1975.10489269}

\bibitem{Strazzullo2022}
Strazzullo M, Ballarin F, Rozza G. {POD}-{G}alerkin model order reduction for
  parametrized nonlinear time-dependent optimal flow control: an application to
  shallow water equations. {\it Journal of Numerical Mathematics} 2022\string;
  30(1)\string: 63--84.
\newblock \href {\doibase doi:10.1515/jnma-2020-0098} {doi:
  doi:10.1515/jnma-2020-0098}

\bibitem{Negri2013}
Negri F, Rozza G, Manzoni A, Quarteroni A. Reduced Basis Method for
  Parametrized Elliptic Optimal Control Problems. {\it SIAM Journal on
  Scientific Computing} 2013\string; 35(5)\string: A2316-A2340.
\newblock \href {\doibase 10.1137/120894737} {doi: 10.1137/120894737}

\bibitem{Stabile2020b}
Stabile G, Rosic B. Bayesian identification of a projection-based reduced order
  model for computational fluid dynamics. {\it Computers \& {F}luids}
  2020\string; 201\string: 104477.
\newblock \href {\doibase 10.1016/j.compfluid.2020.104477} {doi:
  10.1016/j.compfluid.2020.104477}

\end{thebibliography}

\end{document}